\documentclass[10pt]{amsart}
\usepackage{amssymb,amsmath,amsfonts,amsthm,latexsym,ifthen,ushort,stmaryrd,xspace,
enumitem}
\usepackage[english]{babel}
\usepackage{setspace}
\usepackage{bbm}

\DeclareMathAlphabet{\mathpzc}{OT1}{pzc}{m}{it}
\usepackage[latin1]{inputenc}
\usepackage{afterpage}
\newcommand{\kla}[1]{ {\langle #1 \rangle} }
\newcommand{\st}{\;|\;}

\newcommand{\ran}{ {\rm ran} }

\newcommand{\crit}{ {\rm crit} }
\newcommand{\sub}{\subseteq}
\newcommand{\eins}{ {1{\rm\hspace{-0.5ex}l}} }
\newcommand{\rest}{{\restriction}}
\newcommand{\cf}{ {\rm cf} }
\newcommand{\On}{ {\rm On} }

\newcommand{\verl}{{{}^\frown}}%{{\frown\atop }}
\newcommand{\leer}{\emptyset}
\newcommand{\ohne}{\setminus}

\newenvironment{ea*}{\begin{eqnarray*}}{\end{eqnarray*}}
\newcommand{\claim}[2]{
     \begin{enumerate}
       \item[{#1}] {\em #2}
     \end{enumerate}}
\newcommand{\To}{\longrightarrow}

\newcommand{\lh}{{\rm lh}}

\newcommand{\ba}{{\bar{a}}}
\newcommand{\barb}{{\bar{b}}}

\newcommand{\bs}{{\bar{s}}}
\newcommand{\bu}{{\bar{u}}}
\newcommand{\bA}{{\bar{A}}}
\newcommand{\bB}{{\bar{B}}}

\newcommand{\bG}{{\bar{G}}}
\newcommand{\bH}{{\bar{H}}}
\newcommand{\bI}{{\bar{I}}}

\newcommand{\bP}{{\bar{\P}}}

\newcommand{\bS}{{\bar{S}}}
\newcommand{\bT}{{\bar{T}}}

\newcommand{\bX}{{\bar{X}}}

\newcommand{\bd}{{\bar{d}}}
\newcommand{\bj}{{\bar{j}}}
\newcommand{\bk}{{\bar{k}}}

\newcommand{\balpha}{{\bar{\alpha}}}
\newcommand{\bbeta}{{\bar{\beta}}}
\newcommand{\bgamma}{{\bar{\gamma}}}
\newcommand{\bdelta}{{\bar{\delta}}}

\newcommand{\bnu}{{\bar{\nu}}}

\newcommand{\blambda}{{\bar{\lambda}}}

\newcommand{\btheta}{{\bar{\theta}}}
\newcommand{\bpi}{{\bar{\pi}}}

\newcommand{\bsigma}{{\bar{\sigma}}}

\newcommand{\bard}{{\bar{d}}}

\newcommand{\barf}{{\bar{f}}}

\newcommand{\bp}{{\bar{p}}}

\newcommand{\bN}{{\bar{N}}}

\newcommand{\tdelta}{{\tilde{\delta}}}

\newcommand{\tsigma}{{\tilde{\sigma}}}

\newcommand{\va}{{\vec{a}}}
\newcommand{\vb}{{\vec{b}}}
\newcommand{\vc}{{\vec{c}}}

\newcommand{\vp}{{\vec{p}}}

\newcommand{\vBB}{{\vec{\B}}}

\newcommand{\vr}{{\vec{r}}}

\newcommand{\seq}[2]{{\langle#1\;|\;}\linebreak[0]{#2\rangle}}
\renewcommand{\phi}{\varphi}
\newcommand{\card}[1]{\overline{\overline{#1}}}
\newcommand{\ZFC}{\ensuremath{\mathsf{ZFC}}}
\newcommand{\ZFCm}{\ensuremath{{\ZFC}^-}}
\newcommand{\V}{\ensuremath{\mathrm{V}}}
\newcommand{\forces}{\Vdash}
\newcommand{\B}{{\mathord{\mathbb{B}}}}
\renewcommand{\P}{{\mathord{\mathbb P}}}
\newcommand{\Q}{{\mathord{\mathbb Q}}}

\newcommand{\MA}{\ensuremath{\mathsf{MA}}}
\newcommand{\MM}{\ensuremath{\mathsf{MM}}}

\newcommand{\SCFA}{\ensuremath{\mathsf{SCFA}}\xspace}

\newcommand{\CH}{\ensuremath{\mathsf{CH}}\xspace}

\newcommand{\isomorphism}{\stackrel{\sim}{\longleftrightarrow}}

\newcommand{\BSCFA}{\ensuremath{\mathsf{BSCFA}}\xspace}

\newcommand{\Prikry}{P\v{r}\'{\i}kr\'{y}}

\newtheorem{thm}{Theorem}[section]
\newtheorem{cor}[thm]{Corollary}
\newtheorem{lem}[thm]{Lemma}
\newtheorem{obs}[thm]{Observation}

\newtheorem{fact}[thm]{Fact}

\newtheorem{proposition}[thm]{Proposition}
\theoremstyle{definition}
\newtheorem{defn}[thm]{Definition}
\newtheorem{question}[thm]{Question}

\theoremstyle{remark}
\newtheorem{remark}[thm]{Remark}

\newcommand{\bbA}{\mathbb{A}}
\newcommand{\bbB}{\bar{\mathbb{B}}}

\newcommand{\BV}[1]{\llbracket#1\rrbracket}

\newcommand{\eps}{\varepsilon}

\renewcommand{\card}[1]{\mathrm{card}(#1)}

\newcommand{\oo}{{}^\omega\omega}
\newcommand{\interpretation}{\mathsf{intp}}

\def\hook{\upharpoonright}

\def\barG{\overline{G}}

\newcommand{\hooks}{\mathrel{\angle}}
\newcommand{\suc}{\mathop{\mathrm{suc}}}
\newcommand{\Root}{\mathrm{root}}

\begin{document}
\title{Iteration theorems for subversions of forcing classes}
\author[Fuchs]{Gunter Fuchs}
\address[G.~Fuchs]{Mathematics,
          The Graduate Center of The City University of New York,
          365 Fifth Avenue, New York, NY 10016
          \&
          Mathematics,
          College of Staten Island of CUNY,
          Staten Island, NY 10314}
\email{Gunter.Fuchs@csi.cuny.edu}
\urladdr{http://www.math.csi.cuny/edu/\textasciitilde{}fuchs}
\thanks{The first author's research for this paper was supported in part by PSC CUNY research grant 61567-00 49.}

\author[Switzer]{Corey Bacal Switzer}
\address[C.B.~Switzer]{Mathematics, The Graduate Center of The City University of New York, 365 Fifth Avenue, New York, NY 10016}
\email{cswitzer@gradcenter.cuny.edu}

\date{\today}

\begin{abstract}
We prove various iteration theorems for forcing classes related to subproper and subcomplete forcing, introduced by Jensen. In the first part, we use revised countable support iterations, and show that
1) the class of subproper, $\oo$-bounding forcing notions, 2) the class of subproper, $T$-preserving forcing notions (where $T$ is a fixed Souslin tree) and 3) the class of subproper, $[T]$-preserving forcing notions (where $T$ is an $\omega_1$-tree) are iterable with revised countable support. In the second part, we adopt Miyamoto's theory of nice iterations, rather than revised countable support. We show that this approach allows us to drop a technical condition in the definitions of subcompleteness and subproperness, still resulting in forcing classes that are iterable in this way, preserve $\omega_1$, and, in the case of subcompleteness, don't add reals. Further, we show that the analogs of the iteration theorems proved in the first part for RCS iterations hold for nice iterations as well.
\end{abstract}

\keywords{Iterated forcing, revised countable support, subcomplete forcing}

\maketitle

\section{Introduction}

This article brings together two important threads in forcing iteration theory: variations on revised countable support and Jensen's subversions of forcing classes. In the first half we pursue a Boolean algebraic approach to iterating subproper and subcomplete forcing notions while preserving certain properties of the forcing notions being iterated. This is based on what has been done in \cite{Jensen:IterationTheorems}. In the second half we contrast this approach with one involving nice iterations in the sense of Miyamoto \cite{Miyamoto:IteratingSemiproperPreorders}. Here, we use partial preorders and obtain the same iteration and preservation theorems. In this setting, there is a further upshot that we can remove one of the more technical conditions in the definition of subproperness, thus getting an iteration theorem for a (seemingly) more general class of forcing notions. We can similarly omit that condition from the definition of subcomplete forcing notions while maintaining the crucial properties of the forcing class, namely not adding reals, preserving Souslin trees and diamond sequences, and being iterable via nice iterations.

In the final section we provide applications of the theorems from the first two sections, constructing multiple new models of the axiom $\SCFA$.

The definitions of the classes of subproper and subcomplete forcing (defined in the next section) result from modifying the definitions of proper and $\sigma$-closed forcing in a way we call ``subversion'', resulting in properly larger forcing classes which include forcing notions that badly fail to be proper, such as Namba forcing (under $\CH$), {\Prikry} forcing and the forcing to shoot clubs through stationary subsets of $\kappa\cap\mathrm{cof}(\omega)$ for regular $\kappa > 2^{\aleph_0}$, see \cite{Jensen2014:SubcompleteAndLForcingSingapore}.

Both classes come with associated forcing axioms and the subcomplete forcing axiom, $\SCFA$, which is Martin's axiom on $\aleph_1$ for subcomplete forcing is particularly striking as it has much of the strength of $\MM$, for instance implying failure of various square principles, see Jensen \cite{Jensen:FAandCH} and Fuchs \cite{Fuchs:HierarchiesOfForcingAxioms}, \cite{Fuchs:DiagonalReflection}, but is consistent with $\diamondsuit$. Models of $\SCFA +\neg \CH$ came up during the investigation of the first author and Minden in \cite{FuchsMinden:SCforcingTreesGenAbs}, and the present paper grew out of discussions relating to this.

The basic issue is as follows. In trying to uncover consequences of $\SCFA$ one often falls into two situations. Either some consequence of $\MM$ can be shown to follow already from $\SCFA$, usually by simply showing that the forcing used for the $\MM$ application is actually subcomplete, or else this consequence is incompatible with $\diamondsuit$, and hence it cannot follow from $\SCFA$. An example for the first type of situation, Jensen's aforementioned observation that the forcing to shoot a club in ordertype $\omega_1$ through a stationary subset of $\kappa \cap\mathrm{cof}(\omega)$, for regular $\kappa> 2^\omega$, is subcomplete, allows one to conclude that $\mathsf{SCH}$ follows from $\SCFA$, \cite[Corollary 7.4]{Jensen2014:SubcompleteAndLForcingSingapore}, using the known arguments based on Martin's Maximum. As an example for the second type of situation, $\SCFA$ does not imply Souslin's Hypothesis, since Souslin trees exist in any model with a diamond sequence. Another example is that \SCFA does not imply that the nonstationary ideal on $\omega_1$ is saturated (as Martin's Maximum does), again since \SCFA is compatible with $\diamondsuit$, which, in turn, implies the failure of saturation.

%We call the latter type of non-implication ``the $\diamondsuit$-obstruction".
The basic question is whether such non-implications are only obstructed by diamond (or $\CH$), i.e. is it enough to add the failure of $\CH$, say, to $\SCFA$ to resurrect the consequences of $\MM$ that are not already consequences of $\SCFA$? In short, does $\SCFA+\neg\CH$ imply $\MM$?

Of course, we did not expect the answer to this question to be affirmative, but the fact that this was a question illustrates how little we knew about models of $\SCFA+\neg\CH$.

In the final section we will apply our iteration theorems to produce various models of $\SCFA+\neg\CH$ with different constellations of cardinal characteristics of the continuum inconsistent with $\MM$, thus answering the provocative question above in the negative, as expected. This is also interesting on its own right, since it shows that there are strong forcing axioms compatible with various constellations of cardinal characteristics of the continuum often studied in set theory of the reals. A sampling of results along this line is given below.

\begin{thm}
Assuming the consistency of a supercompact cardinal the following are consistent with $\SCFA + \neg \CH$.
\begin{enumerate}
\item
Souslin's Hypothesis fails.
\item
$\mathfrak{d} < \mathfrak{c}$.
\item
$\MA_{\aleph_1} (\sigma {\rm -linked})$ holds while $\MA_{\aleph_1}$ fails.
\end{enumerate}
\end{thm}

The key to proving these results is the proof of iteration theorems for these classes. Specifically we show that certain iterations of subproper forcing notions preserving a fixed Souslin tree $S$ preserve $S$, that certain iterations of $\oo$-bounding subproper forcing notions are $\oo$-bounding and that certain iterations of subproper forcing notions not adding branches through a fixed $\omega_1$-tree $T$ do not add branches through $T$. We can then, starting in a model with a supercompact cardinal, run the usual argument, based on Baumgartner's construction of a model of the proper forcing axiom, to produce a model of the forcing axiom for the relevant forcing class. For example, if we do this for the class of subproper forcing notions that preserve a particular Souslin tree $S$, then in the resulting model, there obviously is a Souslin tree, and $\CH$ fails, because Cohen forcing is in the class, and \SCFA holds, since every subcomplete forcing preserves Souslin trees. So this will be a model of $\SCFA+\neg\CH$ in which Souslin's Hypothesis fails.

We prove these iteration theorems in two different ways. First using RCS iterations, generalizing Jensen's techniques, and second with nice iterations in the sense of Miyamoto, \cite{Miyamoto:IteratingSemiproperPreorders}, who carried out these arguments in the context of semiproper forcing. In the latter case we drop a technical condition on the definition of subproper and subcomplete forcing. We dub these classes $\infty$-subproper and $\infty$-subcomplete forcing notions. We feel both proofs of the iteration theorems give information and perspective the other does not. It also sheds light on the difference between different styles of RCS iterations in this novel context. Since the precise relationship between RCS iterations and nice iterations is still not completely understood this may be of independent interest.

As such this article is broken into two parts. In Section \ref{sec:RCSiterations}, we treat RCS iterations and develop more fully the theory of nicely subproper iterations. This includes the abovementioned iteration theorems. In Section \ref{sec:NiceIterations} we reconsider these theorems, this time using nice iterations. We introduce $\infty$-subproper and $\infty$-subcomplete forcing notions, study their general properties, and we review the machinery of nice iterations needed to prove the iteration theorems in this context. Note that while the two approaches use the same generic word ``nice", it  means something very different in Sections \ref{sec:RCSiterations} and \ref{sec:NiceIterations}. It is simply an unfortunate coincidence that the established terminology in the literature conflicts in this way.

In Section \ref{sec:Applications}, we give the aforementioned applications to the study of forcing axioms, and we conclude with some remarks and open questions in Section \ref{sec:Questions}.

We would like to express our sincerest gratitude to the referee for reading an earlier manuscript of this article very carefully and for bringing many errors to our attention. Their work helped improve this article vastly.

\section{RCS iterations}
\label{sec:RCSiterations}

In this section, we will prove preservation theorems for iterations of subproper forcing notions with revised countable support, and variations thereof. We use a definition of subproperness that uses a slightly different Hull Property, following Jensen \cite[\S 4]{Jensen:IterationTheorems}. It also incorporates a variation mentioned in \cite[\S 4]{Jensen:IterationTheorems} and is somewhat close to what Jensen would call ``very subproper''. Namely, in place of the cardinality of a poset, we use its density, defined as follows.

\begin{defn}
\label{def:DensityOfAPoset}
Given a poset $\P$, $\delta(\P)$ is the smallest cardinal $\kappa$ such that there is a dense subset of $\P$ that has cardinality $\kappa$.
\end{defn}

There are other, maybe more natural, measures of the size of a poset, introduced in \cite{Fuchs:ParametricSubcompleteness}. We could use those as well, and work with the resulting variations of subproperness, but since we don't have any applications of these variations thus far, we chose not to do so. The density of a partial order is related to its chain condition:

\begin{obs}
\label{obs:P_is_delta(P)+cc}
For any poset $\P$, $\P$ is $\delta(\P)^+$-c.c.
\end{obs}

\begin{proof}
Let $A\sub\P$ be an antichain, and let $D\sub\P$ be a dense set of cardinality $\delta(\P)$. Define $f:A\To D$ by choosing, for each $a\in A$, an $f(a)\le a$ with $f(a)\in D$. Then $f$ is injective, and hence, $\card{A}\le\card{D}<\delta(\P)^+$.
\end{proof}

The following definition is due to Jensen. The theory \ZFCm results from \ZFC\ by dropping the Power Set Axiom and replacing the Replacement Scheme with the Collection Scheme, consisting of all sentences of the form  $\forall\vec{z}\forall u\exists v\forall x\in u\ (\exists y\  \phi(x,y,\vec{z}) \To \exists y\in v\ \phi(x,y,\vec{z}))$, for any formula $\phi(x,y,\vec{z})$ (with the free variables listed) in the language of set theory, where $u,v$ are variable that don't occur in $\phi$.

\begin{defn}
\label{def:fullness}
A transitive model $N$ of $\ZFCm$ is \emph{full} if there is an ordinal $\gamma>0$ such that $L_\gamma(N)$ satisfies $\ZFCm$ and $N$ is regular in $L_\gamma(N)$, meaning that if $a\in N$ and $f\in L_\gamma(N)$ is a function $f:a\To N$, then $\ran(f)\in N$.	
\end{defn}

We are now ready to state Jensen's definition of subproperness. Here, when $\B$ is a complete Boolean algebra, we use the notation $\B^+$ for $\B\ohne\{0\}$. Also, here and in the following, we denote by $L_\tau^A$ the structure $\kla{L_\tau[A],\in\rest L_\tau[A],A\cap L_\tau[A]}$.

\begin{defn}
\label{def:SubpropernessWithSupremumCondition}
A complete Boolean algebra $\B$ is \emph{subproper} if every sufficiently large cardinal $\theta$ verifies the subproperness of $\B$, meaning that the following holds: $\B\in H_\theta$, and if $\tau>\theta$ is such that $H_\theta\sub N=L_\tau^A\models\ZFC^-$, and $\sigma:\bN\prec N$, where $\bN$ is countable, transitive and full, and $\bS=\kla{\btheta,\bbB,\ba,\bs,\blambda_1,\ldots,\blambda_n}\in\bN$,
$S=\kla{\theta,\B,a,s,\lambda_1,\ldots,\lambda_n}=\sigma(\bS)$, where $a\in\B^+$ and $\lambda_i>\delta(\B)$ is regular, for $1\le i\le n$, then there is a $c\in\B^+$ such that $c\le a$ and such that whenever $G\sub\B$ is generic with $c\in G$, then there is a $\sigma'\in\V[G]$ such that
\begin{enumerate}[label=(\alph*)]
  \item $\sigma':\bN\prec N$.
  \item $\sigma'(\bS)=\sigma(\bS)$.
  \item $(\sigma')^{-1}``G$ is $\bbB$-generic over $\bN$.
    \item
  \label{item:SupremaCondition}
  Letting $\blambda_0=\On\cap\bN$, for all $i\le n$, we have that $\sup(\sigma')``\blambda_i=\sup\sigma``\blambda_i$.
\end{enumerate}
\end{defn}

Our definition differs slightly from that used in \cite{Jensen:IterationTheorems}, in that we don't require $\tau$ (in the notation of the definition) to be regular. This ensures that the resulting definition is locally based, in the sense of \cite[\S2, p.~6]{Jensen:FAandCH}, and is in line with the definition of subcompleteness employed by Jensen in \cite[p.~3]{Jensen:FAandCH}. There are several ways of defining subproperness directly for posets rather than complete Boolean algebras, but since we are going to work with complete Boolean algebras here, the definition given will do. We will return to the poset definition in Section \ref{sec:NiceIterations}. Jensen has sometimes employed a slight strengthening of condition \ref{item:SupremaCondition} above. We will also return to this later.

\subsection{The subproperness extension lemma}
\label{subsec:SubpropernessExtensionLemma}

The iteration theorems in this section are based on one main lemma, which we will prove here. It is based on what Jensen calls the One Step Lemma, but slightly more abstract. We call it the Subproperness Extension Lemma, in analogy to the context of proper forcing. In order to formulate it, we need some terminology regarding iterated forcing using the Boolean algebraic approach.

If $\B$ is a complete Boolean algebra, then we write $\bbA\sub\B$ to express that $\bbA$ is a complete subalgebra of $\B$, meaning that $\bbA$ is a complete Boolean algebra, and furthermore, that for $X\sub\bbA$, $\bigvee^\bbA X=\bigvee^\B X$ and $\bigwedge^\bbA X=\bigwedge^\B X$.
In this situation, the \emph{retraction} $h_{\B,\bbA}:\B\To\bbA$ is defined by
\[h_{\B,\bbA}(b)=\bigwedge\{a\in\bbA\st b\le_\B a\}.\]
Further, if $G_0\sub\bbA$ is a generic filter, then $G_0$ generates the filter $G=\{b\in\B\st\exists a\in G_0\ a\le b\}$ on $\B$, and writing $b\Rightarrow c$ for $\neg b\lor c$, $G$ induces an equivalence relation on $\B$ defined by identifying $b$ and $b'$ iff $(b\Rightarrow b')\land(b'\Rightarrow b)\in G$. We write $b/G_0$ for the equivalence class of $b$ under that equivalence relation, and we write $\B/G_0$ for the factor algebra. The ordering on $\B/G_0$ is given by $b/G_0\le_{\B/G_0}c/G_0$ iff $(b\Rightarrow c)\in G$. We write $\dot{G}_{\bbA}$ for the canonical name for the $\bbA$-generic ultrafilter. Here and in the following, when stating lemmas and theorems with multiple assumptions and conclusions, we will try to stick to the convention of numbering the assumptions $(A1)-(Am)$ and the conclusions $(C1)-(Cn)$.

\begin{lem}[Subproperness Extension Lemma]
\label{lem:SubpropernessExtensionLemma}
Let $\B$ be a complete Boolean algebra, and let $\bbA\sub\B$ be a complete subalgebra of $\B$. Let $h=h_\bbA:\B\To\bbA$ be the retraction. Let $\delta=\delta(\B)$. Suppose that $\forces_\bbA$`` $\check{\B}/\dot{G}_\bbA$ is subproper, as verified by $\check{\theta}$,'' where $\B\in H_\theta$. Let $N=L_\tau^A$ be a \ZFCm{} model with $H_\theta\sub N$ and $\theta<\tau$, and let $\bN$ be countable, transitive and full. Let $\bS=\kla{\btheta,\bar{\bbA},\bar{\B},\bs,\blambda_1,\ldots,\blambda_n}\in\bN$, and $S=\kla{\theta,\bbA,\B,s,\lambda_1,\ldots,\lambda_n}\in N$, where $\delta<\lambda_i$ is regular, for $1\le i\le n$. Let $\dot{\sigma}_0$, $\dot{t}$, $\dot{\bar{b}}$ and $\dot{b}$ be $\bbA$-names, and let $a\in\bbA^+$ be a condition that forces with respect to $\bbA$:
\begin{enumerate}[label=\textnormal{(A\arabic*)}]
\item $\dot{\sigma}_0:\check{\bN}\prec\check{N}$,
\item $\dot{\sigma}_0(\check{\bS})=\check{S}$ (so $\bar{\bbA}$ and $\bar{\B}$ are complete Boolean algebras in $\bN$),
\item $\dot{t}\in\check{\bN}$, $\dot{\barb}\in\check{\bar{\B}}$, $\dot{b}\in\check{\B}$, and $\check{h}(\dot{b})\in\dot{G}_{\bbA}$,
\item $\dot{\sigma}_0^{-1}``\dot{G}_\bbA$ is $\check{\bN}$-generic for $\bar{\bbA}$,
\item $\dot{\sigma}_0(\dot{\bar{b}})=\dot{b}$.
\end{enumerate}
Then there are a condition $c\in\B^+$ such that $h(c)=a$ and a $\B$-name $\dot{\sigma}_1$ such that whenever $I$ is $\B$-generic with $c\in I$, letting $\sigma_1=\dot{\sigma}_1^I$, $G=I\cap\bbA$ and $\sigma_0=\dot{\sigma}_0^G$, the following conditions hold:
\begin{enumerate}[label=\textnormal{(C\arabic*)}]
\item $\sigma_1:\bN\prec N$,
\item $\sigma_1(\bS)=S$,
\item $\sigma_1(\dot{t}^G)=\sigma_0(\dot{t}^G)$,
\item $\sigma_1(\dot{\bar{b}}^G)=\sigma_0(\dot{\bar{b}}^G)=\dot{b}^G$,
\item $\dot{b}^G\in I$,
\item ${\sigma_1}^{-1}``I$ is $\bar{\B}$-generic over $\bar{N}$,
\item Letting $\blambda_0=\On\cap\bN$, for all $i\le n$, $\sup\sigma_0``\blambda_i=\sup\sigma_1``\blambda_i$.
%$C^N_\delta(\ran(\sigma_1))=C^N_\delta(\ran(\sigma_0))$.
\end{enumerate}
\end{lem}

\begin{proof}
We follow the proof of \cite[\S 2, Lemma 2]{Jensen:SPSCF}.
Let $G$ be any $\bbA$-generic filter with $a\in G$. Then in $\V[G]$, $\B/G$ is subproper, as verified by $\theta$. Let $t=\dot{t}^{G}$, $\bar{b}=\dot{\bar{b}}^G$, $b=\dot{b}^G$ and $\sigma_0=\dot{\sigma}_0^{G}$. Then $\sigma_0:\bN\prec N$, $\sigma_0(\bS)=S$, $h(b)\in G$, $t\in\bN$ and $\bG=\sigma_0^{-1}``G$ is $\bN$-generic for $\bbA$.
Let \[\sigma_0^*:\bN[\bG]\prec N[G]\]
be the unique embedding extending $\sigma_0$ such that $\sigma_0^*(\bG)=G$.

We have that $H_\theta^{V[G]}=H_\theta[G]\sub N[G]$, and $\bN[\bG]$ is full in $\V[G]$. We also have that $b/G\neq 0$. To see this, note that since $b\in\V$, it makes sense to write $b=\check{b}^G$. By \cite[p.~91, Fact 3]{Jensen:SPSCF}, we know that $h(b)=\BV{\check{b}/\dot{G}_\bbA\neq 0}$. So, since $h(b)\in G$, this means that $\BV{\check{b}/\dot{G}_\bbA\neq 0}\in G$, which means precisely that $b/G\neq 0$.

So, since $\B/G$ is subproper in $\V[G]$, there is a condition $d\in\B/G$ with $d\le b/G$ such that whenever $H$ is generic for $\B/G$ over $\V[G]$ with $d\in H$, then in $\V[G][H]$, there is an elementary embedding $\sigma':\bN[\bG]\prec N[G]$ with $\sigma'(\bS)=S$, $\bH=(\sigma')^{-1}``H$ is $\bar{\B}/\bG$-generic over $\bN[\bG]$ and for all $i\le n$, $\sup\sigma^*_0``\blambda_i=\sup\sigma'``\blambda_i$. We may moreover insist that $\sigma'$ maps any finite list of members of $\bN[\bG]$ the same way $\sigma_0^*$ does. Thus, we require that $\sigma'(t)=\sigma^*_0(t)$, $\sigma'(\bar{b})=\sigma_0^*(\bar{b})=b$ and $\sigma'(\bG)=\sigma_0^*(\bG)=G$.

Let us temporarily fix such an $H$, and let $I=G*H$. Let $\sigma_1=\sigma'\rest\bN$. It follows that $\sigma'(\tau^\bG)=(\sigma_1(\tau))^G$, for $\tau\in\bN^{\bar{\bbA}}$, since $\sigma'(\bG)=G$.

Then $\sigma'$, $G$ and $I$ clearly satisfy conditions (1), (2), (3), (4), (5) and (7) above.

It follows also that (6) is satisfied, that is, $\bI=\sigma_1^{-1}``I$ is $\bar{\B}$-generic over $\bN$: since $\bG$ is $\bar{\bbA}$-generic over $\bN$ and $\bH$ is $\bar{\bbB}/\bG$-generic over $\bN[\bG]$, it follows that $\bar{G}*\bar{H}$ is $\bar{\B}$-generic over $\bN$. But, for $\bar{c}\in\bar{\B}$, we have that $\bar{c}\in\bar{G}*\bar{H}$ iff $\bar{c}/\bar{G}\in\bar{H}$ iff $\sigma'(\bar{c}/\bar{G})=\sigma_1(\bar{c})/G\in H$ iff $\sigma_1(\bar{c})\in G*H=I$. Thus, $\bar{G}*\bar{H}=\sigma_1^{-1}``I=\bI$ is $\bar{\B}$-generic over $\bN$, as claimed.

So there are a name $\pi$ in $\V[G]^{\B/G}$ with $\sigma_1=\pi^H$ and a condition $d\in(\B/G)^+$ that forces over $\V[G]$ with respect to $\B/G$ that $\pi$ has the properties listed.

Now, all of this is true in $\V[G]$ whenever $G$ is $\bbA$-generic over $\V$, with $a\in G$, and so, there are names $\dot{d},\dot{\pi}\in\V^\bbA$ such that $d=\dot{d}^{G}$ and $\pi=\dot{\pi}^G$, and $a$ forces the situation described. Let $\dot{\sigma}_1$ be a $\B$-name such that $\dot{\sigma}_1^{G*H}=(\dot{\pi}^G)^H=\sigma_1$.

The only thing that's missing is the condition $c\in\B$ with $h(c)=a$ such that whenever $c\in I$, $I$ is $\B$-generic over $\V$, $G=I\cap\bbA$ and $\sigma_1=\dot{\sigma}_1^I$, and $\sigma_0=\dot{\sigma}_0^G$, it follows that (1)-(7) hold. To find the desired condition, first note that we may choose the name $\dot{d}$ in such a way that $\forces_\bbA\dot{d}\in\check{\B}/\dot{G}_\bbA$ and $a=\BV{\dot{d}\neq 0}_\bbA$. Namely, given the original $\dot{d}$ such that $a$ forces that $\dot{d}\in(\check{\B}/\dot{G}_\bbA)^+$ and all the other statements listed above, there are two cases: if $a=\eins_\bbA$, then since $a\le\BV{\dot{d}\neq 0}$, it already follows that $a=\BV{\dot{d}\neq 0}$ and $\forces_\bbA\dot{d}\in\check{\B}/\dot{G}_\bbA$. If $a<\eins_\bbA$, then let $\dot{e}\in\V^\bbA$ be a name such that $\forces_\bbA\dot{e}=0_{\check{\B}/\dot{G}_\bbA}$, and mix the names $\dot{d}$ and $\dot{e}$ to get a name $\dot{d}'$ such that $a\forces_\bbA\dot{d}'=\dot{d}$ and $\neg a\forces_\bbA\dot{d}'=\dot{e}$. Then $\dot{d}'$ is as desired. Clearly, $\forces_\bbA\dot{d}'\in\check{\B}/\dot{G}_\bbA$. Since $a\forces_\bbA\dot{d}'=\dot{d}$, it follows that $a\le\BV{\dot{d}'\neq 0}$, and since $\neg a\forces_\bbA\dot{d'}=\dot{e}$, it follows that $\neg a\le\BV{\dot{d}'=0}=\neg\BV{\dot{d}'\neq 0}$, so $\BV{\dot{d}'\neq 0}\le a$. So we could replace $\dot{d}$ with $\dot{d}'$.

Then, by \cite[\S0, Fact 4]{Jensen2014:SubcompleteAndLForcingSingapore},%
%%%%%%%%%%%%%%%%%%%%%%%%%%%%%%%%%%
\footnote{There is a slightly confusing misprint in the statement of that fact. It should read: ``Let $\bbA\sub\B$, and let $\forces_\bbA\dot{b}\in\check{\B}/\dot{G}_\bbA$, where $\dot{b}\in\V^\bbA$. There is a unique $b\in\B$ such that $\forces_\bbA\dot{b}=\check{b}/\dot{G}_\bbA$.'' That's what the proof given there shows.} %
%%%%%%%%%%%%%%%%%%%%%%%%%%%%%%%%%%
there is a unique $c\in\B$ such that $\forces_\bbA\check{c}/\dot{G}_\bbA=\dot{d}$, and it follows by \cite[\S0, Fact 3]{Jensen2014:SubcompleteAndLForcingSingapore} that
\[h(c)=\BV{\check{c}/\dot{G}_\bbA\neq 0}_\bbA=\BV{\dot{d}\neq 0}_\bbA=a\]
as wished.
\end{proof}

\subsection{RCS and nicely subproper iterations}
\label{subsec:RCSandNicelySubproperIterations}

We adopt Jensen's approach to RCS iterations.
Thus, an \emph{iteration of length $\alpha$} is a sequence $\seq{\B_i}{i<\alpha}$ of complete Boolean algebras such that for $i\le j<\alpha$, $\B_i\sub\B_j$, and such that if $\lambda<\alpha$ is a limit ordinal, then $\B_\lambda$ is generated by $\bigcup_{i<\lambda}\B_i$, meaning that $\B_\lambda$ is the completion of the collection of all infima and suprema of subsets of $\bigcup_{i<\lambda}\B_i$. In this setting, $\vec{b}=\seq{b_i}{i<\lambda}$ is a \emph{thread} in $\vec{\B}\rest\lambda$ if for every $i\le j<\lambda$, $b_i=h_{\B_j,\B_i}(b_j)$ and $b_j\neq 0$. %The support of a thread $\vec{b}$ is defined to be the set of $j<\lambda$ such that for all $i<j$, $b_i\neq b_j$.
$\B_\lambda$ is an \emph{inverse limit} of $\vec{\B}\rest\lambda$ if for every thread $\vec{b}$ in $\vec{\B}\rest\lambda$, $b^*:=\bigwedge_{i<\lambda}^{\B_\lambda}b_i\neq 0$, and if the set of such $b^*$ is dense in $\B_\lambda$. This characterizes $\B_\lambda$ up to isomorphism. If $\seq{\xi_i}{i<\blambda}$ is monotone and cofinal in $\lambda$ and $\vec{b}=\seq{b_i}{i<\blambda}$ is such that for every $i<\blambda$, $b_i\in\B_{\xi_i}$ and for every $i\le j<\blambda$, $h_{\B_{\xi_j},\B_{\xi_i}}(b_j)=b_i$, then we will consider $\vec{b}$ to be a thread in $\vec{B}\rest\lambda$ as well, since it gives rise to a thread $\vec{c}=\seq{c_i}{i<\lambda}$ in the original sense via the definition $c_i=h_{\B_{\xi_j},\B_i}(b_j)$ where $j$ is such that $\xi_j\ge i$, and vice versa, the restriction of a thread in the original sense to a cofinal index set determines the entire thread, so that these two notions are equivalent.
If $\vBB=\seq{\B_i}{i<\alpha}$ is an iteration as above, then $\alpha$ is the \emph{length} of $\vBB$. %We say that $\vBB$ is \emph{direct} if for every $i+1<\alpha$, $\B_i\neq\B_{i+1}$.

The direct limit takes $\B_\lambda$ as the minimal completion of $\bigcup_{i<\lambda}\B_i$ and is characterized by the property that $\bigcup_{i<\lambda}\B_i\ohne\{0\}$ is dense in $\B_\lambda$. Another way of looking at it is that it is generated by the eventually constant threads.

The \emph{RCS limit} is defined as the inverse limit, except that only RCS threads $\vec{b}$ are used: $\vec{b}=\seq{b_i}{i<\lambda}$ is an \emph{RCS thread} in $\vec{\B}\rest\lambda$ if it is a thread in $\vec{\B}\rest\lambda$ and there is an $i<\lambda$ such that either, for all $j<\lambda$ with $i\le j$, $b_i=b_j$, or $b_i\forces_{\B_i}\cf(\check{\lambda})=\check{\omega}$.

\begin{defn}
Let $\vec{\B}$ be an iteration of length $\alpha$.

Then $\vBB$ is \emph{direct} if for every $i+1<\alpha$, $\B_i\neq\B_{i+1}$.

It is \emph{standard} if it is direct and for every $i+1<\alpha$, letting $\delta_i=\delta(\B_i)$, $\forces_{\B_{i+1}}\card{\check{\delta}_i}\le\omega_1$.

It is an \emph{RCS iteration} if for every limit $\lambda$, $\B_\lambda$ is the RCS limit of $\vec{\B}\rest\lambda$.

Let $\Gamma=\{\B\st\phi_\Gamma(\B,p)\}$ be a class of complete Boolean algebras (defined by some formula $\phi_\Gamma$ in some parameter $p$).

Then an iteration $\vBB$ is an \emph{iteration of forcings in $\Gamma$} if for every $i+1<\alpha$,
$\forces_{\B_i}$``$\check{\B}_{i+1}/\dot{G}_i\in\Gamma$,'' (i.e., $\forces_{\B_i}\phi_\Gamma(\B_{i+1}/\dot{G}_i,\check{p})$).

$\Gamma$ is \emph{standard RCS iterable} if whenever $\vBB$ is a standard RCS iteration of forcings in $\Gamma$, then for every $h\le i<\alpha$, if $G_h$ is generic for $\B_h$, then in $\V[G_h]$, $\B_i/G_h\in\Gamma$ (i.e., $\phi_\Gamma(\B_i/G_h,p)$ holds in $\V[G_h]$).
\end{defn}

In the context of a given iteration $\vec{\B}$ as above, if $i<\alpha$ and $b\in\B_j$, for some $j<\alpha$, we'll just write $h_i(b)$ for $h_{\B_j,\B_i}(b)$. We'll write $\lh(\vec{\B})=\alpha$, the length of the iteration. The following fact summarizes the basic properties of RCS iterations.

\begin{fact}[{\cite[p.~142]{Jensen2014:SubcompleteAndLForcingSingapore}}]
\label{fact:BasicsOnRCSiterations}
Let $\vec{\B}=\seq{\B_i}{i<\alpha}$ be an RCS iteration.
\begin{enumerate}[label=(\arabic*)]
  \item
  \label{RCSFact.item:InverseLimitAtCountableCofinalities}
  If $\lambda<\alpha$ and $\cf(\lambda)=\omega$, then $\B_\lambda$ is the inverse limit of $\vec{\B}\rest\lambda$.
  \item
  \label{RCSFact.item:DirectLimitAtStronglyUncountableCofinalities}
  If $\lambda<\alpha$ and for every $i<\lambda$, $\forces_{\B_i}\cf(\check{\lambda})>\omega$, then $\bigcup_{i<\lambda}\B_i$ is dense in $\B_\lambda$ (that is, $\B_\lambda$ is formed using only eventually constant threads, making it the direct limit).
  \item
  \label{RCSFact.item:PropertiesAreHereditary}
  If $i<\lambda$ and $G$ is $\B_i$-generic, then the above are true in $\V[G]$ about the iteration $\seq{\B_{i+j}/G}{j<\alpha-i}$.
\end{enumerate}
\end{fact}

The following fact gives us some information about the chain conditions satisfied by direct limits in an iteration.

\begin{fact}[Baumgartner, see {\cite[Theorem 3.13]{VialeEtAl:BooleanApproachToSPiterations}}]
\label{fact:lambdacc}
Let $\seq{\B_i}{i<\lambda}$ be an iteration such that for every $\alpha<\lambda$, $\B_\alpha$ is ${<}\lambda$-c.c., and such that the set of $\alpha<\lambda$ such that $\B_\alpha$ is the direct limit of $\vec{\B}\rest\alpha$ is stationary. Then the direct limit of $\vec{\B}$ is ${<}\lambda$-c.c.
\end{fact}

A variation of the RCS iteration theorem for subproper forcing \cite[\S 4, pp.~2, Thm.~5]{Jensen:IterationTheorems} says:

\begin{thm}[Jensen]
\label{thm:RCSiterationOfSubPforcing}
The class of complete subproper Boolean algebras is standard RCS iterable.
%Let $\vBB$ be an RCS iteration of length $\alpha$. For $i<\alpha$, let $\delta_i=\delta(\B_i)$. Assume that for all $i+1<\alpha$:
%\begin{enumerate}
%  \item $\B_i\neq\B_{i+1}$.
%  \label{item:direct}
%  \item $\forces_{\B_i}\check{\B}_{i+1}$ is subproper.
%  \label{item:SubproperIterands}
%  \item $\forces_{\B_{i+1}}\check{\delta}_i$ has cardinality at most $\omega_1$.
%  \label{item:IntermediateCollapses}
%\end{enumerate}
%Then for every $i<\alpha$, $\B_i$ is subproper. In fact, for every $i<\alpha$, if $G_i$ is $\B_i$-generic and $j<\alpha$, then in $\V[G_i]$, $\B_j/G_i$ is subproper.
\end{thm}

In detail, Jensen proved the version of this theorem for subcomplete forcing in \cite[\S 3, Theorem 2]{Jensen:IterationTheorems}, and states that the version for subproper forcing can be reproven easily (see \cite[\S 4, p.~19]{Jensen:IterationTheorems}).

Jensen \cite{Jensen:ExtendedNamba} uses a more flexible notion of iteration of subcomplete forcing notions, and an elegant proof of a generalization of the main iteration theorem of that work is given in \cite{Jensen:IterationTheorems}. We follow the latter presentation here, albeit in the context of subproper forcing. The following is a version of \cite[\S 3, p.~9]{Jensen:IterationTheorems} translated from the subcomplete to the subproper context. We also work with $\delta(\B_i)$ rather than $\card{\B_i}$.

\begin{defn}[after Jensen]
\label{def:NicelyGammaIteration}
Let $\Gamma$ be a class of complete Boolean algebras.
A standard iteration $\vec{\B}=\seq{\B_i}{i<\alpha}$ is \emph{nicely $\Gamma$} if, letting $\delta_i=\delta(\B_i)$, for $i<\alpha$, the following hold:
\begin{enumerate}[label=(\arabic*)]
  \item
  \label{item:GammaIterands}
    Suppose $i+1<\alpha$. Then
    $\forces_{\B_i}\check{\B}_{i+1}/\dot{G}_i\in\Gamma$.
  \item
  \label{item:ConditionsAtCountableCofinality}
  Suppose $\lambda<\alpha$ is a limit ordinal of countable cofinality.
  \begin{enumerate}[label=(\alph*)]
    \item
    \label{item:ThreadsAreNonzero}
    If $\vec{b}$ is a thread in $\vec{\B}\rest\lambda$, then $\bigwedge_{i<\lambda}b_i\neq 0$ in $\B_\lambda$.
    \item
    \label{item:SubpropernessPropagates}
    If for every $i<\lambda$, $\B_i\in\Gamma$, then $\B_\lambda\in\Gamma$.
  \end{enumerate}
  \item
  \label{item:DirectLimitAtUncountableCofinality}
  Suppose that $\lambda<\alpha$ is a limit ordinal such that for every $i<\lambda$, $\forces_{\B_i}\cf(\check{\lambda})>\omega$. Then $\bigcup_{i<\lambda}\B_i$ is dense in $\B_\lambda$, that is, $\B_\lambda$ is the direct limit of $\vec{\B}\rest\lambda$.
  \item
  \label{item:AllOfThisHoldsInV[G_i]}
  Let $i<\alpha$. Then, if $G_i$ is $\B_i$-generic, \ref{item:GammaIterands}-\ref{item:DirectLimitAtUncountableCofinality} hold in $\V[G_i]$ for $\seq{\B_{i+j}/G_i}{j<\alpha-i}$.
\end{enumerate}
Finally, we say that $\Gamma$ is \emph{nicely iterable} if whenever $\vBB=\seq{\B_i}{i<\alpha}$ is a nicely $\Gamma$ iteration, then for every $h\le \delta<\alpha$, if $G_h$ is $\B_h$-generic, then in $\V[G_h]$, $\B_\delta/G_h\in\Gamma$.
%\begin{enumerate}[label=(\arabic*)]
%  \item
%  \label{item:SuccessorConditions}
%  Suppose $i+1<\alpha$.
%  \begin{enumerate}[label=(\alph*)]
%    \item
%    \label{item:SubproperIterands}
%    $\forces_{\B_i}\check{\B}_{i+1}/\dot{G}_i$ is subproper.
%    \item
%    \label{item:IntermediateCollapses}
%    $\forces_{\B_{i+1}}\check{\delta}_i\le\omega_1$.
%  \end{enumerate}
%  \item
%  \label{item:ConditionsAtCountableCofinality}
%  Suppose $\lambda<\alpha$ is a limit ordinal of countable cofinality.
%  \begin{enumerate}[label=(\alph*)]
%    \item
%    \label{item:ThreadsAreNonzero}
%    If $\vec{b}$ is a thread in $\vec{\B}\rest\lambda$, then $\bigwedge_{i<\lambda}b_i\neq 0$ in $\B_\lambda$.
%    \item
%    \label{item:SubpropernessPropagates}
%    If for every $i<\lambda$, $\B_i$ is subproper, then so is $\B_\lambda$.
%  \end{enumerate}
%  \item
%  \label{item:DirectLimitAtUncountableCofinality}
%  Suppose that $\lambda<\alpha$ is a limit ordinal such that for every $i<\lambda$, $\forces_{\B_i}\cf(\check{\lambda})>\omega$. Then $\bigcup_{i<\lambda}\B_i$ is dense in $\B_\lambda$, that is, $\B_\lambda$ is the direct limit of $\vec{\B}\rest\lambda$.
%  \item
%  \label{item:AllOfThisHoldsInV[G_i]}
%  Let $i<\alpha$. Then, if $G_i$ is $\B_i$-generic, \ref{item:SuccessorConditions}-\ref{item:DirectLimitAtUncountableCofinality} hold in $\V[G_i]$ for $\seq{\B_{i+j}/G_i}{j<\alpha-i}$.
%\end{enumerate}
\end{defn}

That is, in a nicely $\Gamma$ iteration, we already know that belonging to $\Gamma$ propagates to limit stages of countable cofinality, but we have almost no restrictions as to how those limit stages are formed. The following observation shows that iterations of subproper forcing as in Theorem \ref{thm:RCSiterationOfSubPforcing} are nicely subproper.

\begin{obs}
\label{obs:StandardRCSIterationsOfSubPAreNicelySubP}
Every standard RCS iteration of subproper forcings is nicely subproper.
\end{obs}

\begin{proof}
Let $i<\alpha$, and let $G_i$ be $\B_i$-generic.
We have to verify that conditions \ref{item:GammaIterands}-\ref{item:DirectLimitAtUncountableCofinality} of Definition \ref{def:NicelyGammaIteration} hold of $\seq{\B_{i+j}/G_i}{j<\alpha-i}$ in $\V[G_i]$.

Condition \ref{item:GammaIterands} is trivial: let $j+1<\alpha-i$. Since $\forces_{\B_{i+j}}\check{\B}_{i+j+1}/\dot{G}_{i+j}$ is subproper,
it clearly has to be that $\forces_{\B_{i+j}/G_i}(\check{\B}_{i+j+1}/\check{G}_i)/\dot{G}_{\B_{i+j}/\check{G}_i}$ is subproper in $\V[G_i]$. It can be shown similarly that $\seq{\B_{i+j}/G_i}{j<\alpha-i}$ is a standard iteration in $\V[G_i]$.

For condition \ref{item:ConditionsAtCountableCofinality}, let $\lambda<\alpha-i$ be a limit ordinal that has countable cofinality in $\V[G_i]$. By Fact \ref{fact:BasicsOnRCSiterations}, part \ref{RCSFact.item:PropertiesAreHereditary}, the first two parts of that fact apply to $\seq{\B_{i+j}/G_i}{j<\alpha-i}$ in $\V[G_i]$. Thus, $\B_\lambda/G_i$ is the inverse limit of $\seq{\B_{i+j}/G_i}{j<\lambda}$. This implies condition \ref{item:ConditionsAtCountableCofinality}\ref{item:ThreadsAreNonzero}.
And by Theorem \ref{thm:RCSiterationOfSubPforcing}, $\B_\lambda/G_i$ is subproper in $\V[G_i]$, so condition  \ref{item:ConditionsAtCountableCofinality}\ref{item:SubpropernessPropagates} holds.

Finally, condition \ref{item:DirectLimitAtUncountableCofinality} hold because part \ref{RCSFact.item:DirectLimitAtStronglyUncountableCofinalities} of Fact \ref{fact:BasicsOnRCSiterations} applies to $\seq{\B_{i+j}/G_i}{j<\alpha-i}$ in $\V[G_i]$, by part \ref{RCSFact.item:PropertiesAreHereditary} of that fact.
%for part \ref{item:SubproperIterands}, let $j+1<\alpha-i$. Since $\forces_{\B_{i+j}}\check{\B}_{i+j+1}/\dot{G}_{i+j}$ is subproper (by assumption \ref{obs.item:SubproperIterands},
%it clearly has to be that $\forces_{\B_{i+j}/G_i}(\check{\B}_{i+j+1}/\check{G}_i)/\dot{G}_{\B_{i+j}/\check{G}_i}$ is subproper in $\V[G_i]$. Part \ref{item:IntermediateCollapses} follows similarly from assumption \ref{obs.item:IntermediateCollapses}.
%
%For condition \ref{item:ConditionsAtCountableCofinality}, let $\lambda<\alpha-i$ be a limit ordinal that has countable cofinality in $\V[G_i]$. By Fact \ref{fact:BasicsOnRCSiterations}, part \ref{RCSFact.item:PropertiesAreHereditary}, the first two parts of that fact apply to $\seq{\B_{i+j}/G_i}{j<\alpha-i}$ in $\V[G_i]$. Thus, $\B_\lambda/G_i$ is the inverse limit of $\seq{\B_{i+j}/G_i}{j<\lambda}$. This implies condition \ref{item:ConditionsAtCountableCofinality}\ref{item:ThreadsAreNonzero}.
%And by Theorem \ref{thm:RCSiterationOfSubPforcing}, $\B_\lambda/G_i$ is subproper in $\V[G_i]$, so condition  \ref{item:ConditionsAtCountableCofinality}\ref{item:SubpropernessPropagates} holds.
%
%Finally, condition \ref{item:DirectLimitAtUncountableCofinality} hold because part \ref{RCSFact.item:DirectLimitAtStronglyUncountableCofinalities} of Fact \ref{fact:BasicsOnRCSiterations} applies to $\seq{\B_{i+j}/G_i}{j<\alpha-i}$ in $\V[G_i]$, by part \ref{RCSFact.item:PropertiesAreHereditary} of that fact.
\end{proof}

The following is a version of a theorem that Jensen proved for subcomplete forcing in \cite[\S3, pp.~9-11]{Jensen:IterationTheorems}. By Observation \ref{obs:StandardRCSIterationsOfSubPAreNicelySubP}, it generalizes Theorem \ref{thm:RCSiterationOfSubPforcing}.

\begin{thm}[Jensen]
\label{thm:NicelySubproperIterations}
The class of subproper Boolean algebras is nicely iterable.
%Suppose $\seq{\B_i}{i<\alpha}$ is nicely subproper. Then for every $h\le i<\alpha$, if $G_h$ is $\B_h$-generic, then in $\V[G_h]$, $\B_i/G_h$ is subproper.
\end{thm}

\begin{proof}
The proof is a virtual repetition of the argument of Jensen's proof of \cite[\S 4, Theorem 5, pp.~3-12]{Jensen:IterationTheorems}, incorporating the changes necessitated by working with $\delta(\B_i)$ (as in \cite[\S 3, p. 2, Theorem 2]{Jensen:IterationTheorems}),
except that it is somewhat simpler, because the limit of countable cofinality case is vacuous now. We have checked that the proof goes through, and so has Jensen (see
\cite[\S 4, last three lines on p.~19]{Jensen:IterationTheorems}).
\end{proof}

%We prove the theorem by induction on $i$. The case $h=i$ is trivial, and the successor case follows from the Two Step Theorem. So let us assume that $i$ is a limit ordinal. We distinguish two cases.
%
%\noindent\emph{Case 1:} $\cf(i)<\card{\B_k}$, for some $k<i$.
%
%It suffices to prove the theorem for $h>k$. By \ref{def:NicelySubproperIteration}, part \ref{item:IntermediateCollapses}, we then know that if $G_h$ is $\B_h$-generic, then in $\V[G_h]$, $\cf(i)\le\omega_1$. By part \ref{item:AllOfThisHoldsInV[G_i]} of Definition \ref{def:NicelySubproperIteration}, we could now work in $\V[G_h]$, and we would be in much the same situation as in $\V$. Thus, we may assume that already in $\V$, $\cf(i)\le\omega_1$. There are then two subcases.
%
%\noindent\emph{Case 1.1:} $\cf(i)=\omega$.
%
%Then $\B_i$ is subproper by part \ref{item:ConditionsAtCountableCofinality}\ref{item:SubpropernessPropagates} of Definition \ref{def:NicelySubproperIteration}.
%
%\noindent\emph{Case 1.2:} $\cf(i)=\omega_1$.
%
%In this case, we have to repeat the argument of Jensen's proof of \cite[\S 4, pp.~3-12]{Jensen:IterationTheorems}.

\subsection{Iterating subproper Souslin tree preserving forcing}
\label{subsec:IteratingSubproperSouslinTreePreservingForcing}

The main idea for this section stems from Miyamoto \cite[Lemma 5.0]{Miyamoto:IteratingSemiproperPreorders}, even though we do not employ his ``nice iterations'' here. In the original setting, a Souslin tree $T$ is fixed, and it is shown that nice limits of nice iterations of semi-proper forcing notions that preserve $T$ also preserve $T$. The corresponding theorem holds for subproper forcing as well, and one can work with RCS iterations rather than nice iterations too, as we shall show. However, we will first prove a different version of this preservation fact, because we want to establish a proof template that we will reuse in different situations later. The proof of that fact will be slightly more complicated and hence more suitable for these later variations.

\begin{defn}
\label{def:SouslinTreePreservingForcing}
A forcing notion $\P$ preserves Souslin trees if for every Souslin tree $T$, $\forces_\P$``$\check{T}$ is Souslin.''
\end{defn}

The main difference between this concept and the preservation of a fixed Souslin tree, when forming iterations of such forcing notions, is that iterands in an iteration of Souslin tree preserving forcing notions are required to preserve the Souslin trees that may have been added by earlier stages of the iteration, not only one fixed $T$, or some collection of Souslin trees in the ground model. We will give the proof of the following theorem in considerable detail, in order to establish a point of reference for later variations of the argument.

\begin{thm}
\label{thm:RCSIteratingSubproperSouslinTreePreservingForcing}
The class of subproper Boolean algebras that are Souslin tree preserving is standard RCS iterable.
%
%Let $\seq{\B_i}{i\le\alpha}$ be an RCS iteration such that for all $i+1<\alpha$, the following hold:
%\begin{enumerate}
%  \item $\B_i\neq\B_{i+1}$,
%  \item $\forces_{\B_i} (\check{\B}_{i+1}/\dot{G}_{\B_i}\ \text{is subproper and Souslin tree preserving})$,
%  \item $\forces_{\B_{i+1}} (\delta(\check{\B}_i)\ \text{has cardinality at most}\ \omega_1)$.
%\end{enumerate}
%Then whenever $h\le\delta<\alpha$ and $G_h$ is $\B_h$-generic, in $\V[G_h]$, $\B_\delta/G_h$ is subproper and Souslin tree preserving.
\end{thm}

\begin{proof}
Let $\vec{\B}=\seq{\B_i}{i<\alpha}$ be a standard RCS iteration of forcings that are subproper and Souslin-tree preserving.
We prove by induction on $\delta$: whenever $h\le\delta<\alpha$ and $G_h$ is $\B_h$-generic, then in $\V[G_h]$, $\B_\delta/G_h$ is subproper and Souslin tree preserving.

It suffices to focus on the Souslin tree preservation, since by Theorem \ref{thm:NicelySubproperIterations}, $\B_\delta/G_h$ is subproper in $\V[G_h]$ whenever $h\le\delta<\alpha$ and $G_h$ is $\B_h$-generic.

The successor case is trivial, as is the case $h=\delta$. So let $\delta$ be a limit ordinal, and let $h<\delta$.

Note that by Observation \ref{obs:StandardRCSIterationsOfSubPAreNicelySubP}, $\vBB$ is nicely subproper. We will use this, rather than that $\vBB$ is an RCS iteration, whenever possible. Doing so will make it easier to slightly generalize the theorem later.

\noindent{\bf Case 1:} there is an $i<\delta$ such that $\cf(\delta)\le\delta(\B_i)$.

Fix such an $i$. Then whenever $i<j<\delta$, $\forces_j\cf(\check{\delta})\le\omega_1$. It suffices to prove
\begin{enumerate}
\item[(A)] if $i<j<\delta$ and $G_j$ is $\B_j$-generic, then in $\V[G_j]$, $\B_\delta/G_j$ is Souslin tree preserving.
\end{enumerate}
For if we have done so, then the full claim follows: let $h\le i$, and let $G_h\sub\B_h$ be generic. By Theorem \ref{thm:NicelySubproperIterations}, we know that $\B_\delta/G_h$ is subproper in $\V[G_h]$. Thus, it suffices to prove that if $T$ is some Souslin tree in $\V[G_h]$ and $H$ is $\B_\delta/G_h$-generic over $\V[G_h]$, then $T$ is a Souslin tree in $\V[G_h][H]$. But letting $G_{i+1}=(G_h*H)\cap\B_{i+1}$, we know inductively that $T$ is Souslin in $\V[G_{i+1}]=\V[G_h][G_{i+1}/G_h]$, and so, by $(A)$, $T$ is Souslin in $\V[G_{i+1}*((G_h*H)/G_{i+1})]=\V[G_h][H]$.

To prove $(A)$, we would now have to fix a $j<\delta$, a $G_j\sub\B_j$ that's generic, and a $T$ such that in $\V[G_j]$, $T$ is Souslin. We'd have to prove in $\V[G_j]$ that $\B_\delta/G_j$ preserves $T$ as a Souslin tree.
But the iteration $\seq{\B_{j+i}/G_j}{i<\delta-j}$ is RCS (see Fact \ref{fact:BasicsOnRCSiterations}) in $\V[G_j]$, and it satisfies everything in $\V[G_j]$ that we assumed about $\vBB$ in $\V$, with the addition of the fact that in $\V[G_j]$, $\cf(\delta)\le\omega_1$.

Thus, it suffices to show:
\begin{enumerate}
\item[(B)] if $\cf(\delta)\le\omega_1$ then $\B_\delta$ is Souslin tree preserving.
\end{enumerate}
For the argument, carried out in $\V[G_j]$ would prove $(A)$.

Note that by arguing in $\V[G_j]$, but pretending $\V[G_j]$ is $\V$, we effectively absorbed $T$ into $\V$. It is this step that's not necessary if one only wants to preserve one fixed ground model Souslin tree.

To prove (B), let us fix a Souslin tree $T$.

Note that if $\cf(\delta)=\omega_1$, then for every $i<\delta$, $\forces_{\B_i}\cf(\check{\delta})=\check{\omega}_1$, as $\B_i$ preserves $\omega_1$. Thus, $\B_\delta$ is the direct limit of $\vec{\B}\rest\delta$, by part \ref{item:DirectLimitAtUncountableCofinality} of Definition \ref{def:NicelyGammaIteration}, that is, $\bigcup_{i<\delta}\B_i$ is dense in $\B_\delta$ in this case. Let us denote this dense set by $X$.

If, on the other hand, $\cf(\delta)=\omega$, then since $\vBB$ is an RCS iteration, then we know by Fact \ref{fact:BasicsOnRCSiterations} that the set $\{\bigwedge_{i<\delta}t_i\st\seq{t_i}{i<\delta}\ \text{is a thread in}\ \vBB\rest\delta\}$ is dense in $\B_\delta$. In case $\cf(\delta)=\omega$, let $X$ be that dense subset of $\B_\delta$.

Let $\pi:\omega_1\To\delta$ be cofinal, with $\pi(0)=0$.

Let $\dot{A}$ be a $\B_\delta$-name for a maximal antichain in $T$, and let $a_0\in X$ be a condition. We will find a countable $\zeta$ and a condition extending $a_0$ that forces that $\dot{A}\sub\check{T}|\check{\zeta}$. Here, $T|\zeta$ is the union of the levels of $T$ below $\zeta$.

Let $N$ be a model of the form $L_\tau^{\hat{A}}$, with $H_\theta\sub N$, such that $\theta$ verifies the subproperness of each $\B_i$, for $i\le\delta$ and $H_\theta$ is large enough to contain the parameters we need. Specifically, letting $S=\kla{\theta,\delta,\vec{\B},X,\dot{A},T,\pi,a_0}$, we want to fix
$M_0\prec N$ with $S\in M_0$, $M_0$ countable and such that, letting $\sigma_0:\bN\To M_0$ be the inverse of the Mostowski collapse (so that $\bN$ is transitive), $\bN$ is full -- it is easy to see that this situation can be arranged; see the discussion following \cite[Def.~2.21]{Fuchs:CanonicalFragmentsOfSRP}. Let $\tdelta=\sup(M_0\cap\delta)$, and let $\Omega=M_0\cap\omega_1=\crit(\sigma_0)=\omega_1^\bN$. Fix an enumeration $\seq{s_n}{n<\omega}$ of $T(\Omega)$, the $\Omega$-th level of $T$.

Let $\sigma_0^{-1}(S)=\bS=\kla{\btheta,\bdelta,\vec{\bar{\B}},\bX,\dot{\bA},\bT,\bpi,\bar{a}_0}$.
Let $\seq{\nu_i}{i<\omega}$ be a sequence of ordinals $\nu_i<\omega_1^\bN$ such that if we let $\bgamma_i=\bpi(\nu_i)$, it follows that $\seq{\bgamma_i}{i<\omega}$ is cofinal in $\bdelta$, and such that $\nu_0=0$, so that $\bgamma_0=0$. Hence, letting $\gamma_i=\sigma_0(\bgamma_i)$, we have that $\sup_{i<\omega}\gamma_i=\sup(M_0\cap\delta)=\tdelta$. Moreover, whenever $\sigma':\bN\prec N$ is such that $\sigma'(\bpi)=\pi$, it follows that for every $i<\omega$, $\sigma'(\bgamma_i)=\gamma_i=\pi(\nu_i)$, since $\sigma'(\bgamma_i)=\sigma'(\bpi(\nu_i))=\sigma'(\bpi)(\nu_i)=\pi(\nu_i)$.

By induction on $n<\omega$, construct sequences $\seq{\dot{\sigma}_n}{n<\omega}$, $\seq{c_n}{n<\omega}$, $\seq{\dot{b}_n}{n<\omega}$ and $\seq{\dot{\bar{b}}_n}{n<\omega}$ with $c_n\in\B_{\gamma_n}$, $\dot{\sigma}_n,\dot{b}_n,\dot{\barb}_n\in\V^{\B_{\gamma_n}}$, such that for every $n<\omega$, $c_n$ forces the following statements with respect to $\B_{\gamma_n}$:
\begin{enumerate}
\item $\dot{\sigma}_n:\check{\bN}\prec\check{N}$,
\item $\dot{\sigma}_n(\check{\bS})=\check{S}$, and for all $k<n$, $\dot{\sigma}_n(\dot{\bar{b}}_k)=\dot{\sigma}_k(\dot{\bar{b}}_k)$,
\item $\dot{\sigma}_n(\dot{\bar{b}}_n)=\dot{b}_n$ and $\dot{\bar{b}}_n\in\check{\bX}$ (and so, $\dot{b}_n\in\check{X}$),
\item $\check{h}_{\gamma_n}(\dot{b}_n)\in\dot{G}_{\B_{\gamma_{n}}}$,
\item $\dot{\sigma}_n^{-1}``\dot{G}_{\B_{\gamma_n}}$ is $\check{\bN}$-generic for $\check{\bar{\B}}_{\gamma_n}$,
\item $\dot{b}_n/\dot{G}_{\gamma_n}$ forces wrt.~$\check{\B_\delta}/\dot{G}_{\gamma_n}$ that there is a node $t<s_{n-1}$ with $t\in\dot{A}$ (for $n>0$),
\item $c_{n-1}=h_{\gamma_{n-1}}(c_{n})$ (for $n>0$),
\item $\dot{b}_n\le_{\check{\B}_\delta}\dot{b}_{n-1}$ (for $n>0$),
%\item $\sup\dot{\sigma}_n``(\On\cap\bN)=\sup\dot{\sigma}_{n-1}``(\On\cap\bN)$.
\end{enumerate}
To start off, in the case $n=0$, we set $c_0=\eins$, $\dot{\sigma}_0=\check{\sigma}_0$ and $\dot{b}_0=\check{a}_0$ and $\dot{\barb}_0=\sigma_0^{-1}(\dot{b}_0)$. Clearly then, (1)-(6) are satisfied for $n=0$ (and (7)-(8), as well as the second part of (2), are vacuous for $n=0$).

Now suppose $m=n-1$, and $\dot{\sigma}_l$, $c_l$, $\dot{b}_l$ and $\dot{\barb}_l$ have been defined so that (1)-(8) are satisfied
for $l\le m$.

An application of Lemma \ref{lem:SubpropernessExtensionLemma} (to $\B_{\gamma_m}\sub\B_{\gamma_n}$) yields a condition $c_n\in\B_{\gamma_n}$ and a $\B_{\gamma_n}$-name $\dot{\sigma}_n$ such that $h_{\gamma_m}(c_n)=c_m$ and whenever $I$ is $\B_{\gamma_n}$-generic over $\V$ with $c_n\in I$, $G=I\cap\B_{\gamma_m}$ and $\sigma_n=\dot{\sigma}_n^I$ it follows that $\sigma_n:\bN\prec N$, $\sigma_n^{-1}``I$ is $\bar{\B}_{\bgamma_n}$-generic over $\bN$,
$\sigma_n(\bS)=\sigma_m(\bS)$, $\sigma_n(\barb_k)=\sigma_m(\barb_k)$ for $k\le m$, where $\sigma_m=\dot{\sigma}_m^G$ and for $k\le m$, $\barb_k=\dot{\barb}_k^{I\cap\B_{\gamma_k}}$ as well as $b_k=\dot{b}_k^{I\cap\B_{\gamma_k}}$.
Moreover, we can arrange that $h_{\gamma_n}(\dot{b}_m^G)\in I$. For this last property, let $\dot{b}$ from the statement of Lemma \ref{lem:SubpropernessExtensionLemma} be a $\B_{\gamma_m}$-name for $h_{\gamma_n}(\dot{b}_m)$ and $\dot{\barb}$ a name for the preimage of $b_m$ under $\sigma_n$. Since inductively, $c_m$ forces that $\check{h}_{\gamma_m}(\dot{b}_m)\in\dot{G}_{\gamma_m}$, assumption (3) of the lemma is satisfied, and we get that $h_{\gamma_n}(\dot{b}_m^G)=\dot{b}^G\in I$, as wished.

With these definitions, all the conditions (1)-(8) are satisfied, as long as they don't concern $\dot{b}_n$ and $\dot{\barb}_n$.
To define $\dot{b}_n$ and $\dot{\barb}_n$, let $I$ and $G$ as described in the previous paragraph. Let $\sigma_n=\dot{\sigma}_n^I$, $b_m=\dot{b}_m^I$, $M_n=\ran(\sigma_n)$, $\bI=\sigma_n^{-1}``I$, $\sigma_n^*:\bN[\bI]\isomorphism M_n[I]\prec N[I]$ with $\sigma_n\sub\sigma_n^*$ and $\sigma_n^*(\bI)=I$.

Working in $\V[I]$, we know that $b_m/I$ forces wrt.~$\B_\delta/I$ that $\dot{A}/I$ is a maximal antichain in $T$ (since $b_m\le\dot{b}_0^I=a_0$). Thus, the set
\[D=\{t\in T\st\exists y\in X\exists\bar{t}\le_T t\quad y\le_\delta b_m\land h_{\gamma_n}(y)\in I\
\text{and}\ y/I\forces_{\B_\delta/I}\check{\bar{t}}\in\dot{A}/I\}\]
is dense in $T$. Note that $D\in M_n[I]$.
Working in $M_n[I]$, let $A\sub D$ be a maximal antichain, so
$A\in M_n[I]$.
Since $\B_{\gamma_n}$ preserves $T$ as a Souslin tree, it follows that $A$ is countable in $M_n[I]$, and so, $A\sub M_n[I]$. Since $D$ is dense in $T$, $A$ is (in $\V[I]$) a maximal antichain in $T$, and $A\sub T|\Omega$ (note that $\crit(\sigma_n)=\omega_1^\bN=\Omega$). Hence, there is a $t\in A$ with $t<_T s_m$. Since $A\sub D$, it follows that $t\in D$. Working in $M_n[I]$ again, let $b_n\in X$, $\bar{t}\le_T t$ witness that $t\in D$, i.e., $b_n\le_{\delta}b_m$, $h_{\gamma_n}(b_n)\in I$ and $b_n/I$ forces wrt.~$\B_\delta/I$ that $\check{\bar{t}}\in\dot{A}/I$. Let $\barb_n=(\sigma^*_m)^{-1}(b_n)$.

Since all of this holds in $\V[I]$ whenever $I$ is $\B_{\gamma_n}$-generic over $\V$ and $c_n\in I$, there are $\B_{\gamma_m}$-names $\dot{\bar{t}}$ for $\bar{t}$ and $\dot{b}_n$ for $b_n$ such that $c_n$ forces all of this wrt.~$\B_{\gamma_n}$. In particular, $c_n$ forces that $\dot{\bar{t}}\in\dot{A}$ and $\dot{\bar{t}}<_Ts_{m}$. So (1)-(8) are satisfied. This finishes the recursive construction.

By (7), the sequence $\seq{c_n}{c<\omega}$ is a thread, and so, $c=\bigwedge_{n<\omega}c_n\in\B_\delta^+$:
if $\cf(\delta)=\omega$, then this follows from part \ref{item:ConditionsAtCountableCofinality}\ref{item:ThreadsAreNonzero} of Definition \ref{def:NicelyGammaIteration}. And if $\cf(\delta)=\omega_1$, then $\gamma:=\sup_{n<\omega}\gamma_n<\delta$ and $\cf(\gamma)=\omega$, so again, $c\in\B_\gamma^+\sub\B_\delta^+$, for the same reason.

We claim that $c$ forces that $\dot{A}$ is bounded in $\check{T}$. To see this, let $G$ be $\B_\delta$-generic over $\V$, with $c\in G$. For $n<\omega$, let $b_n=\dot{b}_n^G$.

\claim{(C)}{For all $n<\omega$, $b_n\in G$.}

\begin{proof}[Proof of (C)]
Let $n<\omega$. We have that $h_{\gamma_n}(b_n)\in G\cap\B_{\gamma_n}$, by (4). For every $l\ge n$, $b_l\le b_n$, by (8), so $h_{\gamma_l}(b_l)\le h_{\gamma_l}(b_n)$. Since  $h_{\gamma_l}(b_l)\in G\cap\B_{\gamma_l}$, this implies that for $l\ge n$,
\[h_{\gamma_l}(b_n)\in G\cap\B_{\gamma_l}.\]
This holds for $l<n$ as well, because in that case, $h_{\gamma_l}(b_n)\ge h_{\gamma_n}(b_n)\in G$.
%Now let $c'\in G$ extend $c$ and decide $\dot{b}_n$, that is, $c'\forces_{\B_\delta}\dot{b}_n=\check{b}_n$.
%Since $c'\forces_{\B_\delta}h_{\gamma_l}(\check{b}_n)\in\dot{G}_{\B_{\gamma_l}}$, it follows that $h_{\gamma_l}(c')\le h_{\gamma_l}(b_n)$, so since $c'\le h_{\gamma_l}(c')$,
%\[c'\le h_{\gamma_l}(b_n)\]
%for every $l\ge n$.

Recall that $\tdelta=\sup_{l<\omega}\gamma_l$. Let $M_n=\ran(\sigma_n)$. It follows that $\sup(\delta\cap M_n)=\tdelta$ (since $\sigma_n(\bpi)=\pi$).

Now, if $\cf(\delta)=\omega_1$, then $\tdelta<\delta$, and we know that $b_n\in X=\bigcup_{i<\delta}\B_i$, so there is an $i<\delta$ such that $b_n\in\B_i$. But since $b_n\in M_n$, the same is true in $M_n$, and this means that there is an $i<\tdelta$ such that $b_n\in\B_i$. But then, letting $l>n$ be such that $\gamma_l>i$, it follows that $b_n=h_{\gamma_l}(b_n)\in G\cap\B_{\gamma_l}$, so $b_n\in G$, as claimed.

If, on the other hand, $\cf(\delta)=\omega$, then since $b_n\in X$, there is a thread $\seq{t_i}{i<\delta}$ in $\vBB\rest\delta$ such that $b_n=\bigwedge_{i<\delta}t_i$. But then, for every $l<\omega$, $t_{\gamma_l}=h_{\gamma_l}(b_n)\in G\cap\B_{\gamma_l}\sub G$. By genericity (and thus $\V$-completeness) of $G$, this implies that $b_n=\bigwedge_{i<\delta}t_i=\bigwedge_{l<\omega}t_{\gamma_l}\in G$.
Note that we used that $\vBB$ is an RCS iteration here.
\end{proof}

Note in particular that $b_0=a_0\in G$. Since this is true whenever $c\in G$, $c\forces_{\B_\delta}\check{a}_0\in\dot{G}$, which implies that $c\le a_0$.
Moreover, $c$ forces that $\dot{A}$ is bounded in $\check{T}$:
working in $\V[G]$ again, where $G\ni c$ is $\B_\delta$-generic, we have that for every $n<\omega$, there is a $t_n\in A=\dot{A}^G$ with $t_n<_T s_n$, by (6). So $A$ cannot contain a node $a$ at a level greater than $\Omega$, because the predecessor of such an $a$ at level $\Omega$ would have to be of the form $s_m$, for some $m$, and $s_m>_Tt_m$. So $a,t_m\in A$ would be comparable.
Thus, $c$ forces wrt.~$\B_\delta$ that $\dot{A}\sub\check{T}|\check{\Omega}$.

\noindent{\bf Case 2:} for all $i<\delta$, $\cf(\delta)>\delta(\B_i)$.

We may also assume that $\cf(\delta)>\omega_1$, for otherwise, $\cf(\delta)\le\omega_1$ and the argument of case 1 goes through (recall that we proved (B)).

It follows as in \cite[p.~143, claim (2)]{Jensen2014:SubcompleteAndLForcingSingapore} that for $i<\delta$, $\card{i}\le\delta(\B_i)$. But then, it follows that $\delta$ is regular, for otherwise, if $i=\cf(\delta)<\delta$, it would follow that $\cf(\delta)=i\le\delta(\B_i)<\cf(\delta)$.

Thus, $\delta$ is a regular cardinal, and $\delta\ge\omega_2$. Hence, $S^\delta_{\omega_1}$, the set of ordinals less than $\delta$ with cofinality $\omega_1$, is stationary in $\delta$. For $\gamma\in S^\delta_{\omega_1}$, since $\B_\gamma$, being subproper, preserves $\omega_1$, it follows that for every $i<\gamma$, $\forces_{\B_i}\cf(\check{\gamma})>\omega$. Thus, since $\vBB$ is nicely subproper, it follows by part \ref{item:DirectLimitAtUncountableCofinality} of Definition \ref{def:NicelyGammaIteration} that $\B_\gamma$ is the direct limit of $\vBB\rest\gamma$.
Moreover, since for $i<\delta$, $\delta(\B_i)<\delta=\cf(\delta)$, it follows by Observation \ref{obs:P_is_delta(P)+cc} that $\B_i$ is ${<}\delta(\B_i)^+$-c.c., and hence ${<}\delta$-c.c.

It follows by Fact \ref{fact:lambdacc} that the direct limit of $\vec{\B}\rest\delta$ is ${<}\delta$-c.c.

Again, since for all $i<\delta$, $\B_i$ is ${<}\delta$-c.c., it follows that $\B_i$ forces that the cofinality of $\delta$ is uncountable. So since $\vBB$ is nicely subproper, it follows that $\B_\delta$ is the direct limit of $\vBB\rest\delta$, and hence that $\B_\delta$ is ${<}\delta$-c.c.

Now let $h<\delta$, let $G_h\sub\B_h$ be generic, and let $T\in\V[G_h]$ be a Souslin tree. We have to show that $T$ is still Souslin in $\V[G_h][H]$, whenever $H\sub\B_\delta/G_h$ is generic over $\V[G_h]$. But if there were an uncountable antichain of $T$ in $\V[G_h][H]=\V[G_h*H]$, then since $\B_\delta$ is $\delta$-c.c.~and $\delta>\omega_1$, it would follow that this antichain exists already in $\V[(G_h*H)\cap\B_i]$, for some $i\in[h,\delta)$. Thus, $\B_i/G_h$ would fail to be Souslin tree preserving in $\V[G_h]$, contradicting our inductive assumption.
\end{proof}

\begin{defn}
\label{def:PreservingSouslinnessOfT}
Let $T$ be a Souslin tree and $\P$ a notion of forcing. $\P$ \emph{preserves the Souslinness of $T$} if $\forces_\P$``$\check{T}$ is Souslin.''
\end{defn}

\begin{thm}
\label{thm:RCSIteratingSubproperTPreservingForcing}
Let $T$ be a Souslin tree. Then the class of subproper Boolean algebras that preserve the Souslinness of $T$ is standard RCS iterable.
%Let $\seq{\B_i}{i\le\alpha}$ be an RCS iteration such that for all $i+1<\alpha$, the following hold:
%\begin{enumerate}
%  \item $\B_i\neq\B_{i+1}$,
%  \item $\forces_{\B_i} (\check{\B}_{i+1}/\dot{G}_{\B_i}\ \text{is subproper and preserves $\check{T}$ as a Souslin tree})$,
%  \item $\forces_{\B_{i+1}} (\delta(\check{\B}_i)\ \text{has cardinality at most}\ \omega_1)$.
%\end{enumerate}
%Then whenever $h\le\delta<\alpha$ and $G_h$ is $\B_h$-generic, in $\V[G_h]$, $\B_\delta/G_h$ is subproper and preserves $T$ as a Souslin tree.
\end{thm}

\begin{proof}
Letting $\vec{\B}=\seq{\B_i}{i<\alpha}$ be a standard RCS iteration of forcings that are subproper and preserve $T$ as a Souslin tree, we prove by induction on $\delta$: whenever $h\le\delta<\alpha$ and $G_h$ is $\B_h$-generic, then in $\V[G_h]$, $\B_\delta/G_h$ is subproper and preserves $T$ as a Souslin tree.

For this, the proof of Theorem \ref{thm:RCSIteratingSubproperSouslinTreePreservingForcing} goes through almost without change. Some steps of the argument are somewhat simpler in the present context, because the Souslin tree $T$ is in the ground model.
\end{proof}

Looking over the proof of Theorem \ref{thm:RCSIteratingSubproperSouslinTreePreservingForcing}, one sees that the assumption that the iteration in question uses revised countable support was only used at stages of the iteration that acquire countable cofinality. Thus, we obtain nice iterability results for the corresponding forcing classes.

As before, the previous RCS iteration theorems imply:

\begin{obs}
\label{obs:StandardRCSIterationsOfSubPTpresAreNicelySubPTpres}
Every standard RCS iteration of subproper and Souslin tree preserving forcings is nicely subproper and Souslin tree preserving. Similarly, if $T$ is a fixed Souslin tree then every standard RCS iteration of subproper and Souslinness of $T$ preserving forcings is nicely subproper and Souslinness of $T$ preserving.
%
%Let $T$ be a Souslin tree, and let
%$\vec{\B}$ be an RCS iteration of length $\alpha$ such that for all $i+1<\alpha$:
%\begin{enumerate}[label=(\arabic*)]
%  \item $\B_i\neq\B_{i+1}$.
%  \label{obs.item:direct}
%  \item $\forces_{\B_i}\check{\B}_{i+1}/\dot{G}_i$ is subproper and $T$-preserving.
%  \label{obs.item:SubproperIterands}
%  \item $\forces_{\B_{i+1}}\check{\B}_i$ has cardinality at most $\omega_1$.
%  \label{obs.item:IntermediateCollapses}
%\end{enumerate}
%Then $\vec{\B}$ is nicely subproper $T$-preserving.
\end{obs}

Thus, we can generalize Theorems \ref{thm:RCSIteratingSubproperSouslinTreePreservingForcing} and
\ref{thm:RCSIteratingSubproperTPreservingForcing}
as follows.

\begin{thm}
\label{thm:NicelyIteratingSubproperT/SouslinnessPreserving}
The class of subproper and Souslin tree preserving Boolean algebras is nicely iterable, and so is the class of subproper Boolean algebras that preserve some fixed Souslin tree $T$.
\end{thm}

\begin{proof}
Recall that in the proof of Theorem \ref{thm:RCSIteratingSubproperSouslinTreePreservingForcing}, we used the fact that the iteration was just nicely subproper, rather than an RCS iteration,  wherever possible. The only places in the argument that used that we were dealing with an RCS iteration occurred at stages of the iteration that acquired countable cofinality. But these stages are trivial in a nicely subproper $T$-preserving iteration.
\end{proof}

\subsection{Nicely subproper iterations of $[T]$-preserving forcing}
\label{subsec:NicelySubPiterationsOf[T]preservingForcing}

\begin{defn}
\label{def:[T]-preservationAndNotAddingBranches}
Let $T$ be an $\omega_1$-tree. Then $[T]$ denotes the set of cofinal branches of $T$, that is, the set of branches of $T$ that have order type $\omega_1$. We say that a forcing notion $\P$ is \emph{$[T]$-preserving} to express that $\P$ cannot add new cofinal branches through $T$, that is, that $\forces_\P[\check{T}]\sub\check{\V}$. As with the preservation of a fixed Souslin tree, there is a more general version of this preservation property: $\P$ \emph{is branch preserving} if for every $\omega_1$-tree $S$, $\P$ is $[S]$-preserving.
\end{defn}

\begin{thm}
\label{thm:RCSIteratingSubproperForcingNotAddingBranchesToOmega1Trees}
The class of subproper and branch preserving forcing notions is standard RCS iterable.
\end{thm}

\begin{proof}
We follow along the lines of the proof of Theorem \ref{thm:RCSIteratingSubproperSouslinTreePreservingForcing}.
So let $\vec{\B}=\seq{\B_i}{i<\alpha}$ be a standard RCS iteration of forcings that are subproper and branch preserving.
We prove by induction on $\delta$: whenever $h\le\delta<\alpha$ and $G_h$ is $\B_h$-generic, then in $\V[G_h]$, $\B_\delta/G_h$ is subproper and branch preserving.

\noindent{\bf Case 1:} $\delta$ is a limit ordinal and there is an $i<\delta$ such that $\cf(\delta)\le\delta(\B_i)$.

As before, it suffices to prove
\begin{enumerate}
\item[(B)] if $\cf(\delta)\le\omega_1$ then $\B_\delta$ adds no cofinal branch to an $\omega_1$-tree.
\end{enumerate}

So let us fix an $\omega_1$-tree $T$.
Depending on whether the cofinality of $\delta$ is $\omega$ or $\omega_1$, let $X=\{\bigwedge_{i<\delta}t_i\st\vec{t}\ \text{is a thread in}\ \vec{\B}\rest\delta\}$ or $X=\bigcup_{i<\delta}\B_i$.

Towards a contradiction, let $\dot{B}$ be a $\B_\delta$-name such that some condition $a_0\in X$ forces that $\dot{B}$ is a new cofinal branch, that is, that $\dot{B}\notin\V^{\B_\gamma}$ for all $\gamma<\delta$.

Let $\pi:\omega_1\To\delta$ be cofinal, with $\pi(0)=0$.
Let $N=L_\tau^A$, with $H_\theta\sub N$, such that $\theta$ verifies the subproperness of each $\B_i$, for $i\le\delta$.
Let $S=\kla{\theta,\delta,\vec{\B},X,\dot{B},T,\pi,a_0}$.
Let $M_0\prec N$ with $S\in M_0$, $M_0$ countable and such that, letting $\sigma_0:\bN\To M_0$ be the inverse of the Mostowski collapse, where $\bN$ is full. Let $\tdelta=\sup(M_0\cap\delta)$, and let $\Omega=M_0\cap\omega_1=\crit(\sigma_0)=\omega_1^\bN$. Fix an enumeration $\seq{s_n}{n<\omega}$ of $T(\Omega)$.

Let $\sigma_0^{-1}(S)=\bS=\kla{\btheta,\bdelta,\vec{\bar{\B}},\bX,\dot{\bB},\bT,\bpi,\bar{a}_0}$.
Let $\seq{\nu_i}{i<\omega}$ be a sequence of ordinals $\nu_i<\omega_1^\bN$ such that if we let $\bgamma_i=\bpi(\nu_i)$, it follows that $\seq{\bgamma_i}{i<\omega}$ is cofinal in $\bN$, and such that $\nu_0=0$, so that $\bgamma_0=0$. Let $\gamma_i=\sigma_0(\bgamma_i)$.

By induction on $n<\omega$, construct sequences $\seq{\dot{\sigma}_n}{n<\omega}$, $\seq{c_n}{n<\omega}$, $\seq{\dot{b}_n}{n<\omega}$ and $\seq{\dot{\bar{b}}_n}{n<\omega}$ with $c_n\in\B_{\gamma_n}$, $\dot{\sigma}_n,\dot{b}_n,\dot{\barb}_n\in\V^{\B_{\gamma_n}}$, such that for every $n<\omega$, $c_n$ forces the following statements:
\begin{enumerate}
\item $\dot{\sigma}_n:\check{\bN}\prec\check{N}$,
\item $\dot{\sigma}_n(\check{\bS})=\check{S}$, and for all $k<n$, $\dot{\sigma}_n(\dot{\bar{b}}_k)=\dot{\sigma}_k(\dot{\bar{b}}_k)$,
\item $\dot{\sigma}_n(\dot{\bar{b}}_n)=\dot{b}_n$ and $\dot{\bar{b}}_n\in\check{\bar{X}}$ (and so, $\dot{b}_n\in\check{X}$),
\item $\check{h}_{\gamma_n}(\dot{b}_n)\in\dot{G}_{\B_{\gamma_{n}}}$,
\item $\dot{\sigma}_n^{-1}``\dot{G}_{\B_{\gamma_n}}$ is $\check{\bN}$-generic for $\check{\bar{\B}}_{\gamma_n}$,
\item for some $t\in T|\Omega$ with $t\not <s_{n-1}$, $\dot{b}_n$ forces wrt.~$\B_\delta$ that $\check{t}\in\dot{B}$, and in particular, $\dot{b}_n$ forces that $\check{s}_{n-1}\notin\dot{B}$ (for $n>0$),
\item $c_{n-1}=h_{\gamma_{n-1}}(c_{n})$ (for $n>0$),
\item $\dot{b}_n\le_{\check{\B}_\delta}\dot{b}_{n-1}$ (for $n>0$),
%\item $C^N_{\delta(\B_{\gamma_{n}})}(\ran(\dot{\sigma}_{n-1}))=C^N_{\delta(\B_{\gamma_{n}})}(\ran(\dot{\sigma}_{n}))$ (for $n>0$).
\end{enumerate}
To start off, in the case $n=0$, we set $c_0=\eins$, $\dot{\sigma}_0=\check{\sigma}_0$ and $\dot{b}_0=\check{a}_0$ and $\dot{\barb}_0=\sigma_0^{-1}(\dot{b}_0)$.

Now suppose $m=n-1$, and $\dot{\sigma}_l$, $c_l$, $\dot{b}_l$ and $\dot{\barb}_l$ have been defined, so that (1)-(8) are satisfied
for $l\le m$. An application of Lemma \ref{lem:SubpropernessExtensionLemma} (to $\B_{\gamma_m}\sub\B_{\gamma_n}$) yields $c_n\in\B_{\gamma_n}$ and a $\B_{\gamma_n}$-name $\dot{\sigma}_n$ such that (1)-(8) are satisfied at level $n$, as long as they don't refer to $\dot{b}_n$ and $\dot{\bar{b}}_n$, as before. As before, we can also arrange that $c_n$ forces that $\check{h}_{\gamma_n}(\dot{b}_m)\in\dot{G}_{\gamma_n}$.

To define $\dot{b}_n$ and $\dot{\barb}_n$, let $I$ be $\B_{\gamma_n}$-generic with $c_n\in I$, and let $G=I\cap\B_{\gamma_m}$. Let $\sigma_n=\dot{\sigma}_n^I$, $b_m=\dot{b}_m^G$, $M_n=\ran(\sigma_n)$, $\bI=\sigma_n^{-1}``I$, $\sigma_n^*:\bN[\bI]\prec M_n[I]$ with $\sigma_n\sub\sigma_n^*$ and $\sigma_n^*(\bI)=I$. We have that $h_{\gamma_n}(b_m)\in I$.

Working in $\V[I]$, we know that $b_m/I$ forces wrt.~$\B_\delta/I$ that $\dot{B}/I$ is a new cofinal branch in $\check{T}$ (that is, a branch that did not exist in $\V[I]$). Consider the set of $z\in\B_{\gamma_n}$ such that there exist $t_1,t_2\in T$ and $y_1,y_2\le_\delta b_m$ such that
\begin{itemize}
  \item $t_1,t_2$ are incomparable in $T$,
  \item $z=h_{\gamma_n}(y_1)=h_{\gamma_n}(y_2)$,
  \item $y_1\forces_{\B_\delta}\check{t}_1\in\dot{B}$,
  \item $y_2\forces_{\B_\delta}\check{t}_2\in\dot{B}$.
\end{itemize}

Then $D\in M_n$, and $D$ is dense in $\B_{\gamma_n}$ below $h_{\gamma_n}(b_m)$. So there is a $z\in I\cap D$, and this is witnessed by some $t_1,t_2,y_1,y_2$. Since $M_n[I]\prec N[I]$, these objects may be chosen in $M_n[I]$. Thus, $t_1,t_2\in T|\Omega$. Since $t_1$ and $t_2$ are incomparable in $T$, at most one of them can be below $s_m$. Say $i<2$ is such that $t_i$ is not below $s_m$. Then we can set $b_n=y_i$.

Since all of this holds in $\V[I]$ whenever $I$ is $\B_{\gamma_n}$-generic over $\V$ and $c_n\in I$, there is a $\B_{\gamma_n}$-name $\dot{b}_n$ for $b_n$ such that $c_n$ forces all of this wrt.~$\B_{\gamma_n}$. Then (1)-(8) are satisfied. This finishes the recursive construction.

As before, the sequence $\seq{c_n}{n<\omega}$ is a thread, and so, $c=\bigwedge_{n<\omega}c_n\in\B_\delta^+$. It follows that $c\forces\dot{b}_n\in\dot{G}_\delta$, for all $n<\omega$. In particular, $c\le a_0$. We claim that $c$ forces that $\dot{B}$ is bounded in $\check{T}$. To see this, let $G$ be $\B_\delta$-generic over $\V$, with $c\in G$, and let $B=\dot{B}^G$. Let $n<\omega$. Since $b_{n+1}:=\dot{b}_{n+1}^G\in G$, we have by (6) that $s_n\notin B$. Since this holds for all $n<\omega$, $B$ contains no node at level $\Omega$ of $T$. So $B$ is bounded in $T$. This is a contradiction.

\noindent{\bf Case 2:} for all $i<\delta$, $\cf(\delta)>\delta(\B_i)$.

Exactly as in the proof of Theorem \ref{thm:RCSIteratingSubproperSouslinTreePreservingForcing}, we may reduce to the case that $\delta>\omega_1$ is a regular cardinal. It follows that $\B_\delta$ is ${<}\delta$-c.c. Thus, if $G_h$ is $\B_h$-generic and $\B_\delta/G_h$ added a new cofinal branch to some $\omega_1$-tree $T\in\V[G_h]$, then already some earlier $\B_\gamma/G_h$ would add this branch, contradicting our inductive hypothesis.
\end{proof}

\begin{thm}
\label{thm:IteratingSubproper[T]-preservingForcing}
Let $T$ be an $\omega_1$-tree. Then the class of subproper forcing notions that are $[T]$-preserving is standard RCS iterable.
\end{thm}

\begin{proof}
We can simplify the argument of the proof of Theorem \ref{thm:RCSIteratingSubproperForcingNotAddingBranchesToOmega1Trees}, as before.
\end{proof}

Again, we can easily modify the arguments of the proofs of the previous two theorems to obtain results about nice iterability. As before, the previous theorems show that every standard iteration of subproper and branch preserving/$[T]$-preserving forcing notions is nicely subproper and branch preserving/$[T]$-preserving, so that the following theorem is a generalization of these theorems.

\begin{thm}
\label{thm:NiceIterabilityOfSubPNotAddingBranches/[T]-presForcing}
The class of subproper and branch preserving forcing notions is nicely iterable, and so is the class of subproper forcing notions that are $[T]$-preserving, for some fixed $\omega_1$-tree $T$.
\end{thm}

\subsection{RCS and nicely subproper iterations of $\oo$-bounding forcing}
\label{subsec:NicelySubPiterationsOfOOboundingForcing}

We will now follow Abraham's handbook article \cite{Abraham:ProperForcing}, where this is carried out for countable support iterations of proper forcing notions, in showing that the stages of certain RCS iterations of $\oo$-bounding and subproper forcing notions are also subproper and $\oo$-bounding.

For $f,g\in\oo$ and $n\in\omega$, we write $f\le_n g$ if for all $k\ge n$, $f(k)\le g(k)$, and we write $f\le^*g$ if there is an $n\in\omega$ such that $f\le_ng$ (and we express this by saying that $g$ eventually dominates $f$). A forcing notion $\P$ is \emph{$\oo$-bounding} if whenever $G$ is $\P$-generic over $\V$ and $f\in(\oo)^{\V[G]}$, then there is a $g\in\V$ such that $f\le^*g$, and in fact, in this case, there is then a $g\in\V$ such that $f\le_0g$. If $\dot{f}$ is a $\P$-name for a real, then a weakly decreasing sequence $\vp=\seq{p_i}{i<\omega}$ of conditions in $\P$ \emph{interprets} $\dot{f}$ if there is an $f^*\in\oo$ such that for every $n<\omega$, $p_n\forces_\P\dot{f}\rest\check{n}=\check{f}^*\rest\check{n}$. In this case, we say that $\vp$ interprets $\dot{f}$ as $f^*$, and we write $f^*=\interpretation(\dot{f},\vp)$ to express this. If $g\in\oo$ and $\dot{f}$ is a $\P$-name for a real, then we say that a weakly decreasing sequence $\vp$ in $\P$ \emph{interprets $\dot{f}$ and respects $g$} if $\interpretation(\dot{f},\vp)\le_0g$.

\begin{thm}[{\cite[Theorem 3.2]{Abraham:ProperForcing}}]
Let $\P$ be an $\oo$-bounding forcing notion. Let $\dot{f}$ be a $\P$-name for a real, let $\kappa$ be a sufficiently large cardinal, and let $H_\kappa\sub N\models\ZFCm$, where $N$ is transitive.
Let $M\prec N$ be countable, with $\P,\dot{f}\in M$. Suppose that $g\in\oo$ eventually dominates all reals in $M$ and $\vp\in M$ interprets $\dot{f}$ and respects $g$. Then there are a condition $p\in\P\cap M$ and a real $h\in\oo\cap M$ such that $h\le_0g$ and $p\forces_\P\dot{f}\le_0\check{h}$. In particular, $p\forces_\P\dot{f}\le_0\check{g}$. Moreover, for any $n<\omega$, there is such a condition $p$ with $p\le p_n$.
\end{thm}

We will use the concept of a \emph{derived sequence}. Suppose that $\bbA\sub\B$ are complete Boolean algebras, and $h=h_{\B,\bbA}$. In this situation, let $\dot{f}$ be a $\B$-name for a real, and let $\vb$ be a sequence in $\B^+$ that interprets $\dot{f}$, i.e., such that $\interpretation(\dot{f},\vb)$ exists. Let us fix a well-order of $\B$, and let $G$ be $\bbA$-generic over $\V$. Recall that $b/G\neq 0$ iff $h(b)\in G$. Define the derived sequence $\va$ recursively, as follows: if $h(b_i)\in G$, then $a_i=b_i$. Note that in this case, $h(b_j)\in G$ for all $j<i$, so $a_j=b_j$ for all $j\le i$. If not, then let $a_i$ be the least element of $\B^+$ (with respect to the fixed well-order) such that $h(a_i)\in G$, $a_i\le a_{i-1}$ and $a_i$ decides $\dot{f}\rest\check{i}$. We write $\delta_G(\vb,\dot{f})$ for the derived sequence, and $\delta_G(\vb,\dot{f})/G$ for the sequence $\seq{a_i/G}{i<\omega}$. Note that by construction, $\delta_G(\vb,\dot{f})/G$ is a weakly decreasing sequence in $(\B/G)^+$. The following lemma on derived sequences is completely general and has nothing to do with subproperness vs.~properness.

\begin{lem}[{\cite[Lemma 3.3]{Abraham:ProperForcing}}]
\label{lem:DerivedSequences}
Let $\bbA\sub\B$ be complete Boolean algebras, where $\bbA$ is $\oo$-bounding, and let $h=h_{\B,\bbA}$. Let $\dot{f}\in\V^\B$ be a name for a real, $b\in\B^+$, and suppose that:
\begin{enumerate}
  \item $\vb$ is a weakly decreasing sequence below $b$ in $\B^+$ that interprets $\dot{f}$,
  \item $M\prec N$ is countable, where for some sufficiently large cardinal $\kappa$, $H_\kappa\sub N\models\ZFCm$, $N$ transitive, with $\bbA,\B,\dot{f},\vb,b\in M$,
  \item $g\in\oo$ bounds the reals of $M$ and $\interpretation(\dot{f},\vb)\le_0g$.
\end{enumerate}
Then there is a condition $0\neq a\le h(b)$ with $a\in M$ such that $a$ forces with respect to $\bbA$ that if $\va=\delta_{\dot{G}_\bbA}(\check{\vb},\check{\dot{f}})$, then $\va$ is below $\check{b}$, interprets $\check{\dot{f}}$ and respects $\check{g}$.
\end{lem}

The following lemma is a version of \cite[Lemma 1.1]{Abraham:ProperForcing} for complete Boolean algebras.

\begin{lem}
\label{lem:DecresingSequences}
Let $\bbA\sub\B$ be complete Boolean algebras, and let $G_0$ be $\bbA$-generic over $\V$. Suppose that in $\V[G_0]$, there is a sequence $\seq{r_i}{i<\omega}$ such that $r_i\in \B$ and $r_{i+1}/G_0\le_{\B/G_0}r_i/G_0$, for all $i<\omega$. Then, in $\V[G_0]$, there is a sequence $\seq{g_i}{i<\omega}$ such that for all $i<\omega$, $g_i\in G=\{b\in\B\st\exists a\in G_0\ a\le b\}$ and, letting $s_i=r_i\land g_i$, for $i<\omega$, it follows that $\vec{s}$ is weakly decreasing in $\B$.
\end{lem}

\begin{proof}
Recall that in this situation, $G$ is a $\V$-complete filter in $\B$, and that the quotient $\B/G_0=\B/G$ consists of equivalence classes with respect to the equivalence relation that identifies $b_0,b_1\in\B$ iff $(b_0\Rightarrow b_1)\land(b_1\Rightarrow b_0)\in G$ (where $b_0\Rightarrow b_1=\neg b_0\lor b_1$). $b/G_0$ denotes the equivalence class of $b$ with respect to this equivalence relation. The partial order on $\B/G_0$ is defined by $b_0/G_0\le_{\B/G_0}b_1/G_0$ iff $(b_0\Rightarrow b_1)\in G$.

Now let $g_i=\bigwedge_{j<i}(r_{j+1}\Rightarrow r_j)$ and $s_i=\bigwedge_{j\le i}r_j$. Then clearly, each $g_i$ is in $G$, and the sequence $\vec{s}$ is weakly decreasing in $\B$. It follows by induction on $i$ that $r_i\wedge g_i=s_i$, completing the proof. This is clear if $i=0$, since $s_0=r_0$ and $g_0=1$. For the inductive step we get that $r_{i+1}\wedge g_{i+1}=r_{i+1}\wedge(r_{i+1}\Rightarrow r_i)\wedge g_i=r_{i+1}\wedge r_i\wedge g_i$, which inductively is $r_{i+1}\wedge s_i=s_{i+1}$.
\end{proof}

The following lemma is the version of \cite[Lemma 3.4]{Abraham:ProperForcing} for subproper forcing with complete Boolean algebras.

\begin{lem}
\label{lem:oosubproperonesteplemma}
Let $\B_0\sub\B_1\sub\B_2$ be complete Boolean algebras. Note that $h_{\B_1,\B_0}=h_{\B_2,\B_0}\rest\B_1$. So let us write $h_0$ for $h_{\B_2,\B_0}$ and $h_1$ for $h_{\B_2,\B_1}$.

Suppose that $\B_0$ is subproper and $\oo$-bounding, and that $\forces_{\B_0}$ ``$\check{\B}_1/\dot{G}_{\B_0}$ is subproper and $\oo$-bounding.'' Let $\dot{f}\in\V^{\B_2}$ be a name for a real, and let $\delta=\delta(\B_1)$.

Let $N=L_\tau^A$ be such that $H_\theta\sub N$, where $\theta$ is sufficiently large, and fix $s\in N$. Let $a\in\B^+_0$, $\btheta, \bar{\B}_0,\bar{\B}_1,\bar{\B}_2, \bar{s}\in\bN$, where $\bN$ is countable, transitive and full, and $\bar{\B}_0\sub\bar{\B}_1\sub\bar{\B}_2$ are complete Boolean algebras in $\bN$.

Let $g\in\oo$ bound all the reals in $\bN$.

Let $\bS=\kla{\btheta,\bar{\B}_0,\bar{\B}_1,\bar{\B}_2,\dot{\barf},\bar{s},\blambda_1,\ldots,\blambda_n}$ and $S=\kla{\theta,\B_0,\B_1,\B_2,\dot{f},s,\lambda_1,\ldots,\lambda_n}$, where $\lambda_i>\delta$ is regular. Let $\dot{\sigma}_0$, $\dot{t}$, $\dot{\bar{b}}$, $\dot{b}$ be $\B_0$-names and suppose that $a$ forces:
\begin{enumerate}
\item[\textnormal{(A1)}] $\dot{\sigma}_0:\check{\bN}\prec\check{N}$,
\item[\textnormal{(A2)}] $\dot{\sigma}_0(\check{\bS})=\check{S}$,
\item[\textnormal{(A3)}] $\dot{\sigma}_0(\dot{\bar{b}})=\dot{b}$ and $\dot{\bar{b}}\in\check{\bar{\B}}_2$ (and so, $\dot{b}\in\check{\B}_2$),
\item[\textnormal{(A4)}] $\check{h}_0(\dot{b})\in\dot{G}_{\B_0}$,
\item[\textnormal{(A5)}] $\dot{t}\in\check{\bN}$,
\item[\textnormal{(A6)}] $\dot{\sigma}_0^{-1}``\dot{G}_{\B_0}$ is $\check{\bN}$-generic for $\check{\bar{\B}}_0$,
\item[\textnormal{(A7)}] Letting $M=\ran(\dot{\sigma}_0)$, there is in $M[\dot{G}_{\B_0}]$ a decreasing sequence $\vb$ of conditions in $\B_2^+$, below $\dot{b}$, such that $\vb/\dot{G}_{\B_0}$ is decreasing in $(\B_2/\dot{G}_{\B_0})^+$, and such that $\interpretation(\check{\dot{f}},\vb)\le_0\check{g}$.
\end{enumerate}

Then there are a condition $c\in\B^+_1$ with $h_0(c)=a$ and a $\B_1$-name $\dot{\sigma}_1$ such that whenever $I$ is $\B_1$-generic with $c\in I$, letting $\sigma_1=\dot{\sigma}_1^I$, $G=I\cap\B_0$ and $\sigma_0=\dot{\sigma}_0^G$, the following conditions hold:
\begin{enumerate}[label=\textnormal{(C\arabic*)}]
\item $\sigma_1:\bN\prec N$,
\item $\sigma_1(\bS)=S$,
\item $\sigma_1(\dot{\bar{b}}^G)=\sigma_0(\dot{\bar{b}}^G)$ and $\sigma_1(\dot{t}^G)=\sigma_0(\dot{t}^G)$,
\item $h_1(\dot{b}^G)\in I$,
\item ${\sigma_1}^{-1}``I$ is $\bar{\B}_1$-generic over $\bar{N}$,
\item Letting $\blambda_0=\On\cap\bN$, we have that for $i\le n$, $\sup\sigma_0``\blambda_i=\sup\sigma_1``\blambda_i$.
%$C^N_\delta(\ran(\sigma_1))=C^N_\delta(\ran(\sigma_0))$,
\item Letting $M=\ran(\sigma_1)$, there is in $M[I]$ a decreasing sequence $\vc$ of conditions in $\B_2^+$ below $\dot{b}^G$ such that $\vc/I$ is decreasing in $(\B_2/I)^+$ and $\interpretation(\dot{f},\vc)\le_0g$.
\end{enumerate}
\end{lem}

\begin{proof}
Let $G_0$ be $\B_0$-generic over $\V$, with $a\in G_0$. Let $\sigma_0=\dot{\sigma}_0^{G_0}$, $M=\ran(\sigma_0)$, $b=\dot{b}^{G_0}$, etc. Then $\sigma_0:\bN\prec M$, and $\sigma_0$ lifts to $\sigma_0^*:\bN[\bG]\prec M[G_0]$, where $\bG=\sigma_0^{-1}``G_0$. Since $M[G_0]\prec N[G_0]$, we will be able to apply Lemma \ref{lem:DerivedSequences} in $\V[G_0]$.

Note that $\oo\cap M[G_0]=\oo\cap\bN[\bG]$. Since $\B_0$ is $\oo$-bounding, $\bN$ thinks that $\bar{\B}_0$ is $\oo$-bounding, and since $\bG$ is $\bar{\B}_0$-generic over $\bN$, it follows that every real of $\bN[\bG]$ is bounded by some real of $\bN$, which is bounded by $g$. Thus, the reals of $\bN[\bG]$ are bounded by $g$, and hence, the reals of $M[G_0]$ are bounded by $g$.

The $\B_2$-name $\dot{f}$ can be viewed as a $\B_2/G_0$-name in the obvious way. Let's write $\dot{f}/G_0$ for this name. Then $\dot{f}/G_0$ is a $\B_2/G_0$-name for a real.

By assumption (A7), let $\vb$ be a decreasing sequence of conditions in $\B_2$, below $b$, such that $\vb/G_0$ is decreasing in $\B_2/G_0$, and such that $\interpretation(\dot{f},\vb)\le_0g$. This means that $\interpretation(\dot{f}/G_0,\vb/G_0)\le_0g$.

Thus, by Lemma \ref{lem:DerivedSequences}, applied to $\bbA=\B_1/G_0$, $\B=\B_2/G_0$, the name $\dot{f}/G_0$, the model $M[G_0]$, the condition $b/G_0$ and the sequence $\vb/G_0$, there is a condition $d\in\B^+_1\cap M[G_0]$ such that $d/G_0\le h_{\B_2/G_0,\B_1/G_0}(b/G_0)$ and such that $d/G_0$ forces with respect to $\B_1/G_0$ that if $\va=\delta_{\dot{G}_{\B_1/G_0}}(\vb/G_0,\dot{f}/G_0)$, then $\vec{a}$ is below $\check{b}/G_0$ in $(\B_2/G_0)^+$, interprets $(\dot{f}/G_0)\check{}$ and respects $g$.

Let $\dot{d}$ be a $\B_0$-name for $d$ such that $a$ forces (with respect to $\B_0$) that all of these properties hold. Let $\dot{\bar{d}}$ be a $\B_0$-name such that $a$ forces that $\dot{\sigma}_0^*(\dot{\bar{d}})=\dot{d}$.

We can now apply Lemma \ref{lem:SubpropernessExtensionLemma} to $\B_0\sub\B_1$, $\dot{\sigma}_0$, $\dot{\bard}$, $\dot{d}$ getting a condition $c\in\B_1$ with $h_0(c)=a$ and a $\B_1$-name $\dot{\sigma}_1$ such that whenever $I$ is $\B_1$-generic over $\V$ with $c\in I$ and $G=I\cap\B_0$, it follows that
\begin{enumerate}
  \item $\sigma_1:\bN\prec N$,
  \item $\sigma_1(\bS)=S$,
  \item $\sigma_1(\dot{\barb}^G)=\sigma_0(\dot{\barb}^G)$ and $\sigma_1(\dot{t}^G)=\sigma_0(\dot{t}^G)$,
  \item $\sigma_1(\dot{\bard}^G)=\sigma_0(\dot{\bard}^G)=\dot{d}^G$,
  \item $\dot{d}^G\in I$,
  \item $\sigma_1^{-1}``I$ is $\bar{\B}$-generic over $\bN$,
  \item for all $i\le n$, $\sup\sigma_0``\blambda_i=\sup\sigma_1``\blambda_i$.
\end{enumerate}
where $\delta=\delta(\B_1)$. To conclude the proof, we verify that we also have: $h_1(\dot{b}^G)\in I$, and
letting $M=\ran(\sigma_1)$, there is in $M[I]$ a decreasing sequence $\vc$ of conditions in $\B_2$ below $\dot{b}^G$ such that $\vc/I$ is decreasing in $\B_2/I$ and $\interpretation(\dot{f},\vc)\le_0g$.

In the present situation, we have $I\sub\B_1$ is generic over $\V$, $G=I\cap\B_0$ is $\B_0$-generic over $\V$, and $I/G=\{x/G\st x\in I\}$ is $\B_1/G$-generic over $\V[G]$. Moreover, $I=G*(I/G)=\{x\in\B_1\st x/G\in I/G\}$.

Let $b=\dot{b}^G\in\B_2$, $b'=h_1(b)$ and $d=\dot{d}^G\in\B_1$. We have that $d/G\le h_{\B_2/G_0,\B_1/G_0}(b)$, since this was forced by $a$. But $h_{\B_2/G,\B_1/G}(b)=h_1(b)/G$ (see \cite[Prop.~4.9]{VialeEtAl:BooleanApproachToSPiterations}). Thus, $d/G\le b'/G$. Since $d\in I$, we have that $d/G\in I/G$, and so, $b'/G\in I/G$. But since $I=\{x\in\B_1\st x/G\in I/G\}$, this means that $b'=h_1(b)\in I$, as wished.

Further, what $d/G$ forces over $\V[G]$ is true in $\V[G][I/G]=\V[I]$. Thus, let $\va=\delta_{I/G}(\vb/G,\dot{f}/G)$.  Then $\vec{a}$ is below $b/G$ in $(\B_2/G)^+$, interprets $\dot{f}/G$ and respects $g$. Note that $\va\in M[I]$, and $\va/(I/G)$ is weakly decreasing in $((\B_2/G)/(I/G))^+$, which we can identify with $\B_2/I$.

For every $i<\omega$, let $a'_i\in\B_2$ be such that $a_i=a'_i/G$. We then have that for every $i<\omega$, $a_i/(I/G)=(a'_i/G)/(I/G)$, which we may identify with $a'_i/I$. Thus, $\seq{a'_i}{i<\omega}$ is a sequence in $\B_2^+$ such that for all $i<\omega$, $a'_{i+1}/I\le_{\B_2/I}a'_i/I$, and $a'_i/I\neq 0$. By replacing $a'_i$ with $a'_i\land b$ if necessary, we may assume that $a'_i\le b$, for all $i$ -- note that $a_i=a'_i/G=(a'_i\land b)/G$ since $a_i/G\le b/G$.

By Lemma \ref{lem:DecresingSequences}, there is a sequence $\seq{q_i}{i<\omega}$ in $\V[I]$ such that each $q_i\in I^*=\{x\in\B_2\st\exists y\in I\ y\le x\}$, and such that if we let $c_i=a'_i\land q_i$, then $\vec{c}$ is weakly decreasing in $\B_2$. Since $a'_i/I\neq 0$, we know that $h_{\B_2,\B_1}(a'_i)\in I$. If $q_i\ge q'_i\in I$, then we have that $h_{\B_2,\B_1}(a'_i\land q'_i)=h_{\B_2,\B_1}(a'_i)\land q'_i\in I$ (we used here \cite[P.~87, second bullet point]{Jensen2014:SubcompleteAndLForcingSingapore} and the fact that $q'_i\in\B_1$), and so, $(a'_i\land q'_i)/I\neq 0$. Since $q_i\ge q'_i$, it follows that $c_i/I\neq 0$. And since $c_i\le a'_i$ for all $i$, it follows that $\vec{c}$ interprets $\dot{f}$ and respects $g$, and $\vec{c}$ is below $b$, as wished.
\end{proof}

We can now follow the proof templates of Theorems \ref{thm:RCSIteratingSubproperSouslinTreePreservingForcing} and
\ref{thm:RCSIteratingSubproperForcingNotAddingBranchesToOmega1Trees} to obtain our iteration theorem for $\oo$-bounding subproper forcing.

\begin{thm}
\label{thm:RCSIterationOfSubproperOO-boundingForcing}
The class of subproper $\oo$-bounding Boolean algebras is standard RCS iterable.
%Let $\seq{\B_i}{i<\alpha}$ be an RCS iteration such that for all $i+1<\alpha$,
%\begin{enumerate}
%  \item $\B_i\neq\B_{i+1}$,
%  \item $\forces_{\B_i} (\check{\B}_{i+1}/\dot{G}_{\B_i}\ \text{is subproper and $\oo$-bounding})$,
%  \item $\forces_{\B_{i+1}} (\delta(\check{\B}_i)\ \text{has cardinality at most}\ \omega_1)$.
%\end{enumerate}
%Then for all $h\le\delta<\alpha$, if $G_h$ is $\B_h$-generic, then in $\V[G_h]$, $\B_\delta/G_h$ is subproper and $\oo$-bounding.
\end{thm}

\begin{proof}
Let $\vec{\B}=\seq{\B_i}{i<\alpha}$ be a standard RCS iteration of subproper and $\oo$-bounding complete Boolean algebras. We prove by induction on $\delta$: whenever $h\le\delta<\alpha$ and $G_h$ is $\B_h$-generic, then in $\V[G_h]$, $\B_\delta/G_h$ is subproper and $\oo$-bounding.

As before, it suffices to focus on the $\oo$-bounding property, and it suffices to focus on the case that $\delta$ is a limit ordinal.

\noindent{\bf Case 1:} there is an $i<\delta$ such that $\cf(\delta)\le\delta(\B_i)$.

Let $i<\delta$ have this property. As before, it suffices to prove:

\begin{enumerate}
\item[(A)] if $i<j<\delta$ and $G_j$ is $\B_j$-generic, then in $\V[G_j]$, $\B_\delta/G_j$ is $\oo$-bounding.
\end{enumerate}

This further reduces to showing

\begin{enumerate}
\item[(B)] if $\cf(\delta)\le\omega_1$ then $\B_\delta$ is $\oo$-bounding.
\end{enumerate}

As before, we can define a dense set $X\sub\B_\delta$, depending on the cofinality of $\delta$: if $\cf(\delta)=\omega_1$, then $X=\bigcup_{i<\delta}\B_i$, and if $\cf(\delta)=\omega_1$, then
$\{\bigwedge_{i<\delta}t_i\st\seq{t_i}{i<\delta}\ \text{is a thread in}\ \vBB\rest\delta\}$.

Let $\pi:\omega_1\To\delta$ be cofinal, with $\pi(0)=0$.

Let $\dot{f}$ be a $\B_\delta$-name for a member of $\oo$, and let $a_0\in\B_\delta$ be a condition. We have to find $g\in\oo$ and a condition extending $a_0$ that forces that $\dot{f}\le_0\check{g}$.
Since $X$ is dense in $\B_\delta$, we may assume that $a_0\in X$.

Let $N=L_\tau^A$, with $H_\theta\sub N$, such that $\theta$ verifies the subproperness of each $\B_i$, for $i\le\delta$.
Let $S=\kla{\theta,\delta,\vec{\B},X,\dot{f},\pi,a_0}$.
Let $M_0\prec N$ with $S\in M_0$, $M_0$ countable and such that, letting $\sigma_0:\bN\To M_0$ be the inverse of the Mostowski collapse (so that $\bN$ is transitive), $\bN$ is full. Let $\sigma_0^{-1}(S)=\bS=\kla{\btheta,\bdelta,\vec{\bar{\B}},\bX,\dot{\barf},\bpi,\bar{a}_0}$.
Let $\seq{\nu_i}{i<\omega}$ be a sequence of ordinals $\nu_i<\omega_1^\bN$ such that if we let $\bgamma_i=\bpi(\nu_i)$, it follows that $\seq{\bgamma_i}{i<\omega}$ is monotone and cofinal in $\bdelta$, and such that $\nu_0=0$, so that $\bgamma_0=0$. Hence, letting $\gamma_i=\sigma_0(\bgamma_i)$, we have that $\sup_{i<\omega}\gamma_i=\sup(M_0\cap\delta)$. Moreover, whenever $\sigma':\bN\prec N$ is such that $\sigma'(\bpi)=\pi$, it follows that $\sigma'(\bgamma_i)=\gamma_i=\pi(\nu_i)$, as before.

Working inside $M_0$, construct a decreasing sequence $\seq{a_i}{i<\omega}$ % in $X$
that interprets $\dot{f}$ as some $f^*\in M_0$ (so $\va$ starts with the given condition fixed above). Let $g$ bound the reals of $M_0$, with $f^*\le_0g$.

By induction on $n<\omega$, construct sequences $\seq{\dot{\sigma}_n}{n<\omega}$, $\seq{c_n}{n<\omega}$, $\seq{\dot{b}_n}{n<\omega}$ and $\seq{\dot{\bar{b}}_n}{n<\omega}$ with $c_n\in\B_{\gamma_n}$, $\dot{\sigma}_n\in\V^{\B_n}$, $\dot{b}_n,\dot{\barb}_n\in\V^{\B_{\gamma_n}}$, such that for every $n<\omega$, $c_n$ forces the following statements:
\begin{enumerate}
\item $\dot{\sigma}_n:\check{\bN}\prec\check{N}$.
\item $\dot{\sigma}_n(\check{\bS})=\check{S}$, and for all $k<n$, $\dot{\sigma}_n(\dot{\bar{b}}_k)=\dot{\sigma}_k(\dot{\bar{b}}_k)$.
\item $\dot{\sigma}_n(\dot{\bar{b}}_n)=\dot{b}_n$ and $\dot{\bar{b}}_n\in\check{\bar{\B}}_{\bar{\delta}}$ (and so, $\dot{b}_n\in\check{\B}_\delta$). Moreover, $\dot{\bar{b}}_n\in\check{\bX}$, so $\dot{b}_n\in\check{X}$.
\item $\check{h}_{\gamma_n}(\dot{b}_n)\in\dot{G}_{\B_{\gamma_{n}}}$.
%\item $\dot{t}\in\check{\bN}$,
\item $\dot{\sigma}_n^{-1}``\dot{G}_{\B_{\gamma_n}}$ is $\check{\bN}$-generic for $\check{\bar{\B}}_{\gamma_n}$.
\item Letting $M_n=\ran(\dot{\sigma}_n)$, there is in $M_n[\dot{G}_{\B_{\gamma_n}}]$ a decreasing sequence of conditions in $\check{\B}_\delta$, below $\dot{b}_n$, such that $\vb/\dot{G}_{\B_\delta}$ is decreasing in $(\check{\B}_{\delta}/\dot{G}_{\B_{\gamma_n}})^+$, and such that $\interpretation(\check{\dot{f}},\vb)\le_0\check{g}$. %Moreover, each member of this decreasing sequence is in $\check{X}$.
\item $\dot{b}_n$ decides $\dot{f}\rest\check{n}$ (with respect to $\B_\delta$), and $\dot{b}_n\forces_{\check{\B}_\delta}\dot{f}\rest\check{n}\le_0\check{g}\rest\check{n}$.
\item $c_{n-1}=h_{\gamma_{n-1}}(c_{n})$ (for $n>0$).
\item $\dot{b}_n\le_{\check{\B}_\delta}\dot{b}_{n-1}$ (for $n>0$).
%\item $C^N_{\delta(\B_{\gamma_{n}})}(\ran(\dot{\sigma}_{n-1}))=C^N_{\delta(\B_{\gamma_{n}})}(\ran(\dot{\sigma}_{n}))$ (for $n>0$).
\end{enumerate}
To start off, in the case $n=0$, we set $c_0=\eins$, $\dot{\sigma}_0=\check{\sigma}_0$ and $\dot{b}_0=\check{a}_0$ and $\dot{\barb}_0=\sigma_0^{-1}(\dot{b}_0)$. Clearly then, (1)-(7) are satisfied for $n=0$ -- for (3) and (6), note that we picked $\vec{a}$ in $X$. Conditions (8)-(9), as well as the second part of (2), are vacuous for $n=0$.

Now suppose $m=n-1$, and $\dot{\sigma}_l$, $c_l$, $\dot{b}_l$ and $\dot{\barb}_l$ have been defined, so that (1)-(9) are satisfied
for $l\le m$.

We want to apply Lemma \ref{lem:oosubproperonesteplemma} in the present situation. Here is the conversion:

\begin{center}
\begin{tabular}{c|c}
Lemma \ref{lem:oosubproperonesteplemma} & Current situation\\
\hline
$\B_0,\B_1,\B_2$ & $\B_{\gamma_m},\B_{\gamma_n},\B_\delta$\\
$\dot{\sigma}_0$ & $\dot{\sigma}_m$\\
$\dot{t}$ & $\kla{\dot{\barb}_0,\ldots,\dot{\barb}_m}$\\
$\dot{\barb},\dot{b}$ & $\dot{\barb}_m,\dot{b}_m$\\
$a$ & $c_m$
\end{tabular}
\end{center}

The assumptions (1)-(7) stated in the lemma are then satisfied, by our inductive assumption, and the lemma then guarantees the existence of certain objects, which we convert to the current situation as follows:

\begin{center}
\begin{tabular}{c|c}
Lemma \ref{lem:oosubproperonesteplemma} & Current situation\\
\hline
$c$ & $c_n$\\
$\dot{\sigma}_1$ & $\dot{\sigma}_n$
\end{tabular}
\end{center}

We are thus left to define $\dot{b}_n$ and $\dot{\barb}_n$. To do this, let $I$ be $\B_{\gamma_n}$-generic over $\V$, and let $G=I\cap\B_{\gamma_m}$, with $c_n\in I$. Working in $\V[I]$, let $\sigma_n=\dot{\sigma}_n^I$, $\sigma_m=\dot{\sigma}_m^G$, $t=\dot{t}^G$, $b_k=\dot{b}_k^G$ and $\barb_k=\dot{\barb}_k^G$, for $k\le m$. Since $\dot{\sigma}_n$ and $c_n$ were chosen according to Lemma \ref{lem:oosubproperonesteplemma}, we then have that $\sigma_n:\bN\prec N$, $\sigma_n(\bS)=S$, $\sigma_n(\barb_k)=b_k$ for $k\le m$, $h_{\gamma_m}(b_m)\in I$ and $\bI=\sigma_n^{-1}``I$ is $\bar{\B}_{\bgamma_n}$-generic over $\bN$. %and $C^N_{\delta(\B_{\gamma_n})}(\ran(\sigma_m))=
%C^N_{\delta(\B_{\gamma_n})}(\ran(\sigma_n))$.
So let $M_n=\ran(\sigma_n)$, and let $\sigma^*_n:\bN[\bI]\prec M[I]$ be such that $\sigma_n\sub\sigma^*_n$ and $\sigma^*_n(\bI)=I$. By conclusion (7) of Lemma \ref{lem:oosubproperonesteplemma}, there is in $M_n[I]$ a sequence $\vr=\seq{r_i}{i<\omega}$ decreasing in $\B_\delta$, below $b_m$, such that $\vr/I$ is decreasing in $(\B_\delta/I)^+$ and $\interpretation(\dot{f},\vr)\le_0g$. %We can easily arrange that each $r_i$ is in $X$, since this is a dense subset of $\B_\delta$.
Let $\dot{b}_n$ be a $\B_{\gamma_n}$-name such that $\dot{b}_n^I=r_n$, and such that $c_n$ forces that $\dot{b}_n$ decides $\dot{f}\rest\check{n}$ and forces $\dot{f}\rest\check{n}\le_0\check{g}\rest\check{n}$. Finally, let $\dot{\barb}_n$ be a $\B_{\gamma_n}$-name for $(\sigma^*_n)^{-1}(\dot{b}_n)$.

This concludes the construction of the sequences $\vec{\dot{\sigma}}$, $\vec{c}$, $\vec{b}$ and $\vec{\bar{b}}$.

Now, the sequence $\seq{c_n}{c<\omega}$ is a thread, and it follows as before that $c=\bigwedge_{n<\omega}c_n\in\B_\delta^+$. %If $\cf(\delta)=\omega$, then this follows from part \ref{item:ConditionsAtCountableCofinality}\ref{item:ThreadsAreNonzero} of Definition \ref{def:NicelyGammaIteration}. And if $\cf(\delta)=\omega_1$, then $\gamma:=\sup_{n<\omega}\gamma_n<\delta$ and $\cf(\gamma)=\omega$, so again, $c\in\B_\gamma^+\sub\B_\delta^+$, for the same reason.

We claim that $c$ forces that $\dot{f}\le_0\check{g}$. To see this, let $G$ be $\B_\delta$-generic over $\V$, with $c\in G$. Let $n<\omega$. Let $b_n:=\dot{b}_n^G\in G$. We show as before:

\claim{(C)}{For all $n<\omega$, $b_n\in G$.}

Now by point (7) in our recursive construction, $b_n$ decides $\dot{f}\rest\check{n}$ to be totally bounded by $g\rest n$. This shows that $\dot{f}^G\le_0 g$.

\noindent{\bf Case 2:} for all $i<\delta$, $\cf(\delta)>\delta(\B_i)$.

It follows as before that $\B_\delta$ the direct limit of $\vBB\rest\delta$, and is ${<}\delta$-c.c.

But then it follows that if $G$ is $\B_\delta$-generic over $\V$, then any real in $\V[G]$ is already in $\V[G\cap\B_\gamma]$, for some $\gamma<\delta$, and hence is bounded by a real in $\V$, since we know inductively that $\B_\gamma$ is $\oo$-bounding.
\end{proof}

As before, the previous RCS iteration theorems imply:

\begin{obs}
\label{obs:StandardRCSIterationsOfSubPooboundingAreNicelySubPoobounding}
Every standard RCS iteration of subproper and $\oo$-bounding forcings is nicely subproper and $\oo$-bounding.
\end{obs}

And as before, we can generalize Theorem \ref{thm:RCSIterationOfSubproperOO-boundingForcing} as follows.

\begin{thm}
\label{thm:NicelyIteratingSubproperOobounding}
The class of subproper and $\oo$-bounding Boolean algebras is nicely iterable.
\end{thm}

\section{Nice Iterations}
\label{sec:NiceIterations}

In this section, we will adopt Miyamoto's nice iterations from \cite{Miyamoto:IteratingSemiproperPreorders}, and prove preservation theorems for generalizations of subproper and subcomplete forcing. Initially, we proved that subcomplete forcing can be iterated in this way, but then realized that we could drop a condition in the definition of subcompleteness. We then observed that the corresponding simplification works for the class of subproper forcing notions as well. We learned afterwards that Miyamoto \cite{Miyamoto:IteratingPreproperness} also made this latter observation for subproper forcing.

\subsection{Subcompleteness and $\infty$-subcompleteness}
\label{subsec:SCAndInftySC}

Let us start by giving the definition of subcompleteness, as well as its simplification. Recall the definitions of the density of a partial order (Definition \ref{def:DensityOfAPoset}) and fullness (Definition \ref{def:fullness}). In order to discuss these variations, we will present the definitions in the more general framework of Fuchs \cite{Fuchs:ParametricSubcompleteness}. We will work with the ``hull condition'' version of subcompleteness. In order to formulate it, we use the following terminology.

\begin{defn}
\label{def:Hulls}
	Let $N=L^A_\tau=\kla{L_\tau[A],\in,A\cap L_\tau[A]}$ be a \ZFCm{} model, $\eps$ an ordinal and $X\cup\{\eps\}\sub N$. Then $C^N_\eps(X)$ is the smallest $Y\prec N$ such that $X\cup\eps\sub Y$.
\end{defn}

\begin{defn}
	\label{def:epsilon-subcompleteness}
	Let $\eps$ be an ordinal.
	A forcing $\P$ is $\eps$-\emph{subcomplete} if there is a cardinal $\theta>\eps$ which \emph{verifies} the $\eps$-subcompleteness of $\P$, which means that $\P\in H_\theta$, and for any \ZFCm{} model $N=L_\tau^A$ with $\theta<\tau$ and $H_\theta\sub N$, any $\sigma:\bN\prec N$ such that $\bN$ is countable, transitive and full and such that $\P,\theta,\eps\in\ran(\sigma)$, any $\bar{G}\sub\bar{\P}$ which is $\bar{\P}$-generic over $\bN$, and any $\bar{s}\in\bN$, the following holds. Letting $\sigma(\kla{\bar{\eps},\bar{\theta},\bar{\P}})=\kla{\eps,\theta,\P}$, and setting $\bar{S}=\kla{\bar{s},\bar{\eps},\bar{\theta},\bar{\P}}$, there is a condition $p\in\P$ such that whenever $G\sub\P$ is $\P$-generic over $\V$ with $p\in G$, there is in $\V[G]$ a $\sigma'$ such that
	\begin{enumerate}
		\item $\sigma':\bN\prec N$,
		\item $\sigma'(\bar{S})=\sigma(\bar{S})$,
		\item $(\sigma')``\bar{G}\sub G$,
		\item
		\label{item:HullCondition}
		$C^N_{\eps}(\ran(\sigma'))=C^N_{\eps}(\ran(\sigma))$.
	\end{enumerate}
\end{defn}

In this parlance, $\P$ is subcomplete iff it is $\delta(\P)$-subcomplete. It is easy to see that increasing $\eps$ weakens the condition of being $\eps$-subcomplete. Thus, we refer to the version of subcompleteness obtained by dropping the hull condition \ref{item:HullCondition} as \emph{$\infty$-subcompleteness}. This makes sense if one interprets $\infty$ as $\On\cap N$ in Definition \ref{def:Hulls}. Since in our context, $N$ is a \ZFCm model of the form $L^A_\tau$, it follows then that $C^N_\infty(\leer)=N$, and hence that the hull condition is vacuous.

It was pointed out in \cite{Fuchs:ParametricSubcompleteness} that the hull condition is somewhat unnatural, because the density of a partial order is not preserved under forcing equivalence. It was shown there that the $\eps$-subcompleteness of a partial order, however, is preserved under forcing equivalence, and it is easy to see that the same is true of $\infty$-subcompleteness. Another flaw of the concept of subcompleteness that was addressed in \cite{Fuchs:ParametricSubcompleteness} is that it is unclear whether factors of subcomplete forcing notions are subcomplete. The result of \cite{Fuchs:ParametricSubcompleteness} that factors of $\eps$-subcomplete forcing notions are $\eps$-subcomplete carries over to $\infty$-subcompleteness; in fact, it simplifies slightly, since one does not need to worry about proving that the factor satisfies the hull condition. As a result the proof of this fact below is much simpler than the corresponding one in \cite{Fuchs:ParametricSubcompleteness}.

We introduced the concept of $\eps$-subcompleteness only in order to motivate $\infty$-subcompleteness; it is the latter class that concerns us here.

Just like in the case of subcompleteness, when verifying that some poset $\P$ us $\infty$-subcomplete, one may assume that some useful parameter $p$ belongs to the range of $\sigma$ (in the setup of Definition \ref{def:epsilon-subcompleteness}). For the case of subcompleteness, this was pointed out in \cite[P.~115-116, Lemma 2.5]{Jensen2014:SubcompleteAndLForcingSingapore}. The argument carries over to the case of $\infty$-subcompleteness (the reader may follow along the proof of \cite[Observation 2.27]{Fuchs:CanonicalFragmentsOfSRP}.) We will tacitly use this fact throughout.

\begin{thm}
	\label{thm:FactorsOfInftySCForcingIsSC}
	Let $\P$ be a poset, and let $\dot{\Q}$ be a $\P$-name for a poset, such that $\P*\dot{\Q}$ is $\infty$-subcomplete. Then $\P$ is $\infty$-subcomplete.
\end{thm}

\begin{proof}
Let $\theta$ be large enough to verify that $\mathbb P * \dot{\mathbb Q}$ is $\infty$-subcomplete. We claim that it is also large enough to verify that $\mathbb P$ is subcomplete. To see this, let $N = L_\tau[A]$, be a ZFC$^-$ model with $\tau > \theta$ regular, and $H_\theta \subseteq N$. Fix a parameter $s \in N$ and let $\sigma:\bN \prec N$ with $\bN$ countable, transitive and full so that $\sigma (\overline{\mathbb P}, \dot{\overline{\mathbb Q}}, \overline{\theta}, \overline{s}) = \mathbb P, \dot{\mathbb Q}, \theta, s$. Let $\barG \subseteq \overline{\mathbb P}$ be generic over $\bN$. Let $(p, \dot{q})$ be a condition witnessing the $\infty$-subcompleteness of $\mathbb P * \dot{\mathbb Q}$ and let $(p, \dot{q}) \in G * H$ be $\mathbb P * \dot{\mathbb Q}$-generic over $\V$. Work in $\V[G*H]$ and let $\sigma ':\bar{N} \prec N$ be an embedding so that $\sigma ' (\overline{\mathbb P}, \dot{\overline{\mathbb Q}}, \overline{\theta}, \overline{s}) = \mathbb P, \dot{\mathbb Q}, \theta, s$ and $\sigma ' ``\bar{G} \subseteq G$. Fixing an enumeration of $\bN$ in order type $\omega$, we can consider the tree $T_{G}$ of finite initial segments of an elementary embedding $\sigma_0  : \bN \prec N$ with $\sigma_0 (\overline{\mathbb P}, \dot{\overline{\mathbb Q}}, \overline{\theta}, \overline{s}) = \mathbb P, \dot{\mathbb Q}, \theta, s$ and so that $\sigma_0 ' `` \bar{G} \subseteq G$. Note that this tree is in fact in $\V[G]$. Moreover, in $\V[G*H]$ it's ill-founded since $\sigma '$ generates an infinite branch. But then by the absoluteness of ill-foundedness, $T_G$ is ill-founded in $\V[G]$. So there is an infinite branch in $\V[G]$ and this branch witnesses that $\mathbb P$ is $\infty$-subcomplete.
\end{proof}

The following observation underlines the simplicity and elegance of the concept of $\infty$-subcomplete forcing.

\begin{proposition}[Essentially Lemma 2.3 of \cite{Fuchs:ParametricSubcompleteness}]
$\infty$-subcomplete forcings are closed under forcing equivalence.
\end{proposition}

\begin{proof}
The argument of \cite[Lemma 2.3]{Fuchs:ParametricSubcompleteness} goes through. Let $\mathbb P$ be $\infty$-subcomplete and $\mathbb Q$ be forcing equivalent to $\mathbb P$. By this, we mean that $\P$ and $\Q$ have the same forcing extensions - this can be expressed in a first order way. To show that $\Q$ is $\infty$-subcomplete, let $\theta$ be large enough to verify the $\infty$-subcompleteness of $\mathbb P$ and assume $\mathcal P(\mathbb Q \cup \mathbb P) \in H_\theta$.
Let $\sigma:\bN\prec N, s, \bar{\Q},\bar{H}$, etc be as in the definition of $\infty$-subcompleteness, where $\bar{H}$ is $\bar{\Q}$-generic over $\bN$ and $\sigma(\bar{\Q})=\Q$. We may assume that $\P\in\ran(\sigma)$, and write $\bar{\P}=\sigma^{-1}(\P)$. By elementarity, $\bN$ believes that $\bP$ and $\bar{\Q}$ are forcing equivalent, and hence, there is a $\bar{G}$ which is $\bP$-generic over $\bN$, such that $\bN[\bG]=\bN[\bar{H}]$. Since $\P$ is $\infty$-subcomplete, there is a condition $p \in \mathbb P$ so that if $G\ni p$ is $\mathbb P$-generic over $V$ then in $V[G]$ there is an embedding $\sigma ' : \bN \prec N$ as in the definition of $\infty$-subcompleteness, so $\sigma'``\bG\sub G$. We may also assume that $\sigma'(\bar{\Q})=\Q$. $\sigma'$ then lifts to an embedding $\tsigma:\bN[\bG]\prec N[G]$ with $\tsigma(\bG)=G$. By elementarity of $\tsigma$, letting $H=\tsigma(\bar{H})$, it follows that $H$ is $\Q$-generic over $N$, and since $N$ contains all subsets of $\Q$, it follows that $H$ is $\Q$-generic over $\V$. Moreover, by elementarity of $\tsigma$, we see that $N[G]=N[H]$, so $G\in\V[H]$, so $\sigma'\in\V[H]$. And clearly, since $H=\tsigma(\bar{H})$, it follows that $\tsigma``\bH\sub H$, that is, $\sigma'``\bH\sub H$. Since $\sigma'\in\V[H]$, there is a condition $q\in H$ that forces the existence of an embedding like $\sigma'$, showing that $\Q$ is $\infty$-subcomplete.
\end{proof}

It is unclear if the corresponding fact is true for subcomplete forcing notions since there are forcing equivalent notions with radically different densities.

We do not know whether the iteration theorems for nicely subcomplete iterations, which give great leeway in how limit stages of the iteration are formed, can be carried out without some version of the hull or suprema condition. But if we use Miyamoto's method of forming limits in nice iterations, it turns out that we do not need any hull or suprema conditions. All other preservation properties of subcomplete forcing notions that we know of actually do not need the hull or suprema condition either, and thus are preservation properties of $\infty$-subcomplete forcing. We list some in the following observation.

\begin{obs}
Let $\P$ be $\infty$-subcomplete.
\begin{enumerate}
	\item $\P$ preserves stationary subsets of $\omega_1$.
	\item $\P$ preserves Souslin trees.
	\item $\P$ preserves the principle $\diamondsuit$.
	\item $\P$ does not add reals.
\end{enumerate}
\end{obs}
	
Apart from simplifying the theory, however, we do not have a particular use of the concept of $\infty$-subcompleteness. In fact, the following question is open:

\begin{question}
Is every $\infty$-subcomplete forcing also subcomplete?
\end{question}

However, we feel that simplifying a highly technical concept such as subcompleteness is worthwhile in its own right.

The same ``$\infty$"-modification made to subcomplete forcing notions can be made to subproper forcing notions as well. The proofs that $\infty$-subproperness is invariant under forcing equivalence and that factors of $\infty$-subproper forcing notions are $\infty$-subproper follow the proofs of the corresponding results given above, so we leave the details to the interested reader and just list this definition and fact. Note that the definition of $\infty$-subproperness results from dropping the suprema condition \ref{item:SupremaCondition} from Definition \ref{def:SubpropernessWithSupremumCondition}. We repeat it in full below for completeness, and since we use partial orders rather than Boolean algebras, as the former will be used in the iteration theorem.

\begin{defn}
\label{def:InftySubproperness}
A forcing notion $\P$ is $\infty$-\emph{subproper} if every sufficiently large cardinal $\theta$ verifies the $\infty$-subproperness of $\P$, meaning that the following holds: $\P\in H_\theta$, and if $\tau>\theta$ is such that $H_\theta\sub N=L_\tau^A\models\ZFC^-$, and $\sigma:\bN\prec N$, where $\bN$ is countable, transitive and full, and $\bS=\kla{\btheta,\bar{\mathbb P},\bp,\bs}\in\bN$,
$S=\kla{\theta,\P,p,s}=\sigma(\bS)$, where $\bp\in\P$, for $1\le i\le n$, then there is a $q \in \P$ such that $q \le p$ and such that whenever $G\sub\P$ is generic with $q\in G$, then there is a $\sigma'\in\V[G]$ such that
\begin{enumerate}[label=(\alph*)]
  \item $\sigma':\bN\prec N$.
  \item $\sigma'(\bS)=\sigma(\bS)$.
  \item $(\sigma')^{-1}``G$ is $\P$-generic over $\bN$.
\end{enumerate}
\end{defn}

\begin{lem}
$\infty$-subproper forcing notions are closed under factors and forcing equivalence.
\end{lem}

Let us make a remark on the suprema condition that is part of Definition \ref{def:SubpropernessWithSupremumCondition} and is omitted in Definition \ref{def:InftySubproperness}. When motivating his definition of subproperness, Jensen writes in \cite[\S0, p.~3]{Jensen:SPSCF} on this condition:
\begin{quotation}
\it We needed (d) to handle certain regular limit points in the iteration. The experts on the subject may well be able to modify or eliminate (d).
\end{quotation}
This prediction came true, as Miyamoto showed, for the case of subproperness, which we learned after proving our iteration theorems for nice iterations. Our contribution is that this can be done for subcompleteness as well, and that some subclasses of subproper forcing, exhibiting more preservation properties, without requiring the hull or suprema condition, can be iterated in this way as well. The additional preservation properties of the subclasses we consider are the same properties that Miyamoto imposed on the class of semiproper forcing in his iteration theorems from \cite{Miyamoto:IteratingSemiproperPreorders} and \cite{Miyamoto:IteratingOmegaOmegaBoundingSPforcing}. Our proofs in this section are adaptations of Miyamoto's arguments in the context of semiproper forcing. The honor of being considered an expert on the subject by Jensen is entirely Miyamoto's.

\subsection{The theory of nice iterations}
\label{subsec:TheoryOfNiceIterations}

We collect here first the relevant facts and definitions from \cite{Miyamoto:IteratingSemiproperPreorders}. For a more in depth discussion, including proofs, see that article. For basic notions of projection etc, see the introduction there. For a sequence $x$ we denote its length by $l(x)$.

\begin{defn}%[Iterations]
Let $\nu$ be a limit ordinal. A sequence of separative partial preorders%
\footnote{Here, $(\P,\le,1)$ is a partial preorder if $(\P,\le)$ is reflexive and transitive, and if $1$ is a greatest element. There may be several such greatest elements since $\P$ is not required to be antisymmetric. If $p\le q$ and $q\le p$, then we write $p\equiv q$.}
of length $\nu$, $\langle (\P_\alpha, \leq_\alpha, 1_\alpha)\; | \; \alpha < \nu\rangle$ is called a {\em general iteration} iff for any $\alpha \leq \beta < \nu$ the following holds
\begin{enumerate}
\item
For any $p \in \P_\beta$, $p$ is a sequence of length $\beta$, $p \upharpoonright \alpha \in \P_\alpha$ and $1_\beta \upharpoonright \alpha = 1_\alpha$.
\item
For any $p \in \P_\alpha$ and any $q \in \P_\beta$, if $p \leq_\alpha q \upharpoonright \alpha$ then $p^\frown q \upharpoonright [\alpha, \beta) \in \P_\beta$ and $p^\frown q \upharpoonright [\alpha, \beta) \leq_\beta q$
\item
For any $p, q \in \P_\beta$, if $p \leq_\beta q$ then $p \hook \alpha \leq_\alpha q \hook \alpha$ and $p\leq_\beta p \hook \alpha^\frown q\hook[\alpha, \beta)$.
\end{enumerate}
A general iteration $\langle (\P_\alpha, \leq_\alpha, 1_\alpha)\; | \; \alpha < \nu\rangle$ is an \emph{iteration} iff for every limit ordinal $\beta<\nu$ and all $p,q\in \P_\beta$, $p\le_\beta q$ iff for all $\alpha<\beta$, $p\rest\alpha\le_\alpha q\rest\alpha$.
\end{defn}

We will use the following fact (and the notation introduced there).

\begin{fact}[see {\cite[Prop.~1.3]{Miyamoto:IteratingSemiproperPreorders}}]
\label{fact:FactorsAndQuotients}
Let $\langle (\P_\alpha, \leq_\alpha, 1_\alpha)\; | \; \alpha < \nu\rangle$ be a general iteration, and let $\alpha\le\beta<\nu$. Then
\begin{enumerate}[label=(\arabic*)]
  \item Let $G_\beta$ be $\P_\beta$-generic over $\V$. Set $G_\beta\rest\alpha=\{p\rest\alpha\st p\in G_\beta\}$, $G_\beta\rest[\alpha,\beta)=\{p\rest[\alpha,\beta)\st p\in G_\beta \}$, and let \[\P_{\alpha,\beta}=\P_\beta/(G_\beta\rest\alpha)=\{p\rest[\alpha,\beta)\st p\in \P_\beta\ \text{and}\ p\rest\alpha\in G_\beta\rest\alpha\}\] equipped with the ordering $p\le_{\alpha,\beta}q$ iff there is an $r\in G_\beta\rest\alpha$ such that $r\verl p\le_\beta r\verl q$.
      Then $G_\beta\rest\alpha$ is $\P_\alpha$-generic over $\V$ and $G_\beta \rest [\alpha, \beta)$ is $\P_{\alpha,\beta}$-generic over $\V[G_\beta\rest\alpha]$.
  \item If $G_\alpha$ is $\P_\alpha$-generic over $\V$ and $H$ is $\P_{\alpha,\beta}=\P_\beta/G_\alpha$-generic over $\V[G_\alpha]$, then $G_\alpha*H=\{p\in \P_\alpha\st p\rest\alpha\in G_\alpha\ \text{and}\ p\rest[\alpha,\beta)\in H\}$ is $\P_\beta$-generic over $\V$, $(G_\alpha*H)\rest\alpha=G_\alpha$ and $(G_\alpha*H)\rest[\alpha,\beta)=H$.
  \item
  \label{item:CriterionForBeingInGeneric}
  Let $G_\beta$ be $\P_\beta$-generic over $\V$. Then a condition $p\in \P_\beta$ is in $G_\beta$ iff $p\rest\alpha\in G_\beta\rest\alpha$ and $p\rest[\alpha,\beta)\in G_\beta\rest[\alpha,\beta)$.
\end{enumerate}
\end{fact}

In what follows we suppress the notation $\leq_\alpha, 1_\alpha$ and identify a partial preorder with its underlying set. We have the following useful proposition.

\begin{proposition}[Proposition 1.7 of \cite{Miyamoto:IteratingSemiproperPreorders}]
\label{generic}
Let $\kla{\P_\alpha\st\alpha < \nu}$ be an iteration and $\beta<\nu$ a limit ordinal. Then for any $p \in\P_\beta$ and any $\P_\beta$-generic $G_\beta$ we have $p\in G_\beta$ if and only if for all $\alpha<\beta$, $p\rest\alpha \in G_\beta\rest\alpha$.
\end{proposition}

From now on we always assume the sequence $\vec{\P} = \kla{\P_\alpha\st\alpha<\nu}$ is an iteration. Let us quickly define the relevant concepts: nested antichain, $S \hooks T$, mixtures and $(T, \beta)$-niceness. We refer the reader to \cite{Miyamoto:IteratingSemiproperPreorders} for an in depth discussion of these ideas and their significance.

\begin{defn}%[The machinery of Nice Iterations]
\label{definition:NestedAC-Mixture-beta-Nice-hooks}
A {\em nested antichain} in $\vec{\P}=\langle\P_\alpha\st\alpha<\nu\rangle$ is a triple $\langle T, \langle T_n \; | \; n < \omega\rangle, \langle {\rm suc}^n_T \; | \; n < \omega\rangle \rangle$ so that
\begin{enumerate}
\item
$T = \bigcup_{n < \omega} T_n$
\item
$T_0$ consists of a unique element of some $\P_\alpha$ for $\alpha < \nu$. We denote this unique element by $\Root(T)$.

\noindent For each $n < \omega$ we have that

\item $T_n \subseteq \bigcup\{\P_\alpha \; | \; \alpha < \nu\}$ and ${\rm suc}^n_T: T_n \to \mathcal P(T_{n+1})$.

\item For $a \in T_n$ and $b \in {\rm suc}^n_T (a)$, $l(a) \leq l(b)$ and $b \hook l(a) \leq a$.

\item If $a \in T_n$ and $b_0,b_1\in\suc^n_T(a)$ are distinct, then $b_0\rest l(a)$ and $b_1\rest l(a)$ are incompatible in $\P_{l(a)}$. Moreover, the set of all $b\rest l(a)$ with $b \in {\rm suc}^n_T(a)$ forms a maximal antichain below $a$ in $\P_{l(a)}$. In particular, this set is not empty.

\item $T_{n+1} = \bigcup \{ {\rm suc}^n_T (a) \; | \; a \in T_n\}$
\end{enumerate}

Given such a nested antichain $T$ in $\vec{\P}$ with $a_0=\Root(T)$, for a condition $p \in \P_\beta$ with $\beta<\nu$ we say that $p$ is {\em a mixture of $T$ up to }$\beta$ if for all $i<\beta$ the condition $p\rest i$ forces that
\begin{enumerate}
\item
$p\rest[i, i+1) \equiv a_0\rest [i, i+1)$ in $\P_{i,i+1}(=\P_{i+1}/\dot{G}_i)$ and $a_0\rest i\in\dot{G}_i$ if $i<l(a_0)$, where $\dot{G}_i$ is the canonical name for the $\P_i$-generic filter.
\item
$p\rest[i,i+1) \equiv s\rest[i,i+1)$ if there are $r, s \in T$ with $s \in {\rm suc}^n_T(r)$ for some $n$ and $l(r) \leq i < l(s)$ and $s\rest i \in \dot{G}_i$.
\item
$p\rest[i, i+1)\equiv 1_{i+1}\rest[i,i+1)$ if there is a sequence $\kla{a_n\st n<\omega}$ so that $a_0\in T_0$ and for all $n<\omega$, $a_{n+1} \in {\rm suc}^n_T(a_n)$ and $l(a_n)\leq i$ and $a_n\in\dot{G}_i\rest l(a_n)$.
\end{enumerate}

If $\beta$ is a limit ordinal we say that a sequence $p$ of length $\beta$ (not necessarily in $\mathbb P_\beta$) is $(T, \beta)$-{\em nice} if for all $\alpha < \beta$, $p \hook \alpha$ is a mixture of $T$ up to $\alpha$.

Finally, given two nested antichains $S$ and $T$ in $\vec{\P}$ we define $S \hooks T$ (``$S$ hooks $T$'') if for every $n < \omega$ and all $b \in S_n$ there is an $a \in T_{n+1}$ so that $l(a) \leq l(b)$ and $b\hook l(a) \leq a$.
\end{defn}

This previous definition combines Definitions 2.0, 2.4 and 2.10 of \cite{Miyamoto:IteratingSemiproperPreorders}. We will need the following characterization of mixtures.

\begin{fact}[see {\cite[Prop.~2.5]{Miyamoto:IteratingSemiproperPreorders}}]
\label{fact:CharacterizationOfMixtures}
Let $T$ be a nested antichain in an iteration $\seq{\P_\alpha}{\alpha<\nu}$, $\beta<\nu$ and $p\in \P_\beta$. Then $p$ is a mixture of $T$ up to $\beta$ iff the following hold:
\begin{enumerate}[label=(\arabic*)]
    \item
    \label{item:Root}
    Let $T_0=\{a_0\}$ and $\mu=\min(l(a_0),\beta)$. Then $a_0\rest\mu\equiv p\rest\mu$.
    \item
    \label{item:General}
    For any $a\in T$, letting $\mu=\min(l(a),\beta)$, we have that $a\rest\mu\le p\rest\mu$.
    \item
    \label{item:HigherUp}
    If $n<\omega$, $a\in T_n$, $b\in\suc_T^n(a)$ and $l(a)\le\beta$, then, letting $\mu=\min(\beta,l(b))$, we have that $b\rest\mu\equiv b\rest l(a)\verl p\rest[l(a),\mu)$.
    \item
    \label{item:Chain}
    For any $i\le\beta$ and any $q\in \P_i$ with $q\le_i p\rest i$, if $q$ forces with respect to $\P_i$ that there is a sequence $\seq{a_n}{n<\omega}$ such that $a_0\in T_0$, and for all $n<\omega$, $a_{n+1}\in\suc_T^n(a_n)$, $l(a_n)\le i$ and $a_n\in\dot{G}_i\rest l(a_n)$, then $q\verl 1_\beta\rest[i,\beta)\equiv q\verl p\rest[i,\beta)$.
\end{enumerate}
\end{fact}

The following is Definition 3.6 in Miyamoto's article.

\begin{defn}%[Nice Iterations]
\label{def:NiceIterations}
An iteration $\kla{\P_\alpha\st\alpha<\nu}$ is called a {\em nice iteration} if
\begin{enumerate}
\item
For any $i$ such that $i+1<\nu$, if $p \in \P_i$ and $\tau$ is a $\P_i$-name such that $p\forces_i$``$\tau \in \P_{i+1}$ and $\tau\rest i \in\dot{G_i}$'' then there is a $q \in \P_{i +1}$ so that $q \hook i = p$ and $p \forces_i\tau\hook[i, i+1) \equiv q \hook [i, i+1)$.
\item
For any limit ordinal $\beta < \nu$ and any sequence $x$ of length $\beta$, $x \in \P_\beta$ if and only if there is a nested antichain $T$ in $\langle \P_\alpha \; | \; \alpha < \beta\rangle$ such that $x$ is $(T, \beta)$-nice.
\end{enumerate}
\end{defn}

In its original formulation, the following corollary assumes the iteration in question to be \emph{rich}, a concept we will not need here. But \cite[Proposition 3.7]{Miyamoto:IteratingSemiproperPreorders} says that nice iterations are rich, so we restate the corollary in terms of nice iterations.

\begin{cor}[Corollary 3.3 of \cite{Miyamoto:IteratingSemiproperPreorders}]
\label{cor:Fullness}
Let $\seq{\P_\alpha}{\alpha<j}$ be a nice iteration, and fix $i\le\beta<j$. Let $p\in\P_i$, and let $\tau\in\V^{\P_i}$ be such that $p\forces_i$ ``$\tau\in\P_\beta$ and $\tau\rest i\in\dot{G}_i$.'' Then there is a $q\in\P_\beta$ such that $q\rest i=p$ and $p\forces_i$ ``$q\rest[i,\beta)\equiv_{\P_{i,\beta}}\tau\rest[i,\beta)$.''
\end{cor}

\begin{lem}[Lemma 2.7 of \cite{Miyamoto:IteratingSemiproperPreorders}]
Let $\nu$ be a limit ordinal and $A \subseteq \nu$ be cofinal. Suppose that $T$ is a nested antichain in an iteration $\langle \P_\alpha \; | \; \alpha < \nu\rangle$ and $p$ is a sequence of length $\nu$ such that $p$ is $(T, \nu)$-nice. Then for any $\beta < \nu$ and any $s \in \P_\beta$ strengthening $p \hook \beta$ we get a nested antichain $S$ so that
\begin{enumerate}
\item
If $T_0 = \{a_0\}$ and $S_0 = \{b_0\}$ then $l(b_0) \in A$ and $l(a_0), \beta \leq l(b_0)$.
\item
For any $b \in S$, $l(b) \in A$.
\item
$r = s^\frown p \hook [\beta, \nu)$ is $(S, \nu)$-nice.
\end{enumerate}
\label{2.7}
\end{lem}

\begin{lem}[Lemma 2.11 of \cite{Miyamoto:IteratingSemiproperPreorders}]
Let $\nu$ be a limit ordinal and $\langle \P_\alpha \; | \; \alpha < \nu\rangle$ an iteration. Let $A \subseteq \nu$ be a cofinal subset of $\nu$. If $T$ and $U$ are nested antichains, $p$ and $q$ are sequences of length $\nu$ with $p$ $(T, \nu)$-nice and $q$ $(U, \nu)$-nice and $r \in T_1$ so that $q \hook l(r) \leq r$ and for all $\alpha \in [l(r), \nu)$, $q \hook \alpha \leq r^\frown p\hook[l(r), \alpha)$; then there is a nested antichain $S$ in $\kla{\P_\alpha\st\alpha<\nu}$ so that $q$ is $(S, \nu)$-nice, if $\{b_0\} = S_0$ then $l(r) \leq l(b_0)$ and $b_0 \hook l(r) \leq r$, for all $s \in S$, $l(s) \in A$ and $S \hooks T$.
\label{2.11}
\end{lem}

We also recall the definition of a fusion structure.

\begin{defn}[Def.~3.4 of \cite{Miyamoto:IteratingSemiproperPreorders}]
\label{def:FusionStructure}
Let $\vec{\P}=\kla{\P_\alpha\st\alpha\le\nu}$ be an iteration, where $\nu$ is a limit ordinal. Given a nested antichain $T$ in $\vec{\mathbb P}\rest\nu$, we call a structure $\langle (q^{(a, n)}, T^{(a, n)}) \st a\in T_n, n<\omega\rangle$ a {\em fusion structure} if for all $n<\omega$ and $a\in T_n$ the following hold:
\begin{enumerate}
\item
$T^{(a, n)}$ is a nested antichain in $\langle \P_\alpha \; | \; \alpha < \nu\rangle$.
\item
$q^{(a, n)} \in \P_\nu$ is a mixture of $T^{(a, n)}$ up to $\nu$.
\item
$a\leq q^{(a, n)}\rest l(a)$ and $l(a)=l(\Root(T^{(a,n)}))$. %if $\{p_0\} = T_0^{(a, n)}$ then $l(a) = l(p_0)$.
\item
For any $b \in {\rm suc}^n_T(a), T^{(b, n+1)} \hooks T^{(a, n)}$ so $q^{(b, n+1)} \leq q^{(a, n)}$.
\end{enumerate}
If $p \in \P_\nu$ is a mixture of $T$ up to $\nu$ then we call $p$ a {\em fusion} of the fusion structure.
\end{defn}

\begin{proposition}[Proposition 3.5 of \cite{Miyamoto:IteratingSemiproperPreorders}]
Let $\kla{\P_\alpha\st\alpha\leq\nu}$ be an iteration, where $\nu$ is a limit ordinal. If $p\in\P_\nu$ is a fusion of a fusion structure $\langle(q^{(a, n)}, T^{(a, n)})\st a\in T_n, n<\omega\rangle$ then $p$ forces that there is a sequence $\kla{a_n\st n<\omega}$ such that the following hold:
\begin{enumerate}
\item
$a_0\in T_0$, and for all $n<\omega$, $a_{n+1}\in{\rm suc}^n_T(a_n)$, $a_n\in\dot{G}_\nu\rest l(a_n)$ and $q^{(a_n, n)}\in\dot{G}_\nu$.
\item
If $\beta=\sup\{l(a_n)\st n<\omega\}$ then $q^{(a_n, n)}\rest\beta\in \dot{G}_\nu\rest\beta$ and $q^{(a_n,n)}\rest[\beta,\nu)\equiv 1_\nu\rest[\beta,\nu)$.
\end{enumerate}
\label{fusion}
\end{proposition}

\subsection{Nice iterations of $\infty$-subcomplete forcing}
\label{subsec:NiceIterationsOfInftySC}

First we prove that $\infty$-sub\-com\-plete forcing is preserved under nice iterations. We use the following notational convention: if $\langle \P_\alpha \; |$ $\; \alpha \leq \nu\rangle$ is an iteration then for $i \leq j \leq \nu$ the poset $\P_{i, j}$, which is defined in Fact \ref{fact:FactorsAndQuotients}, depends on the $\P_i$-generic chosen, so we will identify it with its $\P_i$ name $\P_j/\dot{G}_i$.

The special case $i=0$ and $j=\nu$ of following theorem implies that if every successor stage of a nice iteration $\seq{\P_\alpha}{\alpha\le\nu}$ is forced to be $\infty$-subcomplete, then $\P_\nu$ is $\infty$-subcomplete.

\begin{thm}
\label{thm:NiceIterationsOfSCforcingAreSC}
Let $\vec{\mathbb P}=\kla{\P_\alpha\st\alpha\leq\nu}$ be a nice iteration so that $\P_0=\{1_0\}$ and for all $i$ with $i+1<\nu$, $\forces_i\P_{i,i+1}$ is $\infty$-subcomplete. Then for all $j\le\nu$ the following statement $\phi(j)$ holds:

{\bf if} $i\le j$, $p\in \P_i$, $\dot{\sigma}\in\V^{\P_i}$, $\theta$ is a sufficiently large cardinal, $\tau$ is an ordinal, $H_\theta\sub N=L_\tau^A\models\ZFC^-$, $\bN$ is a transitive model, $\bs,\bar{\vec{\P}},\bar{i},\bar{j}\in\bN$, $\bar{G}_{\bar{i}},\bar{G}_{\bar{i},\bar{j}}\sub\bN$, and the following assumptions hold:
\begin{enumerate}[label=\textnormal{(A\arabic*)}]
\item
  \label{item:FirstAssumption1}
  $p$ forces with respect to $\P_i$ that
  \[\dot{\sigma}(\check{\bar{\vec{\P}}},\check{\bar{i}},\check{\bar{j}},\check{\btheta},\check{\bG}_{\bar{i}})=\check{\vec{\P}},\check{i},\check{j},\check{\theta},\dot{G}_i.\] %and $\dot{\sigma}``\check{\bar{G}}_{\bar{i}}\sub\dot{G}_i$
  \item
  \label{item:StuffMovedRight}
  $\bar{G}_{\bar{i}}$ is $\bar{\P}_{\bar{i}}$-generic over $\bN$ and
  $\bar{G}_{\bar{i},\bar{j}}$ is $\bar{\P}_{\bar{i},\bar{j}}$-generic over $\bN[\bar{G}_{\bar{i}}]$, where $\bar{\P}_{\bar{i},\bar{j}}=\bar{\P}_{\bar{j}}/\bar{G}_{\bar{i}}$.
  \item
  \label{item:Lastassumption}
  $\bN$ is countable, transitive and full, and
  $p$ forces with respect to $\P_i$ that
  \[\dot{\sigma}:\check{\bN}[\check{\bG}_{\bar{i}}]\prec \check{N}[\dot{G}_i].\]
\end{enumerate}
{\bf then} there is a condition $p^*\in \P_j$ such that %$y^*\le x$,
$p^*\rest i=p$ and whenever $G_j\ni p^*$ is $\P_j$-generic, then in $\V[G_j]$, there is a $\sigma'$ such that, letting $\sigma=\dot{\sigma}^{G_i}$, the following hold:
\begin{enumerate}[label=\textnormal{(C\arabic*)}]
  \item
  \label{item:FirstConclusionSC}
  $\sigma'(\bs,\bar{\vec{\P}},\bar{i},\bar{j},\btheta,\bG_{\bar{i}})=\sigma(\bs,\bar{\vec{\P}},\bar{i},\bar{j},\btheta,\bG_{\bar{i}})$.
  \item
  \label{item:Sigma'MovesEverythingCorrectly1}
  $(\sigma')``\bar{G}_{\bar{i},\bar{j}}\sub G_{i,j}$.
  \item
  \label{item:LastConclusionSC}
  $\sigma':\bN[\bG_{\bar{i}}]\prec N[G_i]$.
\end{enumerate}
\label{sciteration}
\end{thm}

Let us stress that our proof is similar to that of \cite[Lemma 4.3]{Miyamoto:IteratingSemiproperPreorders}, adapting to the case of $\infty$-subcomplete forcings in place of semiproper forcings.

\begin{proof}
The proof is by induction on $j$. So let us assume that $\phi(j')$ holds for every $j'<j$.
%Fixing $j$, we prove the claim by induction on $i$.
Fix some $i\le j$.
%assuming the theorem is proven for all $i'<i$ (and also for all $i'$, $j'$ with $i'\le j'<j$).
Since nothing is to be shown when $i=j$, let $i<j$. In particular, the case $j=0$ is trivial.

Let us fix $p\in\P_i$, $\dot{\sigma}\in\V^{\P_i}$, %$x\in P_j$, $y\le x\rest i$,
$\theta$, $\tau$, $A$, $N$, $\bN$, $\bs$, $\bar{\vec{\P}}$, $\bar{i}$, $\bar{j}$, $\bar{G}_{\bar{i}}$, $\bar{G}_{\bar{i},\bar{j}}$ so that assumptions \ref{item:FirstAssumption1}-\ref{item:Lastassumption} hold.

When it causes no confusion, if $a$ is in the range of some elementary embedding $\sigma_0: \bN \prec N$ then we will denote $\sigma_0^{-1}(a)$ by $\overline{a}$. % Often we will not note this explicitly. Finally let us note that if $\sigma_0 : \bN \prec N$ then for any sequence $a \in N$ in the range of $\sigma_0$ the length of $a$, $l(a)$ is such that $l(\overline{a}) = l(\overline{a})$ by elementarity (this being an example of our aforementioned ``bar'' convention).

\medskip

\noindent\emph{Case 1:} $j$ is a limit ordinal.

Let $\{t_n \; | \; n < \omega \}$ enumerate the elements of $\bN$, with $t_0=\emptyset$. Throughout this proof we will identify the $t_n$'s with their check names when it causes no confusion. Also let $\{\overline{p}_n \; | \; n <\omega \}$ enumerate the elements of $\bar{G}_{\bar{i}}*\bar{G}_{\bar{i},\bar{j}}$, with $\bar{p}_0=1_{\bar{\P}_{\bar{j}}}$. It follows then that $p$ forces that $\dot{\sigma}(\check{\overline{p}}_0)=\check{1}_j$, so we write $p_0=1_j$.

Note that since in some forcing extension, there is an elementary embedding from $\bN$ to $N = L_\tau^A$, it follows that $\bN$ is of the form $L_{\bar{\tau}}^{\bar{A}}$, and hence, $\bN$ has a canonical definable well-order, which we denote by $<_\bN$.
Since $\bar{p}_0=1_{\bar{\P}_{\bar{j}}}$, there is a trivial nested antichain $\bar{W}\in\bN$ in $\vec{\bar{\P}}\rest\bar{j}$ so that $\bar{p_0}$ is a mixture of $\bar{W}$ up to $\bar{j}$ and $l(\Root(\bar{W}))=\bar{i}$.
%As noted in \cite{Miyamoto:IteratingSemiproperPreorders}, by Lemma \ref{2.7}, in $\bN$, $\bar{p}_0$ is a mixture up to $\bar{j}$ of some nested antichain in $\bar{\vec{\P}}\rest{\bar{j}}$ whose root has length $\bar{i}$. %
Letting $\bar{W}$ be the $<_\bN$-least one, we know that $p$ forces that $\dot{\sigma}(\check{\bar{W}})$ is the $L_\tau^A$-least nested antichain $W$ in $\vec{\P}\rest j$ such that $p_0$ is a mixture of $W$ up to $j$ and such that the length of the root of $W$ is $i$, since we know that $p$ forces that $\bar{p}_0,\bar{i},\bar{j}$ are mapped to $p_0,i,j$ by $\dot{\sigma}$, respectively.

We will define a nested antichain $\langle T, \langle T_n \; | \; n < \omega \rangle, \langle \suc^n_T \; | \; n < \omega\rangle \rangle$ in $\vec{\P}\rest j$, a fusion structure $\seq{\kla{q^{(a, n)}, T^{(a, n)}}}{n < \omega, a\in T_n}$ in $\seq{\P_\alpha}{\alpha \leq j}$ and a sequence $\langle \dot{\sigma}^{(a, n)} \; | \; n < \omega , a \in T_n\rangle$ so that the following conditions hold.

\begin{enumerate}[label=(\arabic*)]
\item
\label{claim:Start1}
$T_0 = \{p\}$, $q^{(p,0)}=p_0$, $T^{(p,0)}=W$ and $\dot{\sigma}^{(p,0)} = \dot{\sigma}$.
\end{enumerate}
Further, for any $n<\omega$ and $a\in T_n$:
\begin{enumerate}[label=(\arabic*)]
\setcounter{enumi}{1}
\item %6
\label{item:middlepart-beginning1}
%$x^{(a, n)} \leq x$ and
$q^{(a, n)}\in \P_j$ and $\dot{\sigma}^{(a, n)}$ is a $\P_{l(a)}$-name,
\item %7
\label{item:sigma(a,n)1}
$a$ forces the following statements with respect to $\P_{l(a)}$:
\begin{enumerate}[label=(\alph*)]
\item
\label{subitem:iselementary1}
$\dot{\sigma}^{(a, n)} : \bN[\bG_{\bar{i}}]\prec N[\dot{G}_i]$,
\item
\label{subitem:moveseverythingcorrectly1}
$\dot{\sigma}^{(a, n)}(\overline{\theta}, \bar{\vec{\P}}, \overline{s}, \overline{\nu}, \bar{i},\bar{j},\bar{G}_{\bar{i}}) = \theta, \vec{\P}, \tau,\dot{\sigma}(\bs), \nu, i, j, \dot{G}_i$,
%\item ${\mathrm{Hull}}^N(\delta(P_\nu)\cup\ran(\dot{\sigma}^{(a, n)})) = {\mathrm{Hull}}^N(\delta(P_\nu)\cup\ran(\check{\sigma}))$.
\item
\label{subitem:hasx(a,n)initsrange1}
$q^{(a, n)} \leq_j \dot{\sigma}^{(a, n)}(\overline{p}_n)$, and $q^{(a,n)}\in\ran(\dot{\sigma}^{(a,n)})$.
\end{enumerate}
%\item %8
%$a$ decides $\dot{\sigma}^{(a, n)}(t_n)$.
\item %9
\label{item:middlepart-end1}
for some $\bar{q}^{(a,n)}\in\bar{\P}_{\bar{j}}$, $\bar{T}^{(a,n)}\in\bN$ and $\overline{l(a)}\in\bN$, we have that $\bar{q}^{(a, n)} \in \bG_{\bar{i}}*\bar{G}_{\bar{i},\bar{j}}$, $\bar{q}^{(a, n)} \leq \bar{p}_n$ and \\
$a \forces \dot{\sigma}^{(a, n)} (\kla{\bar{q}^{(a, n)}, \bar{T}^{(a, n)}, \overline{l(a)}}) =\kla{q^{(a, n)}, T^{(a, n)}, l(a)}$.
\end{enumerate}
\noindent If $b \in {\rm suc}^n_T (a)$, then
\begin{enumerate}[label=(\arabic*)]
\setcounter{enumi}{4}
\item %10
\label{item:coherence1}
$b \forces_{l(b)} \dot{\sigma}^{(b, n+1)}(t_m) = \dot{\sigma}^{(a, n)}(t_m)$ and $\dot{\sigma}^{(b, n+1)}(\bar{p}_m) = \dot{\sigma}^{(a, n)}(\bar{p}_m)$ for all $m\le n$.

%\ \text{and}\ \dot{\sigma}^{(b, n+1)}(v_m) = \dot{\sigma}^{(a, n)}(v_m),$ where $v_m$ is the $\bN$-least $v$ with $\sigma(t_m) \in \dot{\sigma}^{(a_m, m)}(v)$ and $\card{v}^{\bN} \leq \overline{\delta(P_\nu)}$, where $a_m$ is the predecessor of $a$ in $T_m$. Note that such a $v$ always exists by Fact \ref{ufact}.
\item %11
\label{claim:LastOfFirstList1}
$b\rest l(a)\forces_{l(a)}q^{(b,n+1)},T^{(b,n+1)}\in\ran(\dot{\sigma}^{(a,n)})$.
\end{enumerate}
First let's see that constructing such objects is sufficient to prove the existence of a condition $p^*$ as in the statement of the theorem. So suppose that we have constructed a nested antichain $\langle T, \langle T_n \; | \; n < \omega \rangle, \langle \suc^n_T \; | \; n < \omega\rangle \rangle$, a fusion structure $\seq{\kla{q^{(a, n)}, T^{(a, n)}}}{n < \omega, a\in T_n}$ in $\seq{\P_\alpha}{\alpha \leq j}$ and a sequence $\langle \dot{\sigma}^{(a, n)} \; | \; n < \omega , a \in T_n\rangle$, so that \ref{claim:Start1} through \ref{claim:LastOfFirstList1} above are satisfied. Let $q^*\in \P_j$ be a fusion of the fusion structure, and let $p^*=p\verl q^*\rest[i,j)$. To see that $p^*$ is as wished, note first that by \ref{claim:Start1}, we have that $q^*\rest i\equiv p$, so $p^*\equiv q^*$ and $p^*\rest i=p$, as required. Now let $G_j$ be $\P_j$-generic over $\V$ with $p^* \in G_j$. We have to show that in $\V[G_j]$, there is a $\sigma'$ so that conclusions \ref{item:FirstConclusionSC}-\ref{item:LastConclusionSC} are satisfied.
Since $p^*\equiv q^*$, we have that $q^*\in G_j$. Work in $\V[G_j]$. By Proposition \ref{fusion} there is a sequence $\langle a_n \; | \; n < \omega\rangle \in \V[G_j]$ so that for all $n < \omega$, $a_n\in T_n$, $a_{n+1} \in {\rm suc}^n_T(a_n)$, $a_n \in G_j \hook l(a_n)$ and $q^{(a_n, n)} \in G_j$. Let $\sigma_n$ be the evaluation of $\dot{\sigma}^{(a_n, n)}$ by $G_j$. Then we define $\sigma': \bN \to N$ to be the map such that $\sigma ' (t_n) = \sigma_n (t_n)$. %Note that $a_n$ decides $\dot{\sigma}^{(a, n)}(t_n)$, and this is exactly the value given in the extension.
We claim that $\sigma'$ satisfies the conclusions \ref{item:FirstConclusionSC}-\ref{item:LastConclusionSC}.

Condition \ref{item:FirstConclusionSC} says that $\sigma'$ moves the parameters $\bar{s},\bar{\vec{\P}},\bar{i},\bar{j},\btheta$ and $\bG_{\bar{i}}$ the same way $\sigma=\dot{\sigma}^{G_i}$ does. But this is true of every $\sigma_n$, hence also of $\sigma'$.

Condition \ref{item:Sigma'MovesEverythingCorrectly1} says that $(\sigma')``\bG_{\bar{i},\bar{j}}\sub G_{i,j}$. So let $\bp\in\bG_{\bar{i},\bar{j}}$. We have to show that $\sigma'(\bp)\in G_{i,j}$.
Recall $\bG_{\bar{i}}*\bG_{\bar{i},\bar{j}} = \{\overline{p}_n \; | \; n < \omega \}$. Let $n$ be such that $\bp=\bp_n\rest[\bar{i},\bar{j})$. By \ref{item:sigma(a,n)1}\ref{subitem:hasx(a,n)initsrange1}, we have that $q^{(a_n, n)}\leq_j\sigma_n(\overline{p}_n)$, so since $q^{(a_n, n)}\in G_j$, it follows that $\sigma_n(\overline{p}_n)\in G_j$ as well. By \ref{item:coherence1} and the definition of $\sigma'$,  we have that $\sigma_n(\overline{p}_n) = \sigma'(\overline{p}_n)$.
It follows that  $\sigma'(\bp)=\sigma'(\bp_n\rest[\bar{i},\bar{j}))=\sigma'(\bp_n)\rest[i,j)\in G_{i,j}$, as claimed.

Condition \ref{item:LastConclusionSC} says that $\sigma':\bN[\bar{G}_{\bar{i}}\prec N[G_i]$ is elementary.
Since any one formula can only use finitely many parameters, and $\sigma'\rest\{t_0,\ldots,t_n\}=\sigma_n\rest\{t_0,\ldots,t_n\}$, this is true by \ref{item:coherence1} and \ref{item:sigma(a,n)1}\ref{subitem:iselementary1}.

Therefore it remains to show that the construction described above can actually be carried out. This is done by recursion on $n$. It proceeds as follows. At stage $n+1$ of the construction, we
assume that for all $m\le n$ and all $a\in T_m$,
we have defined $T_m$, $T^{(a,m)}$, $q^{(a,m)}$ and $\dot{\sigma}^{(a,m)}$. Also, for $m<n$ and $a\in T_m$, we assume that $\suc_T^m(a)$ has been defined. Our inductive hypothesis is that for all $m\le n$ and all $a\in T_m$, conditions
\ref{item:middlepart-beginning1}-\ref{item:middlepart-end1}
hold, and that for all $m<n$, all $a\in T_m$ and all $b\in\suc_T^m(a)$, conditions \ref{item:coherence1}-\ref{claim:LastOfFirstList1} are satisfied. In order to define $T_{n+1}$, we will specify $\suc_T^n(a)$, for every $a\in T_n$, which implicitly defines $T_{n+1}=\bigcup_{a\in T_n}\suc_T^n(a)$. Simultaneously, we will define, for every such $a$ and every $b\in\suc_T^n(a)$, the objects $T^{(b,n+1)}$, $q^{(b,n+1)}$ and $\dot{\sigma}^{(b,n+1)}$ in such a way that whenever $a\in T_n$ and $b\in\suc_T^n(a)$, \ref{item:coherence1}-\ref{claim:LastOfFirstList1} are satisfied by $a$ and $b$, and \ref{item:middlepart-beginning1}-\ref{item:middlepart-end1} are satisfied by $b$ and $n+1$ (instead of $a$ and $n$). In order to ensure that we ultimately produce a fusion structure, we arrange that $q^{(b,n+1)}$ is a mixture of $T^{(b,n+1)}$ up to $j$, $b\le q^{(b,n+1)}\rest l(b)$, $l(b)=l(\Root(T^{(b,n+1)}))$ and $T^{(b,n+1)}\hooks T^{(a,n)}$, so that Definition \ref{def:FusionStructure}(2)-(4) hold (and we also assume that the corresponding statements are true of the earlier defined $q^{(c,k)}$ and $T^{(c,k)}$).

For stage 0 of the construction, notice that \ref{claim:Start1} gives the base case where $n =0$ and in this case \ref{item:middlepart-beginning1}-\ref{item:middlepart-end1} are satisfied, $p$ forces that $q^{(p, 0)}=p_0=\dot{\sigma}(\overline{p}_0)$ and $p$ forces that $W=T^{(p,0)}$ has a preimage under $\dot{\sigma}$, namely $\bar{W}$. Note that we have $p\le_{l(p)}q^{(p,0)}\rest l(p)=1_i$ and $l(p)=i=l(\Root(T^{(p,0)}))$, as required in a fusion structure.

At stage $n+1$ of the construction, we work under the assumptions described above. Fixing $a\in T_n$, we have to define $\suc_T^n(a)$.
To this end, let $D$ be the set of all conditions $b\in\bigcup_{l(a)\le\xi<j}\P_\xi$ for which there are a nested antichain $S$ in $\vec{\P}\rest j$ and objects $\dot{\sigma}^b$, $u$, $\bar{u}$, $\bar{S}$ and $\overline{l(b)}$ satisfying the following:

\begin{enumerate}[label=(D\arabic*)]
%\setcounter{enumi}{11}
%\item %12
%$b \in \P_{l(b)}$ and $l(b) < j$. [redundant]
\item %13
\label{item:FirstConditionDefiningPredenseSet1}
$b\rest l(a)\le_{l(a)}a$.
\item %14
$S \hooks T^{(a, n)}$, $\bS\in\bN$, $S\in N$.
\item %15
$u\in \P_j$, $u\le q^{(a,n)}$ and $u$ is a mixture of $S$ up to $j$.
\item
\label{item:EnsuringGettingAFusionStructure}
$b\le_{l(b)}u\rest l(b)$ and $l(b)=l(\Root(S))$.
\item %16
$\overline{u} \in \bG_{\bar{i}}*\bG_{\bar{i},\bar{j}}$, and $\overline{u} \leq \overline{p}_{n+1}$.
%\item %17
%$b$ forces that $\dot{\sigma}^b$ is a function.%, and $b$ decides $\dot{\sigma}^b(t_{n+1})$.
\item %18
\label{item:InitialSegmentOfbForcesStuffIntoRangeOfTheOldEmbedding1}
$b\rest l(a)\forces_{l(a)}\dot{\sigma}^{(a,n)}(\check{\bS},\check{\bar{u}},(\overline{l(b)})\check{})
=\check{S},\check{u},(l(b))\check{}$.
%$b \forces_{l(b)} S, u, l(b) \in\ran (\dot{\sigma}^b)$.
\item %19
\label{item:LastConditionDefiningPredenseSet1}
$b$ forces the following statements with respect to $\P_{l(b)}$: \begin{enumerate}
\item $\dot{\sigma}^b (\overline{\theta},\bar{i},\bar{j},\bar{\vec{\P}},\bG_{\bar{i}}, \overline{s},\overline{u},\bS,(\overline{l(b)})\check{}) = \theta,i,j,\vec{\P},\dot{G}_i,\dot{\sigma}(\bs), u,S,(l(b))\check{}$.
\item $\forall m \leq n$ $\dot{\sigma}^b(t_m) = \dot{\sigma}^{(a, n)}(t_m)$ and $\dot{\sigma}^b(\overline{p}_m) = \dot{\sigma}^{(a, n)}(\overline{p}_m)$.
\item $\dot{\sigma}^b: \bN[\bG_{\bar{i}}] \prec N[\dot{G}_i]$.
%\item ${\mathrm{Hull}}^N(\delta(P_\nu) \cup \ran (\dot{\sigma}^b)) = {\mathrm{Hull}}^N(\delta(P_\nu) \cup \ran (\dot{\sigma}^{(a, n))})$,
\end{enumerate}
\end{enumerate}

Note that if $b\in D$ and $b'\le_{l(b)}b$, then $b'\in D$ as well. It follows that $D\rest l(a):=\{b\rest l(a)\st b\in D\}$ is open in $\P_{l(a)}$.
Thus, it suffices to show that $D\rest l(a)$ is predense below $a$ in $\P_{l(a)}$. For if we know this, $D\rest l(a)$ is dense below $a$, and we may choose a maximal antichain $A\sub D\rest l(a)$ (with respect to $\P_{l(a)}$), which then is a maximal antichain in $\P_{l(a)}$ below $a$. Thus, for every $c\in A$, we may pick a condition $b(c)\in D$ such that $b(c)\rest l(a)=c$, and define $\suc_T^n(a)=\{b(c)\st c\in A\}$ (in order to satisfy Definition \ref{definition:NestedAC-Mixture-beta-Nice-hooks}, part $(5)$). Now, for every $b\in\suc_T^n(a)$, let $S$, $\dot{\sigma}^b$, $u$ and $\bar{u}$ witness that $b\in D$, i.e., let them be chosen in such a way that \ref{item:FirstConditionDefiningPredenseSet1}-\ref{item:LastConditionDefiningPredenseSet1} hold. Set $T^{(b,n+1)}=S$, $\dot{\sigma}^{(b, n+1)}=\dot{\sigma}^b$, $q^{(b, n+1)}=u$, $\bar{q}^{(b,n+1)}=\bar{u}$ and $\bar{T}^{(b,n+1)}=\bS$. Then $a$, $b$ satisfy \ref{item:coherence1}-\ref{claim:LastOfFirstList1} at stage $n$, and $b$ satisfies \ref{item:middlepart-beginning1}-\ref{item:middlepart-end1} at stage $n+1$.

To see that $D\rest l(a)$ is predense below $a$, let
$G_{l(a)}$ be $\P_{l(a)}$-generic over $V$ with $a \in G_{l(a)}$. We have to find a $b\in D$ so that $b \hook l(a) \in G_{l(a)}$. Work in $V[G_{l(a)}]$. Let $\sigma_n =(\dot{\sigma}^{(a, n)})^{G_{l(a)}}$. Since \ref{item:sigma(a,n)1} holds at stage $n$, we have that $\sigma_n:\bN[\bG_{\bar{i}}]\prec N[G_i]$ %${\mathrm{Hull}}^N(\delta(P_\nu) \cup \ran(\sigma)) = {\mathrm{Hull}}^N(\delta(P_\nu) \cup \ran(\sigma_n))$,
and $\sigma_n(\overline{\theta}, \bar{\vec{\P}},\bs,\overline{\nu},\bar{i},\bar{j},\bG_{\bar{i}})
=\theta,\vec{\P},\sigma(\bs),\nu,i,j,G_i$.
%
% \overline{p}_0,...,\overline{p}_n, \overline{T^{(a, n)}}, l(\overline{a}), \overline{x}^{(a, n)}) = \theta, P_\nu, \nu, p_0,...,p_n, T^{(a, n)}, l(a), x^{(a, n)}$.
We also have objects $\bar{q}^{(a,n)},\bT^{(a,n)},l(\overline{a})$ satisfying condition \ref{item:middlepart-end1}, so that $\bar{q}^{(a,n)}\in\bG_{\bar{i}}*\bG_{\bar{i},\bar{j}}$, $\bT^{(a,n)}\in\bN$, $\bar{q}^{(a,n)}\le\bp_n$, $\sigma_n(\bar{q}^{(a,n)},\bT^{(a,n)})=q^{(a,n)},T^{(a,n)}$ and $\sigma_n(\overline{l(a)})=l(a)$.

First we find the requisite $u$ and $\bar{u}$. By elementarity of $\sigma_n$, $\bar{q}^{(a, n)}$ is a mixture of $\bar{T}^{(a, n)}$ up to $\overline{j}$.% Recall that $\sigma_n:\kla{L_{\btau}[\bA][\bG_{\bar{i}}],\in,\bA}\prec\kla{L_\tau[A][G_i],\in,A}$, so in particular, $\bsigma:=\sigma_n\rest L_\btau[\bA]:\kla{L_\btau[\bA],\in,\bA}\prec\kla{L_\tau[A],\in,A}$, and $\bsigma(\bar{q}^{(a,n)},\bT^{(a,n)})=q^{(a,n)},T^{(a,n)}$. Clearly, in $L_\tau[A]$, it is true that $q^{(a,n)}$ is a mixture of $T^{(a,n)}$ up to $j$, so it is true in $L_\btau[\bA]$ that $\bar{q}^{(a,n)}$ is a mixture of $\bT^{(a,n)}$ up to $\bar{j}$, and by absoluteness, this it true in $\V$ as well.

Let $\bT_0^{(a,n)}=\{\ba_0\}$. Let's write $\bG_{\bar{j}}=\bG_{\bar{i}}*\bG_{\bar{i},\bar{j}}$, and for $k\le\bar{j}$, let's set $\bG_k=\bG_{\bar{j}}\rest k$.
Since $\bar{T}_0^{(a,n)}$ is a nested antichain in $\bar{\P}\rest(\bar{j}+1)$, $\bar{j}<j+1$ and $\bar{q}^{(a,n)}\in\bar{\P}_{\bar{j}}$,
Fact \ref{fact:CharacterizationOfMixtures}.\ref{item:Root} applies, yielding that $\ba_0\equiv\bar{q}^{(a,n)}\rest l(\ba_0)\in\bG_{l(\ba_0)}$, since $l(\bar{q}^{(a,n)})=\bar{j}>l(\ba_0)$. So $\ba_0\in\bG_{l(\ba_0)}$.
Let $\bar{r}\in\bT_1^{(a,n)}$, that is, $\bar{r}\in\suc_{\bT^{(a,n)}}^0(\ba_0)$, be such that $\bar{r}\rest l(\ba_0)\in\bar{G}_{l(\ba_0)}$. There is such an $\bar{r}$ by Definition \ref{definition:NestedAC-Mixture-beta-Nice-hooks}(5). By Fact \ref{fact:CharacterizationOfMixtures}.\ref{item:HigherUp}, again since $l(\bar{r})<\bar{j}=l(\bar{q}^{(a,n)})$, it follows that $\bar{r}\equiv\bar{r}\rest l(\ba_0)\verl\bar{q}^{(a,n)}\rest[l(\ba_0),l(\bar{r}))$.
Since $\bar{r}\rest l(\ba_0)\in\bG_{l(\ba_0)}$, this implies that $\bar{r}\rest[l(\ba_0),l(\bar{r}))\equiv\bar{q}^{(a,n)}\rest[l(\ba_0),l(\bar{r}))$ (in the partial order $\bar{\P}_{l(\ba_0),l(\bar{r})}=\bar{\P}_{l(\bar{r})}/\bG_{l(\ba_0)}$), and $\bar{q}^{(a,n)}\rest[l(\ba_0),l(\bar{r}))\in\bG\rest[l(\ba_0),l(\bar{r}))$.
So we have that $\bar{r}\rest l(\ba_0)\in\bG_{l(\ba_0)}$ and $\bar{r}\rest[l(\ba_0),l(\bar{r}))\in\bG\rest[l(\ba_0),l(\bar{r}))$. By Fact \ref{fact:FactorsAndQuotients}.\ref{item:CriterionForBeingInGeneric}, this implies that $\bar{r}\in\bG_{l(\bar{r})}$.

It follows that $\overline{r}\verl\overline{q}^{(a, n)}\rest[l(\bar{r}), \bar{j})\in\bG_{\bar{j}}$, again using Fact \ref{fact:FactorsAndQuotients}\ref{item:CriterionForBeingInGeneric}.
Let $\overline{u} \in \bG_{\bar{j}}$ strengthen both $\overline{r}\verl\overline{q}^{(a, n)}\rest[l(\overline{r}),\bar{j})$ and $\overline{p}_{n+1}$.
By Lemma \ref{2.11}, applied in $\bar{N}$, there is a nested antichain $\bar{S} \hooks \bar{T}^{(a, n)}$ such that $\bar{u}$ is a mixture of $\bS$ up to $\bar{j}$ and such that letting $\bS_0 = \{\bd_0\}$, we have that  $l(\bar{r}) \leq l(\bd_0)$ and $\bd_0 \rest l(\bar{r}) \leq \bar{r}$. Let $S,d_0,u=\sigma_n(\bS,\bd_0,\bu)$, and let $w\in G_{l(a)}$ force this. Since $a\in G_{l(a)}$, we may choose $w$ so that $w\le a$.

Note that $S,d_0,u$ are in $N$ (and hence in $\V$), since $\bS,\bd_0,\bu\in\bN$.

We are going to  apply our inductive hypothesis $\phi(l(d_0))$, noting that $l(d_0)<j$, to $i =l(a) \leq l(d_0)$, the filters $\bG_{l(\overline{a})}$, $\bG_{{l(\overline{a})},l(\bd_0)}$, the models $\bN$, $N$, the condition $w$ (in place of $p$), the name $\dot{\sigma}^{(a,n)}$ (in place of $\dot{\sigma}$ and the parameter $\bs'\in\bN$ which we will specify below).
No matter which $\bs'$ we choose, by the inductive hypothesis, there is a condition $w^*\in\P_{l(d_0)}$ with $w^*\rest l(a)=w$ and a name
$\dot{\sigma}'$ such that $w^*$ forces with respect to $\P_{l(d_0)}$:
%such that if $H\ni w^*$ is generic for $\P_{l(d_0)}$, then in $\V[H]$ there is a $\sigma'$ such that letting $\sigma'_n=(\sigma^{(a,n)})^H$,
\begin{enumerate}[label=(\alph*)]
  \item
  $\dot{\sigma}'(\check{\bs}',\check{\bar{\vec{\P}}},\check{\overline{l(a)}},\check{l(\overline{d}_0)},\check{\btheta},\check{\bG}_{\overline{l(a)}})=\dot{\sigma}^{(a,n)}(\check{\bs}',\check{\bar{\vec{\P}}},\check{\overline{l(a)}},\check{l(\overline{d}_0)},\check{\btheta},\check{\bG}_{\overline{l({a})}})$.
  \item
  $(\dot{\sigma}')``\check{\bar{G}}_{\overline{l(a)},l(\bd_0)}\sub \dot{G}_{l(a),l(d_0)}$.
  \item
  $\dot{\sigma}':\check{\bN}[\check{\bG}_{\overline{l(a)}}]\prec \check{N}[\dot{G}_{l(a)}]$.
\end{enumerate}

By choosing $\bs'$ appropriately, we may insure that it is forced that $\dot{\sigma}'$ moves any finite number of members of $\bN$ the same way $\dot{\sigma}^{(a,n)}$ does.
Thus, we may insist that $w^*$ forces that $\dot{\sigma}'(\bu,\bd_0,\bS)=\dot{\sigma}^{(a,n)}(\bu,\bd_0,\bS)$. Recall that $w$ forced that $\dot{\sigma}^{(a,n)}(\bu,\bd_0,\bS)=u,d_0,S$.
Hence, since $w^*\rest l(a)=w$, we get that $w^*$ forces that $\dot{\sigma}'(\bu,\bd_0,\bS)=u,d_0,S$ as well.

In addition, we may insist that $\sigma'$ moves the parameters $\bar{i}$, $\bar{j}$, $\bar{\vec{\P}}$, $\btheta$, $\bs$, $\bp_0$, $\ldots$, $\bp_n$, $t_0$, $\ldots$, $t_n$ the same way $\dot{\sigma}^{(a,n)}$ does. Note that already $a$ forced with respect to $\P_{l(a)}$ that $\bar{i},\bar{j},\bar{\vec{\P}},\btheta$ are mapped to $i,j,\vec{\P},\theta$ by $\dot{\sigma}^{(a,n)}$.

%Note that $\sigma'$ lifts to an elementary embedding $\sigma':\bN[\bG_{l(\bd_0)}]\prec N[H]$ with $\sigma'(\bG_{l(\bd_0)})=H_{l(d_0)}$. It follows that $\sigma'(\bu\rest l(\bd_0))=u\rest l(d_0)\in H$, since $\bu\in\bG_{\bar{j}}$, so that $\bu\rest l(\bd_0)\in\bG_{l(\bd_0)}$.

%Since $\bu\in\bG$, it follows that $\bu\rest l(d_0)\in\bG\rest l(\bd_0)$, and so, $w^*$ forces that $u\rest l(d_0)=\sigma'(\bu\rest l(\bd_0))\in\dot{G}_{l(d_0)}$. Thus, given any $P_{l(a),l(d_0)}$-generic filter $H$ containing $w^*$, we may find a common extension of $w^*$ and $u\rest l(d_0)$ in $H$. So we may assume that $w^*\le u\rest l(d_0)$.

Now, setting $b=w^*$, $\dot{\sigma}^b=\dot{\sigma}'$, conditions \ref{item:FirstConditionDefiningPredenseSet1}-\ref{item:LastConditionDefiningPredenseSet1} are satisfied, that is, $b\in D$. Most of these are obvious; let us just remark that $b$ forces that $\dot{\sigma}^b(\bG_{\bar{i}})=\dot{G}_i$ because it forces that $\dot{\sigma}^b(\bar{i})=i$ and $\dot{\sigma}^b(\bG_{l(\overline{a})})=\dot{G}_{l(a)}$.
Condition \ref{item:EnsuringGettingAFusionStructure} holds because $l(b)=l(w^*)=l(d_0)$ and $d_0=\Root(S)$. To see that $b=w^*\le u\rest l(d_0)$, note that $w^*$ forces with respect to $\P_{l(d_0)}$
that $\dot{\sigma}'``\bar{G}_{l(\bar{d}_0)}\sub\dot{G}_{l(b)}$ and
that $\dot{\sigma}'(\bar{u}\rest l(\bar{d}_0))=u\rest l(d_0)$. Since $\bar{u}\rest l(\bar{d}_0)\in\bar{G}_{l(\bar{d}_0)}$, it follows that $b$ forces that $u\rest l(d_0)\in\dot{G}_{l(d_0)}$. This means that $b\le_{l(b)}u\rest l(d_0)$, as claimed. Condition \ref{item:InitialSegmentOfbForcesStuffIntoRangeOfTheOldEmbedding1} holds because $b\rest l(a)=w$.
For the same reason, we have that $b\rest l(a)\in G_{l(a)}$, completing the proof that $D\rest l(a)$ is predense below $a$. This concludes the treatment of case 1.

\medskip

\noindent\emph{Case 2:} $j$ is a successor ordinal.%%starthere

Let $j=k+1$. Since we assumed $i<j$, it follows that $i\le k$. Inductively, we know that $\phi(k)$ holds. Note that $\bar{j}$ is of the form $\bk+1$, where $p$ forces with respect to $\P_i$ that $\dot{\sigma}(\bk)=k$, and if we let $\bG_{\bk}=\bG_\bj\rest\bk$, then the assumptions \ref{item:FirstAssumption1}-\ref{item:Lastassumption} are satisfied by
$p\in \P_i$, $\dot{\sigma}\in\V^{\P_i}$, %$x\in P_j$, $y\le x\rest i$,
$\theta$, $\tau$, $A$, $N$, $\bN$, $\bs,\bar{\vec{\P}},\bar{i},\bar{k}\in\bN$, $\bar{G}_{\bar{i}},\bar{G}_{\bar{i},\bar{k}}\sub\bN$ and $k$. By $\phi(k)$, we obtain a condition $p^{**}\in\P_k$ with $p^{**}\rest i=p$ and a $\P_k$-name $\dot{\bsigma}$ such that $p^{**}$ forces with respect to $\P_k$:
\begin{enumerate}[label=(\alph*1)]
  \item
  $\dot{\bsigma}(\check{\bs},\check{\bar{\vec{\P}}},\check{\bar{i}},\check{\bk},\check{\bar{j}},\check{\btheta},\check{\bG}_{\bar{i}})=\dot{\sigma}(\check{\bs},\check{\bar{\vec{\P}}},\check{\bar{i}},\check{\bk},\check{\bar{j}},\check{\btheta},\check{\bG}_{\bar{i}})$.
  \item
  $\dot{\bsigma}``\check{\bar{G}}_{\bar{i},\bar{k}}\sub\dot{G}_{i,k}$.
  \item
  $\dot{\bsigma}:\check{\bN}[\check{\bG}_{\bar{i}}]\prec \check{N}[\dot{G}_i]$.
\end{enumerate}
It follows then that $p^{**}$ forces that $\dot{\bsigma}``\bG_\bk\sub\dot{G}_k$, and hence that $\dot{\bsigma}$ lifts to an elementary embedding from $\bN[\bG_\bk]\prec N[\dot{G}_k]$ that maps $\bG_\bk$ to $\dot{G}_k$. Let $\dot{\tsigma}$ be a $\P_k$-name such that $p^{**}$ forces that $\dot{\tsigma}$ is that lifted embedding.

Temporarily fix a $\P_k$-generic filter $H$ that contains $p^{**}$. In $\V[H]$, the forcing $\P_{k,k+1}=\P_{k,j}=\P_j/H$ is $\infty$-subcomplete. Note that the structure $N[H]=L_\tau^A[H]$ can be thought of as having the form $L_\tau^{A\oplus H}$, and similarly, $\bar{N}[\bar{G}_{\bar{k}}]$ can be thought of as having the form $L_{\bar{\tau}}^{\bar{A}\oplus\bar{G}_{\bar{k}}}$, where $x\oplus y=x\times\{0\}\cup y\times\{1\}$, say. Then, $\bar{N}[\bar{G}_{\bar{k}}]$ is full and has the right form, and letting $\tsigma=\dot{\tsigma}^H$, we have that $\tsigma:\bN[\bG_\bk]\prec N[H]$, and thus, there is a condition $q$ in $\P_{k,j}$ such that $q$ forces the existence of an elementary embedding $\sigma'$ with
\begin{enumerate}[label=(\alph*2)]
  \item
  $\sigma'(\check{\bs},\check{\bar{\vec{\P}}},\check{\bar{i}},\check{\bk},\check{\bar{j}},\check{\btheta},\check{\bG}_{\bar{i}},\check{\bG}_{\bk})=\tsigma(\check{\bs},\check{\bar{\vec{\P}}},\check{\bar{i}},\check{\bk},\check{\bar{j}},\check{\btheta},\check{\bG}_{\bar{i}},\check{\bG}_{\bk})$.
  \item
  $(\sigma')``\check{\bar{G}}_{\bar{k},\bar{j}}\sub\dot{G}_{k,j}$.
  \item
  $\sigma':\check{\bN}[\check{\bG}_{\bar{k}}]\prec \check{N}[\check{H}]$.
\end{enumerate}
Since this holds in $\V[H]$ whenever $p^{**}\in H$, there is a $\P_k$-name $\tau$ which is essentially a name for $q$ above - more precisely, $\tau$ is such that $p^{**}$ forces that $\tau\in \P_j$, $\tau\rest k\in\dot{G}_k$ and $\tau\rest[k,j)$ has the properties of $q$, as listed above.
Since the iteration is nice, there is a condition $p^*\in \P_j$ such that $p^*\rest k=p^{**}$ and $p^*$ forces that $\tau\rest[k,j)\equiv p^*\rest[k,j)$; see Definition \ref{def:NiceIterations}, part (1). We claim that $p^*$ is as wished.

First, note that $p^*\rest i=(p^*\rest k)\rest i=p^{**}\rest i=p$. Now, let $G_j$ be a $\P_j$-generic filter with $p^*\in G_j$. We have to show that in $\V[G_j]$, there is a $\sigma'$ such that, letting $\sigma=\dot{\sigma}^{G_i}$, the following hold:
\begin{enumerate}[label=(\alph*)]
  \item
  $\sigma'(\bs,\bar{\vec{\P}},\bar{i},\bar{j},\btheta,\bG_{\bar{i}})=\sigma(\bs,\bar{\vec{\P}},\bar{i},\bar{j},\btheta,\bG_{\bar{i}})$.
  \item
  $(\sigma')``\bar{G}_{\bar{i},\bar{j}}\sub G_{i,j}$.
  \item
  $\sigma':\bN[\bG_{\bar{i}}]\prec N[G_i]$.
\end{enumerate}
But this follows, because $\V[G_j]=\V[G_k][G_{k,j}]$, $p^*\rest[k,j)\in G_{k,j}$ and $p^{**}\in G_k$. So writing $H$ for $G_k$ puts us in the situation described above; $p^*\rest[k,j)\equiv\dot{q}^{G_j}$, where $\dot{q}$ is a name for the condition $q$ mentioned above. Thus, there is a $\sigma'$ in $\V[G_j]$ such that the conditions (a2)-(c2) listed above hold in $\V[G_j]$. Remembering that $\tsigma$ lifts $\bsigma$ and $\bsigma$ moves the required parameters as prescribed (by (a1)-(c1)), it follows that (a)-(c) are satisfied.
\end{proof}

\subsection{Nice iterations of $\infty$-subproper forcing}

Next we prove a similar theorem for $\infty$-subproper forcing. Even though it was originally shown by Miyamoto, we present here a complete proof we found before learning about his result. We will use this proof as a template later.

\begin{thm}[{Miyamoto \cite{Miyamoto:IteratingPreproperness}}]
\label{thm:NiceIterationsOfSPforcingAreSP}
Let $\vec{\P} = \langle \P_\alpha \; | \; \alpha \leq \nu\rangle$ be a nice iteration so that $\P_0 = \{1_0\}$ and for all $i<\nu$, $\forces_i \P_{i, i+1}$ is $\infty$-subproper. Then for all $j\le\nu$ the following statement $\phi(j)$ holds:

{\bf if} $i\le j$, $p\in \P_i$, $\dot{\sigma}\in\V^{\P_i}$, $q\in\P_j$ is such that $p\le_i q\rest i$,
$\theta$ is a sufficiently large cardinal, $\tau$ is an ordinal, $H_\theta\sub N=L_\tau^A\models\ZFC^-$, $s$ is an element of $N$,
$\bN$ is a countable, full, transitive model so that $\bs,\bar{\vec{\P}},\bar{i},\bar{j},\bar{q}\in\bN$ and $p$ forces with respect to $\P_i$ that the following assumptions hold:
\begin{enumerate}[label=\textnormal{(A\arabic*)}]
  \item
  \label{item:LastAssumptionSP}
  $\dot{\sigma}:\check{\bN}\prec \check{N}$.
  \item
  \label{item:FirstAssumptionSP}
  $\dot{\sigma}(\check{\bs}, \check{\bar{\vec{\P}}},\check{\bar{i}},\check{\bar{j}},\check{\btheta}, \check{\bar{q}})=\check{s}, \check{\vec{\P}},\check{i},\check{j},\check{q}$.
  \item
  \label{item:Generics1}
   $\dot{\sigma}^{-1}``\dot{G_i}$ is $\bar{\P}_{\bar{i}}$-generic over $\check{\bN}$.
\end{enumerate}
{\bf then} there is a condition $p^*\in\P_j$ with $p^*\le_j q$ and
$p^*\rest i=p$ such that whenever $G_j\ni p^*$ is $\P_j$-generic, we have that in $\V[G_j]$, there is a $\sigma'$ such that, letting $\sigma=\dot{\sigma}^{G_i}$% and $\bar{G}_{\bar{i}}=\sigma^{-1}``G_i$
, the following conclusions hold:
\begin{enumerate}[label=\textnormal{(C\arabic*)}]
  \item
  \label{item:LastConclusionSP}
  $\sigma':\bN\prec N$.
  \item
  \label{item:FirstConclusionSP}
  $\sigma'(\bs,\bar{\vec{\P}},\bar{i},\bar{j},\btheta, \bar{q})=s,\vec{\P},i,j,\theta,q$.
  \item
  \label{item:Sigma'MovesEverythingCorrectly2}
  $(\sigma')^{-1}``G_i=\sigma^{-1}``G_i$ and $(\sigma')^{-1}``G_j$ is $\bar{\P}_{\bar{j}}$-generic over $\bN$.
\end{enumerate}
\end{thm}

%\begin{remark}
%\label{rem:SCwrtLiftings}
%Note that in the situation of the theorem, in $\V[G_i]$, letting $\sigma=\dot{\sigma}^{G_i}$ and $\bar{G}_{\bar{i}}=\sigma^{-1}``G_i$, $\sigma$ canonically lifts to an elementary embedding $\sigma^*:\bN[\bar{G}_{\bar{i}}]\prec N[G_i]$ with $\sigma^*(\bar{G}_{\bar{i}})=G_i$, and similarly, $\sigma'$ lifts to an elementary embedding $(\sigma')^*:\bN[\bG_{\bar{i}}]\prec N[G_i]$ with $(\sigma')^*(\bar{G}_{\bar{i}})=G_i$. Since $\bs\in\bN$, $\bs$ could be chosen to be a $\P_i$-name for any desired tuple $\tau_0,\ldots,\tau_{n-1}$ of $\bar{\P}_{\bar{i}}$-names in $\bN$. It would then follow that $\sigma^*(\tau_i^{\bG_{\bar{i}}})=\sigma(\tau_i)^{G_i}=\sigma'(\tau_i)^{G_i}=(\sigma')^*(\tau_i^{\bG_i})$. Also, letting $\bar{G}_{\bar{j}}=(\sigma')^{-1}``G_j$, $\sigma'$ even lifts further, to an elementary embedding from $\bN[\bG_{\bar{j}}]$ to $N[G_j]$.
%
%As a result, the theorem shows that $\P_{i,j}$ is $\infty$-subproper in $\V[G_i]$ ``with respect to liftings of embeddings from $\bN$ to $N$ (where $\bN$, $N\in\V$).''
%\end{remark}

\begin{proof}
%In a slight abuse of notation, let us identify $\sigma$ with $\sigma^*$, and let us fix a canonical $\P_i$-name $\dot{\bar{G}}_{\bar{i}}$ for $\dot{\sigma}^{-1}``\dot{G}_i$ (for which we also could write $\dot{\sigma}^{-1}(\dot{G}_i)$, using this identification).
Like in the previous proof, we induct on $j$. So let us assume that $\phi(j')$ holds for every $j'<j$.
%Fixing $j$, we prove the claim by induction on $i$.
Fix some $i\le j$.
%assuming the theorem is proven for all $i'<i$ (and also for all $i'$, $j'$ with $i'\le j'<j$).
Since nothing is to be shown when $i=j$, let $i<j$. In particular, the case $j=0$ is trivial.

Let us fix $p\in \P_i$, $\dot{\sigma}\in\V^{\P_i}$, $q\in\P_j$ with $p\le_i q\rest i$. Also fix $\theta$, $\tau$, $A$, $N$, $\bN$, $\bs,\bar{\vec{\P}},\bar{i},\bar{j}\in\bN$ so that assumptions \ref{item:LastAssumptionSP}-\ref{item:Generics1} hold.

When it causes no confusion, we will again employ our notational ``bar'' convention.

\medskip

\noindent\emph{Case 1:} $j$ is a limit ordinal.

Let $\seq{t_n}{n < \omega}$ enumerate the elements of $\bN$, starting with $t_0=\leer$. Also let $\seq{\bar{D}_n}{n<\omega}$ enumerate the dense open subsets of $\bar{\P}_{\bar{j}}$ in $\bN$ in such a way that $\bar{q}\in\bar{D}_0$.

As before, there is an $\bN$-definable well ordering of the universe of $\bN$, $<_\bN$.
As noted in \cite{Miyamoto:IteratingSemiproperPreorders}, by Lemma \ref{2.7}, in $\bN$, $\bar{q}$ is a mixture up to $\bar{j}$ of some nested antichain in $\bar{\vec{\P}}\rest{\bar{j}}$ whose root has length $\bar{i}$. Letting $\bar{W}$ be the $<_\bN$-least one, we know that $p$ forces that $\dot{\sigma}(\check{\bar{W}})$ is the $L_\tau^A$-least nested antichain $W$ in $\vec{\P}\rest j$ such that $q$ is a mixture of $W$ up to $j$, since we know that $p$ forces that $\bar{q},\bar{i},\bar{j}$ are mapped to $q,i,j$ by $\dot{\sigma}$, respectively.

We will define a nested antichain $\langle T, \langle T_n \; | \; n < \omega \rangle, \langle \suc^n_T \; | \; n < \omega\rangle \rangle$, a fusion structure $\seq{\kla{q^{(a, n)}, T^{(a, n)}}}{n < \omega, a\in T_n}$ in $\seq{\P_\alpha}{\alpha \leq j}$ and a sequence $\langle \dot{\sigma}^{(a, n)} \; | \; n < \omega, a \in T_n\rangle$ so that the following conditions hold.

\begin{enumerate}[label=(\arabic*)]
\item
\label{claim:Start2}
$T_0 = \{p\}$, $q^{(p,0)}=q$, $T^{(p,0)}=W$ and $\dot{\sigma}^{(p,0)} = \dot{\sigma}$.
\end{enumerate}
Further, for any $n<\omega$ and $a\in T_n$:
\begin{enumerate}[label=(\arabic*)]
\setcounter{enumi}{1}
\item %6
\label{item:middlepart-beginning2}
$q^{(a, n)} \leq q$,
$q^{(a, n)}\in \P_j$ and $\dot{\sigma}^{(a, n)}$ is a $\P_{l(a)}$-name.
\item %7
\label{item:sigma(a,n)2}
$a$ forces the following statements with respect to $\P_{l(a)}$:
\begin{enumerate}[label=(\alph*)]
\item
\label{subitem:iselementary2}
$\dot{\sigma}^{(a, n)} : \check{\bN}\prec \check{N}$.
\item
\label{subitem:moveseverythingcorrectly2}
$\dot{\sigma}^{(a, n)}(\check{\overline{\theta}}, \check{\bar{\vec{\P}}}, \overline{s}, \check{\bar{i}},\check{\bar{j}}, \check{\bar{q}}) = \check{\theta}, \check{\vec{\P}}, \check{s}, \check{i}, \check{j}, \check{q}$.
%\item ${\mathrm{Hull}}^N(\delta(P_\nu)\cup\ran(\dot{\sigma}^{(a, n)})) = {\mathrm{Hull}}^N(\delta(P_\nu)\cup\ran(\check{\sigma}))$.
\item
\label{subitem:hasx(a,n)initsrange2}
 $q^{(a,n)}\in\ran(\dot{\sigma}^{(a,n)})$ and its preimage is in $\overline{D}_n$.
\end{enumerate}
%\item %8
%$a$ decides $\dot{\sigma}^{(a, n)}(t_n)$.
\item %9
\label{item:middlepart-end2}
For some $\bar{q}^{(a,n)}\in\bar{D}_n$, $\bar{T}^{(a,n)}\in\bN$ and ordinal $\overline{l(a)}$, we have that %$\overline{q}^{(a, n)} \in \bG_{\bar{i}}*\bar{G}_{\bar{i},\bar{j}}$, $\overline{q}^{(a, n)} \in \overline{D}_n$ and \\
\[a \forces \dot{\sigma}^{(a, n)} (\kla{\bar{q}^{(a, n)}, \bar{T}^{(a, n)},\overline{l(a)}}) =\kla{q^{(a, n)}, T^{(a, n)}, l(a)}\]
and $(\dot{\sigma}^{(a,n)})^{-1}``\dot{G}_{l(a)}$ is $\bar{\P}_{\overline{l(a)}}$-generic over $\bN$.
\end{enumerate}
\noindent For $b \in {\rm suc}^n_T (a)$:
\begin{enumerate}[label=(\arabic*)]
\setcounter{enumi}{4}
\item %10
\label{item:coherence2}
$b$ forces with respect to $\P_{l(b)}$ that \[(\dot{\sigma}^{(b,n+1)})^{-1}``\dot{G}_{l(a)}=(\dot{\sigma}^{(a,n)})^{-1}``\dot{G}_{l(a)}.\]
Further, if $m\le n$, then
\[b \forces_{l(b)} \dot{\sigma}^{(b, n+1)}(\check{t}_m) = \dot{\sigma}^{(a, n)}(\check{t}_m).\] %[{\bf we also asked that $\dot{\sigma}^{(b, n+1)}(\overline{D}_m) = \dot{\sigma}^{(a, n)}(\overline{D}_m)$. But that doesn't seem to be needed anywhere.}]
%But I think we want:
Also, let $<_T$ be the transitive closure of the order $<'_T$ on $\{\kla{c,k}\st k<\omega\land c\in T_k\}$ defined by $\kla{c,k}<'_T\kla{d,l}$ iff $l=k+1$ and $d\in\suc_T^k(c)$. Then, if $\kla{c,m}\le_T\kla{a,n}$, we have that \[b\forces\dot{\sigma}^{(b,n+1)}(\bar{q}^{(c,m)})=\dot{\sigma}^{(a,n)}(\bar{q}^{(c,m)}).\]
%\ \text{and}\ \dot{\sigma}^{(b, n+1)}(v_m) = \dot{\sigma}^{(a, n)}(v_m),$ where $v_m$ is the $\bN$-least $v$ with $\sigma(t_m) \in \dot{\sigma}^{(a_m, m)}(v)$ and $\card{v}^{\bN} \leq \overline{\delta(P_\nu)}$, where $a_m$ is the predecessor of $a$ in $T_m$. Note that such a $v$ always exists by Fact \ref{ufact}.
\item %11
\label{claim:LastOfFirstList2}
$b\rest l(a)\forces_{l(a)}q^{(b,n+1)},T^{(b,n+1)}\in\ran(\dot{\sigma}^{(a,n)})$.
\end{enumerate}
First let's see that constructing such objects is sufficient to prove the existence of a condition $p^*$ as in the statement of the theorem. Suppose that we have constructed sequences satisfying \ref{claim:Start2} through \ref{claim:LastOfFirstList2} above. Let $q^*\in \P_j$ be a fusion of the fusion structure, and let $p^*=p\verl q^*\rest[i,j)$. By \ref{claim:Start2}, we have that $q^*\rest i\equiv p$, so $p^*\equiv q^*$ and $p^*\rest i=p$, as required.
To see that $p^*$ is as wished, let $G_j$ be $\P_j$-generic over $\V$ with $p^* \in G_j$. Let $\sigma=\dot{\sigma}^{G_j}$ and $\bar{G}_{\bar{i}}=\sigma^{-1}``G_i$. We have to show that in $\V[G_j]$, there is a $\sigma'$ so that conclusions \ref{item:LastConclusionSP}-\ref{item:Sigma'MovesEverythingCorrectly2} are satisfied.
Since $p^*\equiv q^*$, we have that $q^*\in G_j$. Work in $\V[G_j]$. By Proposition \ref{fusion} there is a sequence $\langle a_n \; | \; n < \omega\rangle \in V[G_j]$ so that $a_0\in T_0$ (so $a_0=p$) and for all $n < \omega$, $a_{n+1} \in {\rm suc}^n_T(a_n)$, $a_n \in G_j \rest l(a_n)$ and $q^{(a_n, n)} \in G_j$. Let $\sigma_n$ be the evaluation of $\dot{\sigma}^{(a_n, n)}$ by $G_j$. So $\sigma_n:\bN\prec N$. Define $\sigma':\bN\prec N$ by letting $\sigma'(t_n)=\sigma_n(t_n)$. It follows as before that $\sigma'$ is elementary, so that conclusion \ref{item:LastConclusionSP} is satisfied.
%Notice that $\sigma_n\rest\bar{N}:\bar{N}\prec N$. It follows that for any $\mu\in\bN^{\bP_{\bar{i}}}$, $\sigma_n(\mu^{\bar{G}_{\bar{i}}})=\sigma_n(\mu)^{G_i}$.
%Hence, we can define $\sigma': \bN[\bar{G}_{\bar{i}}]\prec N[G_i]$ by setting $\sigma'(t_n^{\bar{G}_{\bar{i}}})=\sigma_n(t_n)^{G_i}$, if $t_n\in\bN^{\bP_{\bar{i}}}$. Both well-definedness and elementarity can be shown as follows: suppose $\bN[\bar{G}_{\bar{i}}]\models\psi(t_{n_0}^{\bG_{\bar{i}}},\ldots,t_{n_{k-1}}^{\bG_{\bar{i}}})$, where $t_{n_0},\ldots,t_{n_{k-1}}\in\bN^{\bP_{\bar{i}}}$. Let $m=\max\{n_0,\ldots,n_{k-1}\}$. Then $N[G_i]\models\psi(\sigma_m(t_{n_0}^{\bG_{\bar{i}}}),\ldots,\sigma_m(t_{n_{k-1}}^{\bG_{\bar{i}}}))$.
%By the previous remark, $\sigma_m(t_{n_0}^{\bG_{\bar{i}}})=\sigma_m(t_{n_0})^{G_i}, \ldots, \sigma_m(t_{n_{k-1}}^{\bG_{\bar{i}}})=\sigma_m(t_{n_{k-1}})^{G_i}$. By coherence, we have $\sigma_m(t_{n_0})=\sigma_{n_0}(t_{n_0}),\ldots, \sigma_m(t_{n_{k-1}})=\sigma_{n_{k-1}}(t_{n_{k-1}})$, so we obtain that $N[G_i]\models\psi(\sigma_{n_0}(t_{n_0})^{G_i},\ldots,\sigma_{n_{k-1}}(t_{n_{k-1}})^{G_i})$.
%But by definition of $\sigma'$, this is the same as to say that $N[G_i]\models\psi(\sigma'(t_{n_0}^{\bar{G}_{\bar{i}}}),\ldots,\sigma'(t_{n_{k-1}}^{\bG_{\bar{i}}}))$.
Conclusion \ref{item:FirstConclusionSP}, stating that $\sigma'$ moves the given parameters the same way $\sigma$ does, follows since each $\sigma_n$ moves them that way.
%Note that $a_n$ decides $\dot{\sigma}^{(a, n)}(t_n)$, and this is exactly the value given in the extension.
The crucial claim we have to verify is condition \ref{item:Sigma'MovesEverythingCorrectly2}, which has two parts.

The first part states that $(\sigma')^{-1}``G_i=\sigma^{-1}``G_i$. Let $\delta_n=l(a_n)$. By construction, we have that for $m<n<\omega$, $\delta_m\in\ran(\sigma_m)$ and $\bar{\delta}_m=\sigma_m^{-1}(\delta_m)=\sigma_n^{-1}(\delta_m)$.
We also have that $(\sigma_{n+1})^{-1}``G_{\delta_n}=(\sigma_n^{-1})``G_{\delta_n}$. It follows by induction on $n$ that $(\sigma_n)^{-1}``G_{\delta_0}=(\sigma_0)^{-1}``G_{\delta_0}$. The base case is trivial, and for the inductive step, we use that by construction, we have that $(\sigma_{n+1})^{-1}``G_{\delta_n}=(\sigma_n)^{-1}``G_{\delta_n}$.

This implies that $(\sigma_{n+1})^{-1}``G_{\delta_0}=(\sigma_n)^{-1}``G_{\delta_0}$:
for $x\in\P_{\delta_0}$, let
\[x'=x\verl 1_{\delta_{n}}\rest[\delta_0,\delta_{n}).\]
For $\bar{x}\in\bar{\P}_{\bar{\delta}_0}$, let $\bar{x}'$ be defined similarly, with $\bar{\delta}_0$, $\bar{\delta}_{n}$ in place of $\delta_0$, $\delta_{n}$.
Then $x\in G_{\delta_0}$ iff $x'\in G_{\delta_{n}}$. To prove the claimed identity, let $\bar{x}\in(\sigma_{n+1})^{-1}``G_{\delta_0}$. So $x=\sigma_{n+1}(\bar{x})\in G_{\delta_0}$. Then $x'\in G_{\delta_n}$, and $\sigma_{n+1}(\bar{x}')=x'$. So $\bar{x}'\in(\sigma_{n+1})^{-1}``G_{\delta_n}$. Hence, $\bar{x}'\in(\sigma_n)^{-1}``G_{\delta_n}$, that is, $\sigma_n(\bar{x}')\in G_{\delta_n}$. But $\sigma_n(\bar{x}')=\sigma_n(\bar{x})'$. So $\sigma_n(\bar{x})'\in G_{\delta_n}$, which implies that $\sigma_n(\bar{x})\in G_{\delta_0}$. So $\bar{x}\in(\sigma_n)^{-1}``G_{\delta_0}$. For the converse, let $\bar{x}\in(\sigma_n)^{-1}``G_{\delta_0}$. So $\bar{x}'\in\sigma_n^{-1}``G_{\delta_n}=\sigma_{n+1}^{-1}``G_{\delta_n}$. So $\sigma_{n+1}(\bar{x}')=\sigma_{n+1}(\bar{x})'\in G_{\delta_n}$, so $\sigma_{n+1}(\bar{x})\in G_{\delta_0}$, that is, $\bar{x}\in\sigma_{n+1}^{-1}``G_{\delta_0}$.

Inductively, we have that $\sigma_n^{-1}``G_{\delta_0}=\sigma_0^{-1}``G_{\delta_0}$, which together with the identity shown in the previous paragraph yields that $(\sigma_{n+1})^{-1}``G_{\delta_0}=\sigma_0^{-1}``G_{\delta_0}$, completing the induction step.

Since $\delta_0=i$ and $\sigma_0=\sigma$, this means that $(\sigma_n)^{-1}``G_i=\sigma^{-1}``G_i$, for every $n$.

This easily implies that $(\sigma')^{-1}``G_i=\sigma^{-1}``G_i$. From left to right, suppose $\sigma'(\bar{x})\in G_i$. Say $\bar{x}=t_n$. So $\sigma_n(\bar{x})\in G_i$, i.e., $\bar{x}\in\sigma_n^{-1}``G_i=\sigma^{-1}``G_i$. Vice versa, suppose $\sigma(\bar{x})\in G_i$. Let $\bar{x}=t_n$. Then $\bar{x}\in\sigma_n^{-1}``G_i$, i.e., $\sigma'(\bar{x})=\sigma_n(\bar{x})\in G_i$. So $\bar{x}\in(\sigma')^{-1}``G_i$.

The second part of conclusion \ref{item:Sigma'MovesEverythingCorrectly2} is that $\bar{G}_{\bar{j}}=(\sigma')^{-1}``G_j$ is $\bar{\P}_{\bar{j}}$-generic over $\bar{N}$.
This is taken care of in the construction since the pre-image of $q^{(a_n,n)}$ under $\sigma'$ is in $\bar{D}_n$ by \ref{item:sigma(a,n)2}\ref{subitem:hasx(a,n)initsrange2}.
In detail, let $t_m=\bar{q}^{(a_n,n)}$. Then $\sigma'(\bar{q}^{(a_n,n)})=\sigma'(t_m)=\sigma_m(t_m)=\sigma_m(\bar{q}^{(a_n,n)})$.

If $m\le n$, then
\[\sigma'(\bar{q}^{(a_n,n)})
=\sigma_m(t_m)
=\sigma_n(t_m)
=\sigma_n((\bar{q}^{(a_n,n)})
=q^{(a_n,n)},\] so $(\sigma')^{-1}(q^{(a_n,n)})=\bar{q}^{(a_n,n)}\in\bar{D}_n$ in this case. And if $n<m$, then
\[\sigma'(\bar{q}^{(a_n,n)})=\sigma_m(t_m)
=\sigma_m(\bar{q}^{(a_n,n)})
=\sigma_n(\bar{q}^{(a_n,n)})
=q^{(a_n,n)},\]
so again, $(\sigma')^{-1}(q^{(a_n,n)})=\bar{q}^{(a_n,n)}\in\bar{D}_n$.

Thus it remains to see that such a construction can be carried out. This is done by induction on $n$, in a manner similar to the previous proof. Like last time, at stage $n+1$ of the construction, we
assume that $T_m$, $T^{(a,m)}$, $q^{(a,m)}$ and $\dot{\sigma}^{(a,m)}$ have been defined, for all $m\le n$ and all $a\in T_m$. Also, for $m<n$ and $a\in T_m$, we assume that $\suc_T^m(a)$ has been defined. Our inductive hypothesis is that for all $m\le n$ and all $a\in T_m$, conditions
\ref{item:middlepart-beginning2}-\ref{item:middlepart-end2}
hold, and that for all $m<n$, all $a\in T_m$ and all $b\in\suc_T^m(a)$, conditions \ref{item:coherence2}-\ref{claim:LastOfFirstList2} are satisfied. As before, for every $a\in T_n$, we will specify $\suc_T^n(a)$, and for every $b\in\suc_T^n(a)$, the objects $T^{(b,n+1)}$, $q^{(b,n+1)}$ and $\dot{\sigma}^{(b,n+1)}$ in such a way that  \ref{item:coherence2}-\ref{claim:LastOfFirstList2} are satisfied by $a$ and $b$, and \ref{item:middlepart-beginning2}-\ref{item:middlepart-end2} are satisfied by $b$ and $n+1$ (instead of $a$ and $n$). As before, in order to ensure that we ultimately produce a fusion structure, we again arrange that $q^{(b,n+1)}$ is a mixture of $T^{(b,n+1)}$ up to $j$, $b\le q^{(b,n+1)}\rest l(b)$, $l(b)=l(\Root(T^{(b,n+1)}))$ and $T^{(b,n+1)}\hooks T^{(a,n)}$, so that Definition \ref{def:FusionStructure}(2)-(4) hold (and we also assume that the corresponding statements are true of the earlier defined $q^{(c,k)}$ and $T^{(c,k)}$).

For stage 0 of the construction, notice that \ref{claim:Start2} gives the base case where $n=0$ and in this case \ref{item:middlepart-beginning2}-\ref{item:middlepart-end2} are satisfied, $p$ forces that $q^{(p, 0)}=q=\dot{\sigma}(\overline{q}) \in \overline{D}_0$ and $p$ forces that $W=T^{(p,0)}$ has a preimage under $\dot{\sigma}$, namely $\bar{W}$.

At stage $n+1$ of the construction, work under the assumptions described above. Fixing $a\in T_n$, we have to define $\suc_T^n(a)$.
To this end let $D$ be the set of all conditions $b\in\bigcup_{l(a)\le\xi<j}\P_\xi$ for which there are a nested antichain $S$ in $\vec{\P}\rest j$ and objects $\dot{\sigma}^b$, $u$, $\bar{u}$, $\overline{l(b)}$ and $\bar{S}$ satisfying the following:

\begin{enumerate}[label=(D\arabic*)]
\item %12
\label{item:FirstConditionDefiningPredenseSet2}
$b\rest l(a) \le a$.
%\item %13
%$l(a) \leq l(b)$ and .
\item %14
$S \hooks T^{(a, n)}$, $\bS\in\bN$, $S\in N$.
\item %15
$u\in \P_j$, $u \leq q^{(a,n)}$ and $u$ is a mixture of $S$ up to $j$.
\item
\label{item:EnsuringGettingAFusionStructure2}
$b\le_{l(b)}u\rest l(b)$ and $l(b)=l(\Root(S))$.
\item %16
%$\overline{u} \hook \bar{i} \in \bG_{\bar{i}}$, and
$\overline{u} \in \overline{D}_{n+1}$.
%\item %17
%$b$ forces that $\dot{\sigma}^b$ is a function.%, and $b$ decides $\dot{\sigma}^b(t_{n+1})$.
\item %18
\label{item:InitialSegmentOfbForcesStuffIntoRangeOfTheOldEmbedding2}
$b\rest l(a)\forces_{l(a)}\dot{\sigma}^{(a,n)}(\check{\bar{S}},\check{\bar{u}},\check{\overline{l(b)})}=\check{S},\check{u},(l(b))\check{}$.
%$b \forces_{l(b)} S, u, l(b) \in\ran (\dot{\sigma}^b)$.
\item %19
\label{item:LastConditionDefiningPredenseSet2}
$b$ forces the following statements with respect to $\P_{l(b)}$: \begin{enumerate}
\item $\dot{\sigma}^b: \check{\bN}\prec \check{N}$.
\item $\dot{\sigma}^b (\check{\overline{\theta}},\check{\bar{i}},\check{\bar{j}},\check{\bar{\vec{\P}}}, \check{\overline{s}},\check{\bar{u}},\check{\bS}, \check{\bar{q}},(\overline{l(b)})\check{}) = \check{\theta},\check{i},\check{j},\check{\vec{\P}},\dot{\sigma}(\bs), \check{u},\check{S}, \check{q},(l(b))\check{}$, and\\
    $\forall m \leq n\forall c\quad$ $\dot{\sigma}^b(t_m) = \dot{\sigma}^{(a, n)}(t_m)$ %and $\dot{\sigma}^b(\overline{D}_m) = \dot{\sigma}^{(a, n)}(\overline{D}_m)$,
    and \\ \hspace*{13ex}$\kla{c,m}\le_T\kla{a,n}\To\dot{\sigma}^b(\bar{q}^{(c,m)})=\dot{\sigma}^{(a,n)}(\bar{q}^{(c,m)})$.
%\item ${\mathrm{Hull}}^N(\delta(P_\nu) \cup \ran (\dot{\sigma}^b)) = {\mathrm{Hull}}^N(\delta(P_\nu) \cup \ran (\dot{\sigma}^{(a, n))})$,
\item $(\dot{\sigma}^b)^{-1}``\dot{G}_{l(a)}=(\dot{\sigma}^{(a,n)})^{-1}``\dot{G}_{l(a)}$ and\\
    $(\dot{\sigma}^b)^{-1}``\dot{G}_{l(b)}$ is $\bar{\P}_{\overline{l(b)}}$-generic over $\bN$.
\end{enumerate}
\end{enumerate}

As before, $D\rest l(a):=\{b\rest l(a)\st b\in D\}$ is open in $\P_{l(a)}$, and it suffices to show that $D\rest l(a)$ is predense below $a$ in $\P_{l(a)}$, because knowing this, we may choose a maximal antichain $A\sub D\rest l(a)$, which then is a maximal antichain in $\P_{l(a)}$ below $a$. For every $c\in A$, we may then pick a condition $b(c)\in D$ such that $b(c)\rest l(a)=c$, and define $\suc_T^n(a)=\{b(c)\st c\in A\}$. Now, for every $b\in\suc_T^n(a)$, let $S$, $\dot{\sigma}^b$, $u$ and $\bar{u}$ be chosen in such a way that \ref{item:FirstConditionDefiningPredenseSet2}-\ref{item:LastConditionDefiningPredenseSet2} hold. Set $T^{(b,n+1)}=S$, $\dot{\sigma}^{(b, n+1)}=\dot{\sigma}^b$, $q^{(b, n+1)}=u$, $\bar{q}^{(b,n+1)}=\bar{u}$ and $\bar{T}^{(b,n+1)}=\bS$. Then $a$, $b$ satisfy \ref{item:coherence2}-\ref{claim:LastOfFirstList2} at stage $n$, and $b$ satisfies \ref{item:middlepart-beginning2}-\ref{item:middlepart-end2} at stage $n+1$.

To see that $D\rest l(a)$ is predense below $a$, let
$G_{l(a)}$ be $\P_{l(a)}$-generic over $V$ with $a \in G_{l(a)}$. We have to find a $b\in D$ so that $b\rest l(a) \in G_{l(a)}$. Work in $V[G_{l(a)}]$. Let $\sigma_n = (\dot{\sigma}^{(a, n)})_{G_{l(a)}}$. Since \ref{item:sigma(a,n)2} holds at stage $n$, we have that $\sigma_n:\bN\prec N$
%${\mathrm{Hull}}^N(\delta(P_\nu) \cup \ran(\sigma)) = {\mathrm{Hull}}^N(\delta(P_\nu) \cup \ran(\sigma_n))$,
and $\sigma_n(\overline{\theta}, \bar{\vec{\P}},\bs,\overline{\nu},\bar{i},\bar{j})
=\theta,\vec{\P},\sigma(\bs),\nu,i,j$.
We also have objects $\bar{q}^{(a,n)},\bT^{(a,n)},l(\overline{a})$ satisfying condition \ref{item:middlepart-end2}, so that %$\bar{q}^{(a,n)}\in\bG_{\bar{i}}*\bG_{\bar{i},\bar{j}}$,
$\bT^{(a,n)}\in\bN$, %$\bar{q}^{(a,n)}\le\bp_n$,
$\sigma_n(\bar{q}^{(a,n)},\bT^{(a,n)})=q^{(a,n)},T^{(a,n)}$ and $\sigma_n(\overline{l(a)})=l(a)$. We also know that $\bar{G}_{\overline{l(a)}}=\sigma_n^{-1}``G_{l(a)}$ is $\bar{\P}_{\overline{l(a)}}$-generic over $\bN$.

Let's first find $u$ and $\bar{u}$ again. By elementarity, $\overline{q}^{(a, n)}$ is a mixture of $\overline{T}^{(a, n)}$ up to $\overline{j}$. %We have that $\sigma_n:\kla{L_{\btau}[\bA][\bG_{\bar{i}}],\in,\bA}\prec\kla{L_\tau[A][G_i],\in,A}$, so in particular, $\bsigma:=\sigma_n\rest L_\btau[\bA]:\kla{L_\btau[\bA],\in,\bA}\prec\kla{L_\tau[A],\in,A}$, and $\bsigma(\bar{q}^{(a,n)},\bT^{(a,n)})=q^{(a,n)},T^{(a,n)}$. Clearly, in $L_\tau[A]$, it is true that $q^{(a,n)}$ is a mixture of $T^{(a,n)}$ up to $j$, so it is true in $L_\btau[\bA]$ that $\bar{q}^{(a,n)}$ is a mixture of $\bT^{(a,n)}$ up to $\bar{j}$, and by absoluteness, this it true in $\V$ as well.
Let $\bT_0^{(a,n)}=\{\ba_0\}$. Let $a_0=\sigma_n(\ba_0)$. So $T_0^{(a,n)}=\{a_0\}$. We know that $a\le q^{(a,n)}\rest l(a)$, so $q^{(a,n)}\rest l(a)\in G_{l(a)}$, so $\bar{q}^{(a,n)}\rest\overline{l(a)}\in\bar{G}_{\overline{l(a)}}$. Since $\bar{q}^{(a,n)}$ is a mixture of $\bar{T}^{(a,n)}$ up to $\bar{j}$ in $\bN$, it follows that  $\bar{q}^{(a,n)}\rest l(\ba_0)\equiv\ba_0$, by Fact \ref{fact:CharacterizationOfMixtures}.\ref{item:Root}, since $l(\bar{q}^{(a,n)})=\bar{j}>l(\ba_0)$.
Since we have ensured that $l(a)=l(\Root(T^{(a,n)})$, we have by elementarity that $\overline{l(a)}=l(\bar{a}_0)$.
So $\ba_0\in\bar{G}_{\bar{l(a)}}$.
Let $\bar{r}\in\bT_1^{(a,n)}$ be such that $\bar{r}\rest l(\ba_0)\in\bG_{l(\ba_0)}$. There is such an $\bar{r}$ by Definition \ref{definition:NestedAC-Mixture-beta-Nice-hooks}(5), since $\bG_{l(\ba_0)}$ is $\bP_{l(\ba_0)}$-generic over $\bN$. So we have that $\bar{r}\rest\overline{l(a)}\in\bar{G}_{\overline{l(a)}}$.

%By Fact \ref{fact:CharacterizationOfMixtures}.\ref{item:HigherUp}, again since $l(\bar{r})<\bar{j}=l(\bar{q}^{(a,n)})$, it follows that $\bar{r}\equiv\bar{r}\rest l(\ba_0)\verl\bar{q}^{(a,n)}\rest[l(\ba_0),l(\bar{r})$.
%Since $\bar{r}\rest l(\ba_0)\in\bG_{l(\ba_0)}$, this implies that $\bar{r}\rest[l(\ba_0),l(\bar{z}))\equiv\bar{q}^{(a,n)}\rest[l(\ba_0),l(\bar{r}))$ (in the partial order $\bar{P}_{l(\ba_0),l(\bar{r})}=\bar{P}_{l(\bar{r})}/\bG_{l(\ba_0)}$), and $\bar{q}^{(a,n)}\rest[l(\ba_0),l(\bar{r}))\in\bG\rest[l(\ba_0),l(\bar{r}))$.
%So we have that $\bar{r}\rest l(\ba_0)\in\bG_{l(\ba_0)}$ and $\bar{r}\rest[l(\ba_0),l(\bar{r}))\in\bG\rest[l(\ba_0),l(\bar{r}))$. By Fact \ref{fact:FactorsAndQuotients}.\ref{item:CriterionForBeingInGeneric}, this implies that $\bar{r}\in\bG_{l(\bar{r})}$.

Now let $\bar{u}\in\bP_{\bar{j}}$ strengthen $\tilde{q}=\bar{r}^{\frown}\bar{q}^{(a, n)}\rest [{l(\bar{r})},\bar{j})$ so that $\overline{u}\in\overline{D}_{n+1}$ and $\bar{u}\rest\overline{l(a)}\in\bG_{\overline{l(a)}}$. This can be achieved since $\tilde{D}=\{x\in\bar{D}_{n+1}\st x\le\tilde{q}\}$ is dense below $\tilde{q}$ in $\bar{\P}_{\bar{j}}$, so $\tilde{D}\rest\overline{l(a)}=\{x\rest\overline{l(a)}\st x\in\tilde{D}\}$ is dense below $\bar{r}\rest\overline{l(a)}$ in $\bar{\P}\rest\overline{l(a)}$, so there is a $\tilde{u}\in\tilde{D}\rest\overline{l(a)}\cap\bG_{\overline{l(a)}}$, and $\tilde{u}=\bar{u}\rest\overline{l(a)}$ for some $\bar{u}\in\tilde{D}$ which then is as wished.

By Lemma \ref{2.11}, applied in $\bN$, we can pick a nested antichain $\bar{S} \hooks \bar{T}^{(a, n)}$ such that $\bar{u}$ is a mixture of $\bS$ up to $\bar{j}$ and such that letting $\bS_0 = \{\bd_0\}$, we have that  $l(\bar{r}),\overline{l(a)}\leq l(\bd_0)$ and $\bd_0 \rest l(\bar{r}) \leq \bar{r}$.
Moreover, we have that $\bar{u}\rest l(\bd_0)\equiv\bd_0$, since $\bar{u}$ is a mixture of $\bS$ up to $\bar{j}$. In particular, $\bd_0\rest\overline{l(a)}\in\bG_{\overline{l(a)}}$.

Let $S,d_0,u=\sigma_n(\bS,\bd_0,\bu)$. So $u\rest l(d_0)\equiv d_0$. Note that $d_0\rest l(a)\in G_{l(a)}$.

Also, let $\bs'\in\bN$ be a tuple of parameters to be specified below, and let $s'=\sigma_n(\bs')$.

Let $w\in G_{l(a)}$ force these facts about $\dot{\sigma}^{(a,n)}$ and the check names for the parameters $\bar{S},\bar{d_0},\bu,\bs'$ and $S,d_0,u,s'$. Since $a,d_0\rest l(a)\in G_{l(a)}$, we may choose $w$ so that $w\le_{l(a)}a,d_0\rest l(a)$.

We are now going to apply our inductive hypothesis $\phi(l(d_0))$ (noting that $l(d_0)<j$) to $l(a)$ (in place of $i$; note that $l(a)\leq l(d_0)$ as required), %the filters $\bG_{l(\overline{a})}$, $\bG_{{l(\overline{a})},l(\bd_0)}$,
the models $\bN$, $N$, the condition $w$ (in place of $p$), the condition $d_0$ (in place of $q$; note that $w\le_{l(a)} d_0\rest l(a)$, as required), the name $\dot{\sigma}^{(a,n)}$ (in place of $\dot{\sigma}$) and the parameter $\bs'\in\bN$.

No matter which $\bs'$ is chosen, by the inductive hypothesis, there are a condition $w^*\in\P_{l(d_0)}$ with $w^*\rest l(a)=w$ and $w^*\le d_0$, and a $\P_{l(d_0)}$-name $\dot{\sigma}'$ such that $w^*$ forces:
%such that if $H\ni w^*$ is generic for $\P_{l(d_0)}$, then in $\V[H]$ there is a $\sigma'$ such that letting $\sigma'_n=(\sigma^{(a,n)})^H$,
\begin{enumerate}[label=(\alph*)]
  \item
  $\dot{\sigma}':\check{\bN}\prec \check{N}$.
  \item
  $\dot{\sigma}'(\check{\bs}',
  \check{\bar{\vec{\P}}},
  \check{l(\overline{a})},
  \check{l(\overline{d}_0)},
  \check{\btheta})=
  \dot{\sigma}^{(a,n)}(\check{\bs}',
  \check{\bar{\vec{\P}}},
  \check{l(\overline{a})},\check{l(\overline{d}_0)},\check{\btheta})$.
  \item
  $(\dot{\sigma}'{}^{-1})`` \dot{G}_{l(a)}=\dot{\sigma}^{-1}``\dot{G}_{l(a)}$ and
  $(\dot{\sigma}'{}^{-1})`` \dot{G}_{l(d_0)}$ is $\P_{l(d_0)}$-generic over $\check{\bN}$.
\end{enumerate}

Thus, we let $\bs'$ be a name for the finite tuple of further objects of which we want $w^*$ to force that $\dot{\sigma}'$ moves them the same way as $\dot{\sigma}^{(a,n)}$: $\bu,\bd_0,\bS,\bar{i},\bar{j},\bar{\vec{\P}},\btheta,\bs,t_0,\ldots,t_n$ and for every $\kla{c,m}\le_T\kla{a,n}$,  $\bar{q}^{(c,m)}$.
Recall that $w$ forced that $\dot{\sigma}^{(a,n)}(\bu,\bd_0,\bS)=u,d_0,S$.
Hence, since $w^*\rest l(a)=w$, we get that $w^*$ forces that $\dot{\sigma}'(\bu,\bd_0,\bS)=u,d_0,S$ as well.

Note that already $a$ forced with respect to $\P_{l(a)}$ that $\bar{i},\bar{j},\bar{\vec{\P}},\btheta$ are mapped to $i,j,\vec{\P},\theta$ by $\dot{\sigma}^{(a,n)}$.

%Note that $\sigma'$ lifts to an elementary embedding $\sigma':\bN[\bG_{l(\bd_0)}]\prec N[H]$ with $\sigma'(\bG_{l(\bd_0)})=H_{l(d_0)}$. It follows that $\sigma'(\bu\rest l(\bd_0))=u\rest l(d_0)\in H$, since $\bu\in\bG_{\bar{j}}$, so that $\bu\rest l(\bd_0)\in\bG_{l(\bd_0)}$.

%Since $\bu\in\bG$, it follows that $\bu\rest l(d_0)\in\bG\rest l(\bd_0)$, and so, $w^*$ forces that $u\rest l(d_0)=\sigma'(\bu\rest l(\bd_0))\in\dot{G}_{l(d_0)}$. Thus, given any $P_{l(a),l(d_0)}$-generic filter $H$ containing $w^*$, we may find a common extension of $w^*$ and $u\rest l(d_0)$ in $H$. So we may assume that $w^*\le u\rest l(d_0)$.

Now, set $b=w^*$, $\dot{\sigma}^b=\dot{\sigma}'$. It follows that the conditions \ref{item:FirstConditionDefiningPredenseSet2}-\ref{item:LastConditionDefiningPredenseSet2} are satisfied, that is, $b\in D$. Most of these are straightforward to verify. Let us remark that condition \ref{item:EnsuringGettingAFusionStructure2} holds because $b=w^*\le_{l(b)}d_0\equiv u\rest l(d_0)$, so $b\le u\rest l(b)$ as $l(b)=l(d_0)$. Condition \ref{item:InitialSegmentOfbForcesStuffIntoRangeOfTheOldEmbedding2} holds because $b\rest l(a)=w$. This completes the proof that $D\rest l(a)$ is predense below $a$, and hence the treatment of case 1.

\medskip

\noindent\emph{Case 2:} $j$ is a successor ordinal.

Let $j=k+1$. Since we assumed $i<j$, it follows that $i\le k$. Inductively, $\phi(k)$ holds. Note that $\bar{j}$ is of the form $\bk+1$, and $p$ forces with respect to $\P_i$ that $\dot{\sigma}(\bk)=k$, the assumptions \ref{item:LastAssumptionSP}-\ref{item:Generics1} are satisfied by
$p\in \P_i$, $\dot{\sigma}\in\V^{\P_i}$, $q\rest k\in\P_k$,
$\theta$, $\tau$, $A$, $N$, $\bN$,  $\bs,\bar{\vec{\P}},\bar{i},\bar{k}\in\bN$, and $k$. By $\phi(k)$, we obtain a condition $p^{**}\in\P_k$ with $p^{**}\rest i=p$ and $p^{**}\le_k q\rest k$, as well as a $\P_k$-name $\dot{\bsigma}$ such that $p^{**}$ forces with respect to $\P_k$:
\begin{enumerate}[label=(\alph*1)]
  \item
  $\dot{\bsigma}:\check{\bN}\prec \check{N}$.
  \item
  $\dot{\bsigma}(\check{\bs},\check{\bar{\vec{\P}}},\check{\bar{i}},
  \check{\bk},\check{\bar{j}},\check{\btheta},\check{\bar{q}})
  =\dot{\sigma}(\check{\bs},\check{\bar{\vec{\P}}},\check{\bar{i}},\check{\bk},
  \check{\bar{j}},\check{\btheta},\check{\bar{q}})$.
  \item
  $\dot{\bsigma}^{-1}``\dot{G}_i=\dot{\sigma}^{-1}``\dot{G}_i$ and
  $\dot{\bsigma}^{-1}``\dot{G}_k$ is $\bar{\P}_{\bar{k}}$-generic over $\check{\bN}$.
\end{enumerate}
As pointed out in the beginning of the proof, there is a $\P_k$-name $\dot{\bsigma}^*$ such that $p^{**}$ forces that $\dot{\bsigma}$ lifts nicely - in detail, $p^{**}$ forces with respect to $\P_k$:
\begin{enumerate}[label=(\alph*2)]
  \item
  $\dot{\bsigma}^*:\check{\bN}[\dot{\bsigma}^{-1}``\dot{G}_k]\prec \check{N}[\dot{G}_k]$.
  \item
  $\dot{\bsigma}^*(\check{\bs},\check{\bar{\vec{\P}}},\check{\bar{i}},
  \check{\bk},\check{\bar{j}},\check{\btheta},\check{\bar{q}})
  =\sigma(\check{\bs},\check{\bar{\vec{\P}}},\check{\bar{i}},\check{\bk},
  \check{\bar{j}},\check{\btheta},\check{\bar{q}})$.
  \item
  $(\dot{\bsigma}^*)^{-1}``\dot{G}_i=\dot{\sigma}^{-1}``\dot{G}_i$,
  $\dot{\bsigma}^*(\dot{\sigma}^{-1}``\dot{G}_i)=\dot{G}_i$
  and
  $(\dot{\bsigma}^*)^{-1}``\dot{G}_k$ is $\bar{\P}_{\bar{k}}$-generic over $\check{\bN}$.
\end{enumerate}

Temporarily fix a $\P_k$-generic filter $G_k$ that contains $p^{**}$.
In $\V[G_k]$, the forcing $\P_{k,k+1}=\P_{k,j}=\P_j/G_k$ is $\infty$-subproper by assumption.
Note that since $p^{**}\le_k q\rest k$, $q^*=q\rest[k,j)\in\P_{k,j}$.
Let $\bsigma^*=(\dot{\bsigma}^*)^{G_k}$. Then $q^*$ is in the range of $\bsigma^*$ and  we have that $\bsigma^*:\bN[(\bsigma^*)^{-1}``G_k]\prec N[G_k]$. Also, letting $\sigma=\dot{\sigma}^{G_k}$, $\sigma^{-1}``G_i=(\bsigma^*)^{-1}``G_i$ and $\bsigma^*(\sigma^{-1}``G_i)=G_i$.

Since $\bN$ is full, so is $\bN[\bG_\bk]$, and since $H_\theta\sub N$, $H_\theta^{\V[G_k]}=H_\theta[G_k]\sub N[G_k]$. Hence, by $\infty$-subproperness of $\P_{k,j}$ in $\V[G_k]$, there is a condition $r$ in $\P_{k,j}$ with $r\le_{k,j}q^*$ such that, writing $\dot{G}_{k,j}$ for the canonical $\P_{k,j}$-name for the generic filter, $r$ forces the existence of an elementary embedding $\sigma'$ with
\begin{enumerate}[label=(\alph*3)]
  \item
  $\sigma':\check{\bN}[\bsigma^{-1}``G_k]\prec N[G_k]$.
  \item
  $\sigma'(\check{\bs},\check{\bar{\vec{\P}}},\check{\bar{i}},\check{\bk},
  \check{\bar{j}},\check{\btheta},\check{\bar{q}},\sigma^{-1}``G_i)
  =\bsigma^*(\check{\bs},\check{\bar{\vec{\P}}},\check{\bar{i}},\check{\bk},
  \check{\bar{j}},\check{\btheta},\check{\bar{q}},\sigma^{-1}``G_i)$.
  \item
  $(\sigma')^{-1}``\dot{G}_{k,j}$ is $\bar{\P}_{\bar{k},\bar{j}}$-generic over $\check{\bN}[\bsigma^{-1}``G_k]$.
\end{enumerate}
Since this holds in $\V[G_k]$ whenever $p^{**}\in G_k$, there is a $\P_k$-name $\tau$ which is essentially a name for $r$ above - more precisely, $p^{**}$ forces that $\tau\in \P_j$, $\tau\rest k\in\dot{G}_k$ and $\tau\rest[k,j)$ has the properties of $r$, as listed above.
Since the iteration is nice, there is a condition $p^*\in \P_j$ such that $p^*\rest k=p^{**}$ and $p^*$ forces that $\tau\rest[k,j)\equiv p^*\rest[k,j)$; see Definition \ref{def:NiceIterations}, part (1). We claim that $p^*$ is as wished.

First, note that $p^*\rest i=(p^*\rest k)\rest i=p^{**}\rest i=p$. Also, $p^{**}=p^*\rest k\le q\rest $k and $p^{**}$ forces that $p^*\rest[k,j)\equiv\tau\rest[k,j)\le_{k,j}q^*$, so $p^*\le q$.

Now, let $G_j$ be a $\P_j$-generic filter with $p^*\in G_j$. Then $\V[G_j]=\V[G_k][G_{k,j}]$, where $p^{**}\in G_k$ and $p^*\rest[k,j)\in G_{k,j}$. Since $p^*\rest[k,j)$ has the properties of the condition $r$ above in $\V[G_k]$, there is in $\V[G_j]$ a $\sigma'$ with (a3)-(c3). Letting $\sigma=\dot{\sigma}^{G_j}$, it follows then that $\bG_i:=\sigma^{-1}``G_i=\sigma'{}^{-1}``G_i$, since $\sigma'(\bG_i)=G_i$. Moreover, since $(\sigma')^{-1}``G_{k,j}$ is $\bar{\P}_{\bar{k},\bar{j}}$-generic over $\bN[\bG_{\bar{k}}]$ and $\bG_{\bar{k}}=(\sigma')^{-1}``G_k$ is $\bar{\P}_{\bar{k}}$-generic over $\bN$, $(\sigma')^{-1}``G_j$ is $\bar{\P}_{\bar{j}}$-generic over $\bN$.
Thus, the restriction of $\sigma'$ to $\bN$ has the desired properties.
%
%We have to show that in $\V[G_j]$, there is a $\sigma'$ such that, letting $\sigma=\dot{\sigma}^{G_i}$, the following hold:
%\begin{enumerate}[label=(\alph*)]
%  \item
%  $\sigma':\bN\prec N$.
%  \item
%  $\sigma'(\bs,\bar{\vec{\P}},\bar{i},\bar{j},\btheta,\bar{q})
%  =\sigma(\bs,\bar{\vec{\P}},\bar{i},\bar{j},\btheta,\bar{q})$.
%  \item
%  $(\sigma'{}^{-1})``G_{i}=\sigma^{-1}``G_i$ and $\sigma'{}^{-1}``G_j$ $\bar{\P}_i$-generic over $\bN$.
%\end{enumerate}
%But this follows, because $\V[G_j]=\V[G_k][G_{k,j}]$, where $p^{**}\in G_k$ and $p^*\rest[k,j)\in G_{k,j}$ , so writing $H$ for $G_k$, this is the situation described above. Moreover, $p^*\rest[k,j)\in G_{k,j}$ and $p^*\rest[k,j)\equiv\dot{r}^{G_j}$, where $\dot{r}$ is a name for the condition $r$ mentioned above. Thus, there is a $\sigma'$ in $\V[G_j]$ such that the conditions (a2)-(c2) listed above hold in $\V[G_j]$. Remembering that $\tsigma$ lifts $\bsigma$ and $\bsigma$ moves the required parameters as prescribed (by (a1)-(c1)), it follows that (a)-(c) are satisfied.
\end{proof}

So $\infty$-subproper and $\infty$-subcomplete forcings are nicely iterable when using nice iterations. Let us make one strengthening of these theorems that will be useful in applications. The key step in both proofs was the construction of the fusion sequence in the limit stage and in particular the conditions $u$ and $\bar{u}$. In both proofs we needed $u$ to be as strong as a certain condition, but in fact, we could have strengthened it further if we liked. We get the following theorem, a version of \cite[Lemma 4.3]{Miyamoto:IteratingSemiproperPreorders}:

\begin{thm}
Let $\vec{\P}=\seq{\P_\alpha}{\alpha\le\nu}$
be a nice iteration with $\P_0 = \{1_0\}$ such that for all $i$ with $i+1< \nu$,  $\forces_i \P_{i, i+1}$ is $\infty$-subproper, where $\nu$ is a limit ordinal.
Let $\theta$, $N$, $\bN$, $\bar{s}$, $s$ be as in Theorem \ref{thm:NiceIterationsOfSPforcingAreSP}, and let $\sigma:\bN\prec N$.
Let $q\in\P_\nu$ and suppose $\seq{E_n}{n<\omega}$ is a sequence of subsets of $\P_\nu$ such that for each $n<\omega$,
any $u\le_\nu q$, any $i<\nu$ and any $\P_i$-name $\dot{\sigma}$ we have that $1_i$ forces:

``if $\dot{\sigma}:\bN\prec N$, $u,q,\theta,\vec{\P},s,i\in\ran(\dot{\sigma})$, $\dot{\sigma}(\bar{s},\bar{\theta},\bar{\vec{\P}})=s,\theta,\vec{\P}$, $u\rest i\in\dot{G}_i$,  and $\dot{\sigma}^{-1}``\dot{G}_i$ is $\dot{\sigma}^{-1}(\P_i)$-generic over $\bN$, then there is an $r\in E_n\cap\ran(\dot{\sigma})$ such that $r\le_\nu u$ and $r\rest i\in\dot{G}_i$.''

Then, there is a $p^*\le_\nu q$ such that whenever $G_\nu$ is $\P_\nu$-generic with $p^*\in G_\nu$, there is a $\sigma'\in\V[G_\nu]$ such that the conclusions (a)-(c) of the previous theorem hold (with $i=0$, $j=\nu$, $p=1_0$,$\dot{\sigma}=\check{\sigma}$), {\bf and} in addition, $G_\nu\cap E_n\neq\leer$ for every $n<\omega$.
\label{thmEn}
\end{thm}

The proof of
\cite[Lemma 4.3]{Miyamoto:IteratingSemiproperPreorders} describes in the context of semiproper forcing the modifications to the construction of the fusion sequence in the limit case needed to ensure that the sets $E_n$ ($n<\omega$) are met. The corresponding modifications can be carried out in the context of $\infty$-subproper forcing as well. These ideas are due to Miyamoto, and we refer the reader to his article for the details.

%\begin{thm}
%Let $\langle \P_\alpha \; | \; \alpha \leq \nu \rangle$
%be a nice iteration with $\nu$ limit, so that $\P_0 = \{1_0\}$ and either for all $i$ with $i + 1 < \nu$,  $\forces_i \P_{i, i+1}$ is $\infty$-subproper or for all such $i$, $\forces_i \P_{i,i+1}$ is $\infty$-subcomplete.
%Let $\theta$, $N$, etc. be as in either Theorem \ref{thm:NiceIterationsOfSCforcingAreSC} or Theorem \ref{thm:NiceIterationsOfSPforcingAreSP} and suppose for all $n < \omega$ we have that $E_n \subseteq\P_\nu$ satisfies the following: for every $p \in \mathbb P_\nu$ and every $\alpha < \nu$ if there is a $\P_\alpha$-name for an embedding $\dot{\sigma} : \bN \prec N$ with $p$ forced to be in the range of $\dot{\sigma}$ and for any $u \leq p$ in the range of $\dot{\sigma}$ with $u \hook \alpha \in \dot{G}_\alpha$, there is an $s \leq u$ in the range of $\dot{\sigma}$ so that $s \hook \alpha \in \dot{G}_\alpha$ and $s \in E_n$. Then there is a $q \leq p$ forcing that there is a decreasing sequence $q_0 \geq q_1 \geq ... \geq q_n \geq ...$ all in $G$ so that $q_{n+1} \in E_n$. In particular, $q$ forces $G \cap E_n \neq \emptyset$ for all $n < \omega$.
%\label{thmEn}
%\end{thm}
%
%\begin{proof}
%In the case of semiproper forcing this is checked in detail by Miyamoto as \cite[Lemma 4.3]{Miyamoto:IteratingSemiproperPreorders}. Making the exact same modification he makes in that case to our proofs of Theorems \ref{thm:NiceIterationsOfSCforcingAreSC} and \ref{thm:NiceIterationsOfSPforcingAreSP} works here. The reader is referred to Miyamoto's paper for the details.
%\end{proof}

\subsection{Preserving Properties of Trees}

In this section we lift some results about preservation of properties of trees from \cite{Miyamoto:IteratingSemiproperPreorders} to the context of $\infty$-subproper forcing. The proofs are in the same spirit as those given using RCS iterations, but since those results do not apply to $\infty$-subproper forcing we give them here in the modified context as well. Both results are instances of Theorem \ref{thmEn}.

\begin{lem}
Let $S = (S, \leq_S)$ be a Souslin tree and $\vec{\P}=\seq{\P_\alpha}{\alpha\le\nu}$ be a nice iteration of $\infty$-subproper forcings as in Theorem \ref{thm:NiceIterationsOfSPforcingAreSP}, such that for each $i+1\le\nu$, $\forces_i$
``if $\check{S}$ is Souslin, then $\P_{i,i+1}$ preserves $\check{S}$ as a Souslin tree.'' Then $\forces_\nu$ ``$\check{S}$ is a Souslin tree.''
\label{souslinpre}
\end{lem}

We stress that the proof of this theorem is similar to that of Lemma 5.0 and Theorem 5.1 of \cite{Miyamoto:IteratingSemiproperPreorders}. %Note that as a consequence, if $S$ is a Souslin tree and $\seq{\P_\alpha}{\alpha\le\nu}$ is a sequence of $\infty$-subproper forcing notions such that for every $i+1<\nu$, $\forces_i$ ``if $\check{S}$ is Souslin, then $\P_{i,i+1}$ preserves $\check{S}$ as a Souslin tree,'' then $\forces_\nu$ ``$\check{S}$ is Souslin.'' In this sense, $\infty$-subproper, $S$-preserving forcing is nicely iterable.

\begin{proof}
If not, then let $\vec{\P}=\seq{\P_i}{i\le\nu}$ be a counterexample with $\nu$ minimal. Clearly, $\nu$ must be a limit ordinal, and by minimality, for every $i<\nu$, $\forces_i$``$\check{S}$ is a Souslin tree.''
Let $\dot{A}$ be a $\P_\nu$ name such that $q \in \P_\nu$ forces that $\dot{A}$ is a maximal antichain in $\check{S}$. We need to find a $p^* \leq q$ forcing that $\dot{A}$ is countable. Fix $\theta$ sufficiently large that $\dot{A}, \mathbb P, S \in H_\theta$ and fix $\sigma:\bN \prec N$ as in the standard setup; we can ensure that $\bN$ is full using the usual arguments; see the remark following Def.~2.21 of \cite{Fuchs:CanonicalFragmentsOfSRP}. Let $s=\kla{S,\dot{A}}$, $\bar{s}=\sigma^{-1}(s)$.
Set $\delta = \omega_1^\bN=\sigma^{-1}(\omega_1)$. We may assume that $\sigma$ does not move the nodes in $\bar{S}=\sigma^{-1}(S)$, since we may assume $S\subseteq H_{\omega_1}$. Enumerate the $\delta^{\rm th}$ level of $S$ as $\langle s_n \; | \; n < \omega\rangle$. For each $n$, define $E_n = \{r \in \P_\nu \; | \; \exists s \in S \; {\rm such \; that} \; s <_S s_n \; {\rm and} \; r \forces s \in \dot{A}\}$. We claim that the $E_n$'s satisfy the ``predensity condition'' of Theorem \ref{thmEn}. If we can do this then it follows that there is a $p^*\le q$ forcing that $\dot{G} \cap E_n \neq \emptyset$ for all $n< \omega$, and hence that the levels of the nodes occurring in $\dot{A}$ are all below $\delta$, rendering $\dot{A}$ countable.

To check the predensity condition, fix $n<\omega$, $u\le_\nu q$, $i<\nu$ and a $\P_i$-name $\dot{\sigma}$. Suppose that $1_i$ forces that
$\dot{\sigma}:\bN\prec N$, $u,q,\theta,\nu,\vec{\P}\in\ran(\dot{\sigma})$, $u\rest i\in\dot{G}_i$, $\dot{\sigma}(\bar{\theta},\bar{s},\bar{i})=\theta,s,i$,
and $\dot{\sigma}^{-1}``\dot{G}_i$ is $\dot{\sigma}^{-1}(\P_i)$-generic over $\bN$.
Let $G_i$ be $\P_i$-generic, $\sigma_i=\dot{\sigma}^{G_i}$ and $\bar{G}_{\bar{i}}=\sigma_i^{-1}``G_i$.

%fix $u\leq p$, in the range of $\sigma$, $\alpha < \nu$ and a $P_\alpha$-name $\dot{\sigma_\alpha}$ which will evaluate to an embedding witnessing the $\infty$-subproperness of $P_\alpha$. Let $G_\alpha$ be $\P_\alpha$-generic over $V$, $\sigma_\alpha = (\dot{\sigma}_\alpha)^{G_\alpha}$ and $\sigma_\alpha(\overline{u}, \overline{\alpha}, \overline{S}) = u, \alpha, S$.
%We want to find a $\bar{r}\in\bar{\P}_{\bar{\nu}}$ so that $\bar{r}\le \bar{u}$, $\sigma_\alpha(\bar{r}\rest\bar{i})\in G_i$ and $r\in E_n$; then $r=\sigma_i(\bar{r})$ is as required in the ``predensity'' condition of Theorem \ref{thmEn}.

Recall that $\sigma_i$ lifts to an embedding $\sigma^*_i: \bN[\barG_{\bar{i}}]\prec N[G_i]$ with $\sigma^*_i(\bar{G}_{\bar{i}})=G_i$.
Working in $\bN[\barG_{\bar{i}}]$, let %$D$ be the set of $s\in\bar{S}$ for which there is a condition $\bar{r} \in\bar{\P}_{\bar{\nu}}$ which strengthens $\overline{u}$ and so that $\bar{r}\rest\bar{i} \in \bar{G}_{\bar{i}}$ and $\bar{r} \forces s \in\bar{\dot{A}}$. In symbols
$D=\{s\in\bar{S}\st\exists\bar{r}\in\bar{\P}_{\bar{\nu}}\  \bar{r}\le\bar{u},\ \bar{r}\rest\bar{i} \in \bG_{\bar{i}}\ \text{and}\ \bar{r}\forces_{\bar{\nu}}s\in\bar{\dot{A}}\}$.
We claim that as a subset of $\bar{S}$, $D$ is predense, that is, every $t\in\bar{S}$ is comparable with some element of $D$.

To see this, work in $\bN$. There, $\bar{\dot{A}}$ is forced to be a maximal antichain in $\check{\bar{S}}$ by $1_{\bar{\P}_{\bar{\nu}}}$. So fixing $t\in\bar{S}$, the set $E$ of $\bar{r}\in\bar{\P}_{\bar{\nu}}$ such that $\bar{r}\le_{\bar{\P}_{\bar{\nu}}}\bar{u}$ and there is an $s\in\bar{S}$ that is comparable to $t$ and $\bar{r}\forces_{\bP_{\bnu}}$ ``$s\in\bar{\dot{A}}$'' is dense below $\bar{u}$ as a subset of $\bP_{\bnu}$. Hence, the set $E\rest\bar{i}=\{\bar{r}\rest\bar{i}\st\bar{r}\in E\}$ is dense below $\bar{u}\rest\bar{i}$ as a subset of $\bP_{\bar{i}}$. Since $\bar{G}_{\bar{i}}$ is $\bP_{\bar{i}}$-generic over $\bN$ and $\bar{u}\rest\bar{i}\in\bar{G}_{\bar{i}}$, it follows that $\bar{G}_{\bar{i}}\cap E\rest\bar{i}\neq\leer$. So let $\bar{r}\in E$ be such that $\bar{r}\rest{\bar{i}}\in\bar{G}_{\bar{i}}\cap E\rest\bar{i}$, and let $s\in\bar{S}$ witness that $\bar{r}\in E$. Then $s\in D$, as witnessed by $\bar{r}$, and $s$ is comparable to $t$. This shows that $D$ is a predense subset of $\bar{S}$.

By assumption, $S$ is Souslin in $\V[G_i]$, hence in $N[G_i]$, and hence, $\bar{S}$ is Souslin in $\bN[\bar{G}_{\bar{i}}]$, by the elementarity of $\sigma^*_i$. Now the set $b_n=\{t\in\bar{S}\st t<_Ss_n\}$ is a cofinal branch of $\bar{S}$, and hence it is $\bar{S}$-generic over $\bN[\bar{G}_{\bar{i}}]$. So $b_n\cap D\neq\leer$, as $D$ is predense in $\bar{S}$. Let $s\in D$, $s<_Ss_n$. Let $\bar{r}$ witness this and set $r=\sigma_i(\bar{r})$. Then $\sigma_i(s)=s$ witnesses that $r\in E_n$, $r\le_\nu u$ and $r\rest i\in G_i$. Clearly, $r\in\ran(\sigma_i)$, so the predensity condition is satisfied.
\end{proof}

A similar modification of Lemma 5.2 and Theorem 5.3 of \cite{Miyamoto:IteratingSemiproperPreorders} can be used to prove the preservation of ``not adding uncountable branches through trees.''

\begin{lem}
Let $T$ be an $\omega_1$-tree and let $\vec{\P}=\seq{\P_\alpha}{\alpha\le\nu}$ be a nice iteration of $\infty$-subproper forcings such that for each $i<\nu$ with $i + 1 \leq \nu$, $\forces_i ``\P_{i, i+1}$ does  not add an uncountable branch through $\check{T}$'' then $\P_\nu$ does not add an uncountable branch through $T$.
\end{lem}

\begin{proof}
If not, then let $\vec{\P}$ be a counterexample of minimal length $\nu$. Then $\nu$ is a limit ordinal, and we have that for every $i<\nu$, $\P_i$ does not add an uncountable branch through $T$.
Let $\dot{b}$ be a $\P_\nu$-name and $q\in\P_\nu$ a condition that forces $\dot{b}$ to be a new uncountable branch through $T$, that is, a branch that did not exist in $\V$. %We will find a $p^*\leq q$ forcing that actually $\dot{b}$ is countable.
Fix $\theta$ sufficiently large that $\dot{b},\vec{\P}, T \in H_\theta$ and fix $\sigma:\bN \prec N$ as before. Let $\delta = \omega_1^\bN$. As before, assume that $\sigma$ does not move the nodes in $\bar{T}=\sigma^{-1}(T)$. So $\bar{T}=T|\delta$. Enumerate the $\delta^{\rm th}$ level of $T$ as $\langle t^n \; | \; n < \omega\rangle$. For $n<\omega$, define
\[E_n = \{r \in\P_\nu \; | \; \exists t \in T_{<\delta}\ \text{such that}\ t \nleq_T t^n \; {\rm and} \; r \forces \check{t} \in \dot{b}\}.\]
We will again show that the predensity condition of Theorem \ref{thmEn} is satisfied. Knowing this, it then follows there is a $p^*\leq q$ forcing that $\dot{G}_\nu\cap E_n \neq \emptyset$ for all $n< \omega$. But if $G_\nu$ is $\P_\nu$-generic, $r\in G_\nu\cap E_n$ and $t$ witnesses that $r\in E_n$, then $t_n\notin\dot{b}^{G_\nu}$ (because $t\in\dot{b}^{G_\nu}$, so if $t_n$ were in $\dot{b}^{G_\nu}$, then it would follow that $t<_Tt_n$, which does not hold). Thus, $\dot{b}^{G_\nu}$ does not intersect the $\delta$-th level of $T$, a contradiction.

So fix $n<\omega$, $u\le_\nu q$, $i<\nu$ and a $\P_i$-name $\dot{\sigma}$. Suppose that $1_i$ forces that
$\dot{\sigma}:\bN\prec N$, $u,q,\theta,\nu,\vec{\P}\in\ran(\dot{\sigma})$, $u\rest i\in\dot{G}_i$, $\dot{\sigma}(\bar{\theta},\bar{s},\bar{i})=\theta,s,i$,
and $\dot{\sigma}^{-1}``\dot{G}_i$ is $\dot{\sigma}^{-1}(\P_i)$-generic over $\bN$.
Let $G_i$ be $\P_i$-generic, $\sigma_i=\dot{\sigma}^{G_i}$ and $\bar{G}_{\bar{i}}=\sigma_i^{-1}``G_i$.
Let $\sigma^*_i: \bN[\barG_{\bar{i}}]\prec N[G_i]$ be the lifting of $\sigma_i$.

%We want to find a $\overline{r} \in \bN \cap \overline{P}_\nu$ so that $\overline{r} \leq \overline{u}$, $\sigma_\alpha (\overline{r} \hook \overline{\alpha}) = r \hook \alpha \in G_\alpha$ and $r \in E_n$.

Since $\P_i$ does not add a new uncountable branch to $T$, $T$ has the same uncountable branches in $\V[G_i]$ as it has in $\V$. This implies that $G_i$ intersects the set $D$ of conditions $x\in \P_i$ such that there are incomparable nodes $t_1$, $t_2\in T_{{<}\delta}$ and conditions $u_1,u_2\le_\nu u$ such that $u_1$ forces that $t_1\in\dot{b}$, $u_2$ forces that $t_2\in\dot{b}$ and $u_1\rest i=u_2\rest i=x$. %is predense below $u\rest i$ in $\P_i$. To see this, it suffices to show that $G_i$ intersects $D$.
To see this, let $G'$, $G''$ be mutually generic over $\V[G_i]$ for $\P_{i,\nu}$, such that $u\in G_i*G'$ and $u\in G_i*G''$. Then $\dot{b}^{G_i*G'}\neq\dot{b}^{G_i*G''}$, or else $b=\dot{b}^{G_i*G'}=\dot{b}^{G_i*G''}\in\V[G_i][G']\cap\V[G_i][G'']=\V[G_i]$, meaning that already $\P_i$ would add a new uncountable branch to $T$. Let $t_1\in\dot{b}^{G_i*G'}\ohne\dot{b}^{G_i*G''}$, $t_2\in\dot{b}^{G_i*G''}\ohne\dot{b}^{G_i*G'}$. Then $t_1$ and $t_2$ are incomparable in $T$. Let $u'_1\in G_i*G'$, $u'_2\in G_i*G''$ be such that $u'_1\forces_\nu$ ``$\check{t}_1\in\dot{b}$'' and $u'_2\forces_\nu$ ``$\check{t}_2\in\dot{b}$''. Since both $u'_1\rest i$ and $u'_2\rest i$ are in $G_i$, we may find a common extension $x\le_i u'_1,u'_2$ with $x\in G_i$. Then, setting $u_1=x\verl u'_1\rest[i,\nu)$, $u_2=x\verl u'_2\rest[i,\nu)$, we have that $t_1, t_2, u_1, u_2$ witness that $x\in D$, and $x\in G_i$, as wished.

By the elementarity of $\sigma^*_i$, $\bar{D}=(\sigma^*_i)^{-1}(D)$ intersects $\bar{G}_{\bar{i}}$. Let $\bar{r}\in\bar{D}\cap\bar{G}_{\bar{i}}$, and let $t_1,t_2,\bar{u}_1,\bar{u}_2$ witness this. Since $t_1$ and $t_2$ are incomparable in $\bar{T}=T\rest\delta$, at least one of them is not $T$-below $t^n$. Say $t_1\not<t^n$. Then $r=\sigma_i(\bar{u}_1)$ is in $E_n$, as witnessed by $t_1=\sigma_i(t_1)$, $r\le u$, $r\rest i\in G_i$ and $r\in\ran(\sigma_i)$, verifying the predensity condition.
\end{proof}

\subsection{Nice iterations of $\infty$-subproper $\oo$-bounding forcing}
\label{subsec:NiceIterationsOfOO-bounding}

In this section we will prove that $\infty$-subproper and $\oo$-bounding forcing notions are nicely iterable. Throughout, by a real we will mean a member of Baire space. We will use the concepts of a descending sequence of conditions in a poset \emph{interpreting} a name for a real, the sequence \emph{respecting} a real (in its interpretation of the given name for a real), and derived sequences, as outlined in Subsection \ref{subsec:NicelySubPiterationsOfOOboundingForcing} and described in detail in \cite[Section 3]{Abraham:ProperForcing}.

\begin{thm}
\label{thm:NiceIterationsOfOOBoundingInftySPforcing}
Nice iterations of $\oo$-bounding $\infty$-subproper forcing notions are $\oo$-bounding.
\end{thm}

\begin{proof}
Let $\seq{\P_i}{i\le\lambda}$ be a nice iteration of $\infty$-subproper, $\oo$-preserving forcing posets.

Let $\phi(j)$ be the statement that $\P_j$ is $\oo$-preserving and $\infty$-subproper. We prove $\phi(j)$ by induction on $j\le\lambda$. So suppose $\phi(\bar{j})$ holds for every $\bar{j}<j$.

We have to show that $\P_j$ is $\oo$-preserving. This follows trivially from the inductive hypothesis if $j$ is a successor ordinal. So assume that $j$ is a limit ordinal.

Fix a condition $q\in\P_j$ and a $\P_j$-name $\dot{x}$ such that $q$ forces that $\dot{x}$ is a real (in Baire space). We have to find an extension of $q$ that forces $\dot{x}$ to be bounded by some real in $\V$.

To this end, let $\theta$ be sufficiently large, $H_\theta\sub L_\tau^A\models\ZFCm$, $\theta<\tau$, $\sigma:\bN\prec N$ with $\dot{x},p,\vec{\P},j\in\ran(\sigma)$, where $\bN$ is countable, transitive and full. Use the bar notation for the preimages of these objects under $\sigma$. To emphasize the similarity to earlier constructions, the reader may think of $p=\leer$, the only member of $\P_0$ (which could also be denoted $1_0$), and $i=0$.

In $\bN$, pick some $\le_{\bar{\P}_{\bar{j}}}$-descending sequence of conditions $\seq{\bar{s}_k}{k<\omega}$ with $\bar{s}_0\le\bar{q}$ which interprets $\dot{\bar{x}}$ as some real $x_0$.

Fix a real $y$ to that for every $x\in\oo\cap\bN$, $x\le^* y$, and so that $x_0\le_0 y$.

We will find a condition $p^*$ extending $q$ in $\P_j$, such that $p^*$ forces that $\dot{x}\le_0\check{y}$.

Note that for any $\alpha<\beta<j$, $\P_\alpha$ and $\P_\beta$ are $\oo$-bounding. It follows that $\P_\alpha$ forces that $\P_{\alpha,\beta}$ is $\oo$-bounding. Hence, a variation of the proof of Lemma \ref{lem:oosubproperonesteplemma}, using Theorem \ref{thm:NiceIterationsOfSPforcingAreSP} instead of Lemma \ref{lem:SubpropernessExtensionLemma}, \cite[Lemma 3.3]{Abraham:ProperForcing} instead of Lemma \ref{lem:DerivedSequences} and again applying Corollary \ref{cor:Fullness}, yields the following version of Lemma \ref{lem:oosubproperonesteplemma} for the stages of the iteration:

\begin{itemize}
\item[$(*)$]
\label{item:oosubpropernessextension}
Let $\alpha\le\beta<j$. Fix some $s\in\bN$. Let $w\in\P_\alpha$, $q\in\P_j$ with $w\le_\alpha q\rest\alpha$.
%and let $g\in\oo$ bound all the reals in $\bN$.
Let $S=\kla{\theta,\vec{\P},\alpha,\beta,j,q,\dot{x},s}$. Let $\bar{q}\in\bN$ and let $\dot{\sigma}_\alpha$ be a $\P_\alpha$-name. Let
$\bar{S}=\kla{\bar{\theta},\vec{\bar{\P}},\balpha,\bbeta,\bar{j},\bar{q},\dot{\bar{x}},\bar{s}}\in\bN$ so that $w$ forces with respect to $\P_\alpha$:
\begin{enumerate}[label=(A\arabic*)]
\item $\dot{\sigma}_\alpha:\check{\bN}\prec\check{N}$,
\item $\dot{\sigma}_\alpha(\check{\bS})=\check{S}$,
%\item $\dot{\sigma}_\alpha(\dot{\bar{q}})=\dot{q}\in\P_j$,
%\item $\dot{q}\rest\check{\alpha}\in\dot{G}_\alpha$,
\item $\dot{\sigma}_\alpha^{-1}``\dot{G}_\alpha$ is $\check{\bN}$-generic for $\check{\bar{\P}}_{\balpha}$,
\item
\label{item:SequenceExists1}
Letting $M=\ran(\dot{\sigma}_\alpha)$, there is in $M[\dot{G_\alpha}]$ a decreasing sequence of conditions in $\P_{\alpha,j}$, below $\dot{q}\rest[\alpha,j)$, which interprets $\check{\dot{x}}$ and respects $\check{y}$.
\end{enumerate}

\noindent{\bf Then:} there are a condition $w^*\in\P_\beta$ with $w^*\rest\alpha=w$ and $w^*\le q\rest\beta$, and a $\P_\beta$-name $\dot{\sigma}_\beta$ such that whenever $G_\beta$ is $\P_\beta$-generic with $c\in G_\beta$, letting $\sigma_\beta=\dot{\sigma}_\beta^{G_\beta}$, $G_\alpha=G_\beta\rest\alpha$ and $\sigma_\alpha=\dot{\sigma}_\alpha^{G_\alpha}$, the following conditions hold:
\begin{enumerate}[label=(C\arabic*)]
\item
\label{item:FirstConsequence**}
 $\sigma_\beta:\bN\prec N$,
\item $\sigma_\beta(\bS)=S$,
%\item $\sigma_\beta(\dot{\bar{q}}^{G_\alpha})=\sigma_\alpha(\dot{\bar{q}}^{G_\alpha})$,
%\item $\dot{q}^{G_\alpha}\rest\beta\in G_\beta$,
\item
\label{item:ThirdConsequence**}
 $\sigma_\beta^{-1}``G_\alpha=\sigma_\alpha^{-1}``G_\alpha$ and ${\sigma_\beta}^{-1}``G_\beta$ is $\bar{\P}_{\bar{\beta}}$-generic over $\bar{N}$.
\item
\label{item:LastConsequence**}
Letting $M=\ran(\sigma_\beta)$, there is in $M[G_\beta]$ a decreasing sequence $\vc$ of conditions in $\P_{\beta,j}$ below $q\rest[\beta,j)$ which interprets $\dot{x}$ and respects $y$.
\end{enumerate}
\end{itemize}

\begin{proof}[Proof of $(*)$.] We give an argument, for completeness.
Let $G_\alpha$ be $\P_\alpha$-generic with $w\in G_\alpha$. Let $\sigma_\alpha=\dot{\sigma}_\alpha^{G_\alpha}$ and let $M=\ran(\sigma_\alpha)$. Let $\bar{G}_{\bar{\alpha}}=\sigma_\alpha^{-1}``G_\alpha$, and let $\sigma^*_\alpha:\bN[\bar{G}_{\bar{\alpha}}]\prec N[G_\alpha]$ be the unique elementary embedding that maps $\bar{G}_{\bar{\alpha}}$ to $G_\alpha$. Observe that the reals of $\bN[\bar{G}_{\balpha}]$ are eventually bounded by $g$, as are the reals of $M[G_\alpha]$. Let $\dot{x}/G_\alpha$ be the canonical $\P_{\alpha,j}$-name such that if $G_{\alpha,j}$ is $\P_{\alpha,j}$-generic, then $(\dot{x}/G_\alpha)^{G_\alpha,j}=\dot{x}^{G_\alpha*G_{\alpha,j}}$. It follows from \ref{item:SequenceExists1} that there is in $M[G_\alpha]$ a decreasing sequence $\vec{r}$ in $\P_{\alpha,j}$ below $q\rest[\alpha,j)$ which interprets $\dot{x}/G_\alpha$ and respects $y$. Applying \cite[Lemma 3.3]{Abraham:ProperForcing} to $M[G_\alpha]$, the orderings $\P_{\alpha,\beta}$ and $\P_{\alpha,j}$ and the condition $q\rest[\alpha,j)$, there is a condition $d\in\P_\beta$  with $d\rest\alpha\in G_\alpha$ (so $d\rest[\alpha,\beta)\in\P_{\alpha,\beta}$) such that $d\rest[\alpha,\beta)\in M[G_0]$ and $d\rest[\alpha,\beta)\le_{\alpha,\beta}q\rest[\alpha,\beta)$, with the key property that $d\rest[\alpha,\beta)$ forces with respect to $\P_{\alpha,\beta}$ that there is in $M[G_\alpha][\dot{G}_{\alpha,\beta}]$ a decreasing sequence $\vec{r}$ in $\P_{\beta,j}$ below $q\rest[\alpha,j)$ which interprets $\dot{x}$ and respects $y$; in fact the derived sequence $\delta_{G_\beta}(\vec{r},\dot{x})$ is such a sequence. There is a name for $d$, forced by $w$ to have all the properties listed, but since $w$ then also forces that $d\rest\alpha\in\dot{G}_\alpha$, it follows by Corollary \ref{cor:Fullness} that this name can be identified with a condition in $\P_\beta$ such that $d\rest\alpha=w$. It follows that $\bar{d}=\sigma_\alpha^{-1}(d)\in\bar{\P}_{\bar{\beta}}$.

Let $\dot{\vec{r}}\in M$ be a $\P_\alpha$-name for the sequence $\vec{r}$, and let $\dot{\vec{\bar{r}}}$ be its preimage under $\sigma_\alpha$. I.e., $\dot{\vec{r}}^{G_\alpha}=\vec{r}$ and $\sigma_\alpha(\dot{\vec{\bar{r}}})=\dot{\vec{r}}$.

Let $w'\in G_\alpha$, $w'\le w$, be such that $w'\forces\dot{\sigma}_\alpha(\check{\bar{d}})=\check{d}$ and $\dot{\sigma}_\alpha((\dot{\vec{\bar{r}}})\check{})=(\dot{\vec{r}})\check{}$, and such that $w'$ forces ``the interpretation of $(\dot{\vec{r}})\check{}$ is a decreasing sequence as described''.

Note that $w'\le w=d\rest\alpha$. By Theorem \ref{thm:NiceIterationsOfSPforcingAreSP}, there is a $\tilde{w}\in\P_\beta$ with $\tilde{w}\le_\beta d$ (so $\tilde{w}\rest[\alpha,\beta)\le q\rest[\alpha,\beta)$) and $\tilde{w}\rest\alpha=w'$ such that whenever $H_\beta$ is $\P_\beta$-generic over with $\tilde{w}\in H_\beta$, letting $H_\alpha=H\rest\alpha$, there is in $\V[H_\beta]$ an elementary embedding $\sigma_\beta:\bN\prec N$ satisfying conditions \ref{item:FirstConsequence**}-\ref{item:ThirdConsequence**} (using $\sigma_\alpha=\dot{\sigma}_\alpha^{H_\alpha}$ here), such that also $\sigma_\alpha(\bar{d})=d=\sigma_\beta(\bard)$ and $\sigma_\alpha(\dot{\vec{\bar{r}}})=\dot{\vec{r}}$.

Now let us let $G_{\alpha,\beta}$ be $\P_{\alpha,\beta}$-generic over $\V[G_\alpha]$, with $\tilde{w}\in G_\alpha*G_{\alpha,\beta}=G_\beta$. Note that $d\rest[\alpha,\beta)\in H_{\alpha,\beta}$, since $\tilde{w}\le d$. So by what $d\rest[\alpha,\beta)$ forced,
	the derived sequence $\delta_{G_\beta}(\dot{\vec{r}}^{G_\alpha},\dot{x})$ is below $q\rest[\beta,j)$, interprets $\dot{x}$ and respects $y$. Moreover, it is in $\ran(\sigma_\beta)[G_\beta]$.

In particular, $\tilde{w}$ forces the existence of an embedding $\sigma_\beta$ with all the desired properties \ref{item:FirstConsequence**}-\ref{item:LastConsequence**}. Since such a $\tilde{w}$ exists, no matter what $G_\alpha$ is (as long as $w\in G_\alpha$), there is a $\P_\alpha$-name $\dot{\tilde{w}}$ such that $w$ forces that $\dot{\tilde{w}}\rest\alpha\in G_\alpha$, and that $\dot{\tilde{w}}\rest[\alpha,j)$ forces the existence of such an embedding $\sigma_\beta$. Furthermore, $w$ forces that $\dot{\tilde{w}}\rest[\alpha,\beta)\le q\rest[\alpha,\beta)$. But, again by Corollary \ref{cor:Fullness}, there is then a $w^*\in\P_\beta$ such that $w^*\rest\alpha=w$ and such that $w$ forces that $w^*\rest[\alpha,\beta)\equiv\dot{\tilde{w}}\le q\rest[\alpha,\beta)$. This is the $w^*$ we were looking for.
\end{proof}

As before, let $\seq{t_n}{n < \omega}$ enumerate the elements of $\bN$, starting with $t_0=\leer$. Since we already know that the iteration is $\infty$-subproper, we don't have to worry about meeting any dense sets, so we don't need to fix an enumeration of the dense open subsets of $\bar{\P}_{\bar{j}}$ in $\bN$ this time. But we again let $\bar{W}$ be a nested antichain in $\bar{\vec{\P}}\rest{\bar{j}}$ such that $\bar{q}$ is a mixture of $\bar{W}$ up to $\bar{j}$, whose root has length $\bar{i}=0$, $<_{\bN}$-minimal with this property, and we let $W$ be the $<_N$-least nested antichain in $\vec{\P}\rest j$ whose root has length $i$ and such that $q$ is a mixture of $W$ up to $j$.

Much like in the $\infty$-subproperness iteration theorem, we will construct a fusion structure whose fusion can serve as $p^*$. Specifically, we construct a sequence
\[\seq{q^{(a,n)}, T^{(a, n)}
%, \dot{x}^{(a, n)}
, \dot{\sigma}^{(a, n)}}
%, \langle \dot{s}_k^{(a, n)} \;  | \; k < \omega \rangle\rangle\;
{n < \omega \; {\rm and} \; a \in T_n}\]
which satisfies all the properties %\ref{claim:Start2}-\ref{claim:LastOfFirstList2}
listed in the proof of Theorem \ref{thm:NiceIterationsOfSPforcingAreSP}, with two extra requirements (items \ref{item:ExtraStuffOn x^(a,n)} and \ref{item:ExtraStuffOnxdot}) regarding the interpretation of $\dot{x}$ (but dropping the requirement regarding the enumeration of dense open sets in $\bN$):
\begin{enumerate}[label=(\arabic*)]
\item
\label{claim:Start2OOB}
$T_0 = \{p\}$, $q^{(p,0)}=q$, $T^{(p,0)}=W$, $\dot{\sigma}^{(p,0)} = \check{\sigma}$.
%$\dot{s}^{(p,0)}_k=\dot{s}_k\in N^{\P_i}$ and $\dot{x}^{(p,0)}=\check{x}_0\in\bN^{\bP_{\bar{i}}}$.
\end{enumerate}
Further, for any $n<\omega$ and $a\in T_n$:
\begin{enumerate}[label=(\arabic*)]
\setcounter{enumi}{1}
\item %6
\label{item:middlepart-beginning2OOB}
$a\in\P_{l(a)}$, where $l(a)<j$. \\
$q^{(a, n)}\in \P_j$, $q^{(a, n)}\le_j q$ %, $a\le q^{(a,n)}\rest l(a)$
and
$\dot{\sigma}^{(a, n)}$ is a $\P_{l(a)}$-name.\\
%$\dot{x}^{(a,n)}$ is a $\P_{l(a)}$-name, and $q^{(a,n)}\rest l(a)$ forces that $\dot{x}^{(a,n)}$ is an element of Baire space and that $\dot{x}^{(a,n)}\le_0\check{y}$.\\
%For each $k<\omega$, $\dot{s}^{(a,n)}_k$ is a $\P_{l(a)}$-name, and $q^{(a,n)}\rest l(a)$ forces that $\dot{s}^{(a,n)}_k\in\check{\P}_j$.
\item %7
\label{item:sigma(a,n)2OOB}
$a$ forces the following statements with respect to $\P_{l(a)}$:
\begin{enumerate}[label=(\alph*)]
\item
\label{subitem:iselementary2OOB}
$\dot{\sigma}^{(a, n)} : \check{\bN}\prec \check{N}$.
\item
\label{subitem:moveseverythingcorrectly2OOB}
$\dot{\sigma}^{(a, n)}(\check{\bar{\theta}}, \check{\bar{\vec{\P}}}, %\check{\bar{s}},
\check{\bar{i}},\check{\bar{j}}, \check{\bar{q}}%,\check{\vec{\bar{s}}}
) = \check{\theta}, \check{\vec{\P}}, \check{s}, \check{i}, \check{j}, \check{q}%,\check{\vec{s}}
$.
\item
\label{subitem:hasx(a,n)initsrange2OOB}
$q^{(a,n)}\in\ran(\dot{\sigma}^{(a,n)})$. %and its preimage, $\bar{q}^{(a,n)}$, is in $\overline{D}_n$.
\end{enumerate}
%\item %8
%$a$ decides $\dot{\sigma}^{(a, n)}(t_n)$.
\item %9
\label{item:middlepart-end2OOB}
%For some $\bar{T}^{(a,n)}\in\bN$ and ordinal $\overline{l(a)}$, we have that $\overline{q}^{(a, n)} \in \bG_{\bar{i}}*\bar{G}_{\bar{i},\bar{j}}$, $\overline{q}^{(a, n)} \in \overline{D}_n$ and \\
There are $\bar{T}^{(a,n)}$, $\overline{l(a)}$ %, $\dot{\bar{x}}^{(a,n)}$, as well as $\dot{\bar{s}}^{(a,n)}_k$, for every $k<\omega$,
which are forced by $a$ to be the preimages of $T^{(a,n)}$, $l(a)$ %, $\dot{x}^{(a,n)}$, $\dot{s}^{(a,n)}_k$
under $\dot{\sigma}^{(a,n)}$, in the sense that $a\forces\dot{\sigma}^{(a,n)}((\bar{T}^{(a,n)})\check{})=(T^{(a,n)})\check{}$. %$\dot{\sigma}^{(a,n)}(\dot{\bar{s}}^{(a,n)})=\dot{s}^{(a,n)}$, etc.
Moreover, $a$ forces that
$(\dot{\sigma}^{(a,n)})^{-1}``\dot{G}_{l(a)}$ is $\bar{\P}_{\overline{l(a)}}$-generic over $\bN$.
\item
\label{item:ExtraStuffOn s_k^(a,n)}
\label{item:ExtraStuffOn x^(a,n)}
%For each $k<\omega$, $\dot{s}_k^{(a, n)}$ is a $\P_{l(a)}$-name for an element of $\P_j$ so that
%$q^{(a, n)}\rest l(a)$
$a$ forces wrt.~$\P_{l(a)}$ that in the range of $\dot{\sigma}^{(a,n)}$ there is a decreasing sequence of conditions in $\P_{l(a),j}$  below $q^{(a,n)}\rest[l(a),j)$,
such that the sequence interprets $\dot{x}$ and respects $y$.
%\forces_{l(a)} \dot{s}_0^{(a, n)}=(q^{(a, n)})\check{}$ and $\dot{s}_{k+1}^{(a, n)} \leq_j \dot{s}_{k}^{(a, n)}$ and $\dot{s}_k^{(a, n)}\rest l(a)\in\dot{G}_{l(a)}$.
%\item
%\label{item:ExtraStuffOn x^(a,n)}
%$q^{(a,n)}\rest l(a)\forces_{\P_{l(a)}}\interpretation(\dot{x},\vec{\dot{s}}^{(a,n)})=\dot{x}^{(a,n)}$.
\end{enumerate}
\noindent For $b \in {\rm suc}^n_T (a)$:
\begin{enumerate}[label=(\arabic*)]
\setcounter{enumi}{5}
\item %10
\label{item:coherence2OOB}
$b$ forces with respect to $\P_{l(b)}$ that \[(\dot{\sigma}^{(b,n+1)})^{-1}``\dot{G}_{l(a)}=(\dot{\sigma}^{(a,n)})^{-1}``\dot{G}_{l(a)}.\]
Further, if $m\le n$, then
\[b \forces_{l(b)} \dot{\sigma}^{(b, n+1)}(\check{t}_m) = \dot{\sigma}^{(a, n)}(\check{t}_m).\]
Also, let $<_T$ be the transitive closure of the order $<'_T$ on $\{\kla{c,k}\st k<\omega\land c\in T_k\}$ defined by $\kla{c,k}<'_T\kla{d,l}$ iff $l=k+1$ and $d\in\suc_T^k(c)$. Then, if $\kla{c,m}\le_T\kla{a,n}$, we have that \[b\forces\dot{\sigma}^{(b,n+1)}(\bar{q}^{(c,m)})=\dot{\sigma}^{(a,n)}(\bar{q}^{(c,m)}).\]
\item %11
\label{claim:LastOfFirstList2OOB}
There are $\bar{q}^{(b,n+1)}$, $\bar{T}^{(b,n+1)}$ such that\\
$b\rest l(a)\forces_{l(a)}
\dot{\sigma}^{(a,n)}((\bar{q}^{(b,n+1)})\check{},(\bar{T}^{(b,n+1)})\check{})=
(q^{(b,n+1)})\check{},(\bar{T}^{(b,n+1)})\check{}$.
\item
\label{item:ExtraStuffOnxdot}
$q^{(b,n+1)}$ decides $\dot{x}\rest n$, and
$q^{(b, n+1)}\forces_j\dot{x}\rest n \leq_0 \check{y} \rest n$.
\end{enumerate}

Further, as before, in order to ensure that we end up constructing a fusion sequence, we require that $T^{(a,n)}$ is a nested antichain in $\vec{\P}\rest j$, that $q^{(a,n)}$ is $(T^{(a,n)},j)$-nice, that $l(\Root(T^{(a,n)}))=l(a)$, that $a\le q^{(a,n)}\rest l(a)$
 and that for $b\in\suc^n_T(a)$, $T^{(b,n+1)}\hooks T^{(a,n)}$ and $b\rest l(a)\le a$.

Supposing we can construct such a sequence and letting $p^*$ be a fusion of the fusion sequence, then in any extension by a generic filter $G$ containing $p^*$, it follows that $\dot{x}^G \leq_0 y$, because in $\V[G]$, there is a sequence $\seq{a_n}{a<n}$ such that for all $n<\omega$, $q^{(a_n,n)}\in G$. So by \ref{item:ExtraStuffOnxdot}, $\dot{x}^G\rest n\le_0y\rest n$ for every $n$. We don't have to worry about piecing together the embeddings $\sigma^{(a_n,n)}$ as before, because we already know that $\P_j$ is $\infty$-subproper.

Thus it suffices to show that such a sequence can be constructed. This is done by recursion on $n<\omega$. The case $n=0$ is given by \ref{claim:Start2OOB}. For the inductive step, suppose for some $n<\omega$ we have constructed  $q^{(a,n)}, T^{(a, n)}$ and $\dot{\sigma}^{(a, n)}$ for some $a \in T_n$ satisfying \ref{claim:Start2OOB}-\ref{item:ExtraStuffOn x^(a,n)} and ${\rm suc}^m_T$ has been defined for all $m<n$. Assume also $q^{(a, n)}$ decides $\dot{x}\rest(n-1)$ (this is for free if $n=0$ and it follows from \ref{item:ExtraStuffOnxdot} at stage $n$ if $n>0$). We have to define $\suc_T^n(a)$. To this end, let $D$ be the set of all $b\in\bigcup_{l(a)\le\xi<j}\P_\xi$ such that there are a nested antichain $S$ in $\vec{\P}\rest j$ as well as objects $\dot{\sigma}^b$, $u$, $\bar{u}$, $\bar{S}$, $\overline{l(b)}$, %$\vec{\dot{r}}=\seq{\dot{r}_k^b}{k<\omega}$,
%$\vec{\dot{\bar{r}}}=\seq{\dot{\bar{r}}_k^b}{k<\omega}$,
%$\dot{x}^b$, $\dot{\bar{x}}^b$
satisfying the following conditions:

\begin{enumerate}[label=(D\arabic*)]
%\item
%$l(a)\le l(b)<j$.
%[already said above]
\item
\label{item:FirstConditionDefiningPredenseSet2OOB}
$b\rest l(a)\le a$.
\item
$S \hooks T^{(a, n)}$, $\bS\in\bN$, $S\in N$.
\item
$u\in \P_j$, $u \leq q^{(a,n)}$ and $u$ is a mixture of $S$ up to $j$.
\item
\label{item:EnsuringGettingAFusionStructure3}
$b\le_{l(b)}u\rest l(b)$ and $l(b)=l(\Root(S))$.
%\item
%$\overline{u} \in \overline{D}_{n+1}$.
\item
\label{item:InitialSegmentOfbForcesStuffIntoRangeOfTheOldEmbedding2OOB}
$b\rest l(a)\forces_{l(a)}\dot{\sigma}^{(a,n)}(\check{\bar{S}},\check{\bar{u}},(\overline{l(b)})\check{})=\check{S},\check{u},(l(b))\check{}$.
\item
\label{item:LastConditionDefiningPredenseSet2OOB}
$\dot{\sigma}^b$ is a $\P_{l(b)}$-name, and
$b$ forces with respect to $\P_{l(b)}$:
\begin{enumerate}
\item $\dot{\sigma}^b: \check{\bN}\prec \check{N}$.
\item $\dot{\sigma}^b (\check{\overline{\theta}},\check{\bar{i}},\check{\bar{j}},\check{\bar{\vec{\P}}}, \check{\overline{s}},\check{\overline{u}},\check{\bS}, \check{\bar{q}},(\overline{l(b)})\check{}
    %,(\dot{\bar{x}}^b)\check{},(\dot{\bar{r}}_k)\check{}
    ) = \check{\theta},\check{i},\check{j},\check{\vec{\P}},\dot{\sigma}(\bs), \check{u},\check{S}, \check{q},(l(b))\check{}%,(\dot{x}^b)\check{},(\dot{r}_k)\check{}
    $,
    and\\
    $\forall m \leq n\forall c\quad$ $\dot{\sigma}^b(t_m) = \dot{\sigma}^{(a, n)}(t_m)$ %and $\dot{\sigma}^b(\overline{D}_m) = \dot{\sigma}^{(a, n)}(\overline{D}_m)$,
    and \\ \hspace*{13ex}$\kla{c,m}\le_T\kla{a,n}\To\dot{\sigma}^b(\bar{q}^{(c,m)})=\dot{\sigma}^{(a,n)}(\bar{q}^{(c,m)})$.
\item $(\dot{\sigma}^b)^{-1}``\dot{G}_{l(a)}=(\dot{\sigma}^{(a,n)})^{-1}``\dot{G}_{l(a)}$ and $(\dot{\sigma}^b)^{-1}``\dot{G}_{l(b)}$ is $\bar{\P}_{\overline{l(b)}}$-generic over $\bN$.
\end{enumerate}
\item
\label{item:D-ExtraStuffOn s_k^(a,n)}
\label{item:D-ExtraStuffOn x^(a,n)}
$b$ forces with respect to $\P_{l(b)}$ that in the range of $\dot{\sigma}^b$ there is sequence $\vec{q'}$ of conditions in $\P_j$, such that $\vec{q'}/\dot{G}_{l(b)}$ is decreasing in $\P_{l(b),j}$, below $u/\dot{G}_{l(b)}$, interprets $\dot{x}$ and respects $y$.
\item
\label{item:D-ExtraStuffOnxdot}
$u$ decides $\dot{x}\rest n$, and
$u\forces_j\dot{x}\rest\check{n}\le_0\check{y}\rest\check{n}$.
\end{enumerate}

Let $\alpha=l(a)$. As in the proofs of the two iteration theorems, $D\rest\alpha$ is open in $\P_\alpha$, and it suffices to show that $D\rest\alpha$ is predense below $a$ in $\P_\alpha$.

To prove that $D\rest\alpha$ is predense below $a$, let $G_\alpha$ be a $\P_\alpha$-generic filter with $a\in G_\alpha$. Work in $\V[G_\alpha]$. Let $\sigma_n$ be the evaluation of $\dot{\sigma}^{(a, n)}$ by $G_\alpha$. We have that $\sigma_n:\bN\prec N$ moves the various parameters we care about correctly, $\bar{G}_{\balpha}=\sigma_n^{-1}``G_{l(a)}$ is $\bar{\P}_{\balpha}$-generic over $\bN$, and in $\ran(\sigma_n)$ there is a sequence $\vec{q}$ of conditions in $\P_j$, below $q^{(a,n)}$, such that $q_m\rest\alpha\in G_{\alpha}$ for each $m<\omega$, decreasing in $\P_{\alpha,j}$, interpreting $\dot{\bar{x}}$, and respecting $y$. Let $q_n$ be the $n$-th element of this sequence, which decides $\dot{x}\rest n$, and let $\bar{q}_n$ be its preimage under $\sigma_n$. Let $v:n\To\omega$ be such that $q_n\forces\dot{x}\rest n=\check{v}$.

By elementarity, $\overline{q}^{(a, n)}$ is a mixture of $\overline{T}^{(a, n)}$ up to $\overline{j}$.
Let $\bT_0^{(a,n)}=\{\ba_0\}$. Let $a_0=\sigma_n(\ba_0)$. We know that $a\le q^{(a,n)}\rest\alpha$, so $q^{(a,n)}\rest\alpha\in G_\alpha$, so $\bar{q}^{(a,n)}\rest\balpha\in\bar{G}_{\balpha}$. By Fact \ref{fact:CharacterizationOfMixtures}.\ref{item:Root}, $\bar{q}^{(a,n)}\rest l(\ba_0)\equiv\ba_0$.
%Since $\bar{q}^{(a,n)}$ is a mixture of $\bar{T}^{(a,n)}$ up to $\bar{j}$ in $\bN$, it follows that  $\bar{q}^{(a,n)}\rest l(\ba_0)\equiv\ba_0$, by Fact \ref{fact:CharacterizationOfMixtures}.\ref{item:Root}. % since $l(\bar{q}^{(a,n)})=\bar{j}>l(\ba_0)$.
Since we have ensured that $\alpha=l(\Root(T^{(a,n)}))$, we have by elementarity that $\balpha=l(\bar{a}_0)$.
So $\ba_0\in\bar{G}_{\balpha}$.
Let $\bar{r}\in\bT_1^{(a,n)}$ be such that $\bar{r}\rest\balpha\in\bG_{\balpha}$.
Let $\bar{\beta}=l(\bar{r})$ and $\beta=\sigma_n(\bar{\beta})$.
Let $\bar{u}\in\bP_{\bar{j}}$ strengthen $\tilde{q}=\bar{r}^{\frown}\bar{q}_n\rest [{l(\bar{r})},\bar{j})$ so that $\bar{u}\rest\balpha\in\bG_{\balpha}$.

Let $\tilde{u}\in\bar{G}_{\balpha}$ with $\tilde{u}\le\bar{r}\rest\balpha,\bar{q}_n\rest\balpha$ (as both of these conditions are in $\bar{G}_\balpha$), and then set $\bar{u}=\tilde{u}\verl\bar{q}_n\rest[\balpha,\bar{j}]$. So we have:
\[\bar{u}\rest\balpha\in\bar{G}_\balpha,\ \bar{u}\le_{\bar{j}}\bar{q}_n, \ \bar{u}\rest\balpha\le\bar{r}\rest\balpha\ \text{and}\ \bar{u}\rest[\balpha,\bar{j})\equiv\bar{q}_n\rest[\balpha,\bar{j}).\]
By Lemma \ref{2.11}, applied in $\bN$, we can pick a nested antichain $\bar{S} \hooks \bar{T}^{(a, n)}$ such that $\bar{u}$ is a mixture of $\bS$ up to $\bar{j}$ and such that letting $\bS_0 = \{\bd_0\}$, we have that  $l(\bar{r}),\bar{l(a)}\leq l(\bd_0)$ and $\bd_0 \rest l(\bar{r}) \leq \bar{r}$.
Moreover, we have that $\bar{u}\rest l(\bd_0)\equiv\bd_0$, since $\bar{u}$ is a mixture of $\bS$ up to $\bar{j}$. In particular, $\bd_0\rest\overline{l(a)}\in\bG_{\overline{l(a)}}$.

Let $S,d_0,u,q_n=\sigma_n(\bS,\bd_0,\bu,\bar{q}_n)$, and let $l(d_0) = \beta=\sigma_n(l(\bd_0))=\sigma_n(\bar{\beta})$.
We have that $u\rest l(d_0)\equiv d_0$. Note that $d_0\rest l(a)\in G_{l(a)}$. Also, let $\bar{s}'$ be a tuple of parameters in $\bN$ we want our new embedding to move the same way $\sigma_n$ does, and let $s'=\sigma(\bar{s}')$.

Let $w\in G_{l(a)}$ force these facts about $\dot{\sigma}^{(a,n)}$ and the check names for the parameters $\bar{S},\bar{d_0},\bu,\bs'$ and $S,d_0,u,s',v$. Since $a,d_0\rest l(a)\in G_{l(a)}$, we may choose $w$ so that $w\le_{l(a)}a,d_0\rest l(a)$.

Now we apply $(*)$ to %$d=d_0$, $\dot{\bar{q}}=\check{\bar{u}}$, $\dot{q}=\check{u}$
$w$, $q=u$, $\bar{q}=\bar{u}$ and $\dot{\sigma}_\alpha=\dot{\sigma}^{(a,n)}$. Note here that the sequence $\seq{q_m\rest[\alpha,j)}{n\le m<\omega}$ is forced to be decreasing and below $u$ in $\P_{\alpha,j}$ by $w$.
We obtain a condition $b=w^*\in\P_\beta$ with $b\rest\alpha=w$ and $b\le u\rest\beta\equiv d_0$, and a $\P_\beta$-name $\dot{\sigma}'$ such that $b$ forces with respect to $\P_\beta$:
\begin{enumerate}[label=(\alph*)]
  \item
  $\dot{\sigma}'$ and $\dot{\sigma}^{(a,n)}$ move the parameters
  $\check{\bar{\vec{\P}}}$, $\check{l(\overline{a})}$, $\check{\overline{\beta}}$, $\check{\btheta}$, $\check{\bu}$, $\check{\bd}_0$, $\bar{s}'$, $\bar{q}_n$ and $\check{\bS}$ the same way.
  \item
  $(\dot{\sigma}'{}^{-1})``\dot{G}_\beta$ is generic over $\check{\bN}$.
  \item
  $(\dot{\sigma}'{}^{-1})``\dot{G}_{l(\alpha)}=(\dot{\sigma}^{(a,n)})^{-1}``G_{l(\alpha)}$.
  \item there is in $\ran(\dot{\sigma}')$ a sequence of conditions $\vec{q'}$ in $\P_j$ such that $\vec{q'}\rest\dot{G}_{l(b)}$ is decreasing and below $u/\dot{G}_{l(b)}$ in $\P_{l(b),j}$, that interprets $\dot{x}$ and respects $y$.
\end{enumerate}

Note that $w$ forced that $\dot{\sigma}^{(a,n)}(\bu,\bd_0,\bS)=u,d_0,S$ and hence, since $b\rest l(a)=w$, $b$ forces that $\dot{\sigma}'(\bu,\bd_0,\bS)=u,d_0,S$ as well. In addition, we may insist that $\sigma'$ moves all the required parameters the same way $\dot{\sigma}^{(a,n)}$ does (by listing them in $\bar{s}'$).

To finish, we claim that $b\in D$ as witnessed by $u$ and $\dot{\sigma}'$. The only new point that requires checking (compared to the proof of Theorem \ref{thm:NiceIterationsOfSPforcingAreSP}) is \ref{item:D-ExtraStuffOnxdot}: that $u$ decides $\dot{x}$ and forces that $\dot{\bar{x}}\rest n \leq_0 \check{y} \rest n$. But this is clear since $\bar{u}\le\bar{q}_n\forces_{\bar{j}}\dot{x}\rest n=\check{v}\le_0\check{y}\rest n$.
\end{proof}

\begin{remark}
In the second author's PhD thesis (written under the direction of the first author) a general version of Theorem \ref{thm:NiceIterationsOfOOBoundingInftySPforcing} is proved that implies that many preservation results on the reals for proper forcing hold for $\infty$-subproper forcing as well including the Sacks and Laver properties.
\end{remark}

\section{Applications}
\label{sec:Applications}

In this section we provide some applications of the preservation theorems proved in the previous two sections, as announced in the introduction. In what follows, one can choose to iterate using either RCS iterations or nice iterations. If using the latter, ``subcomplete'' and ``subproper'' can be replaced by their ``infinity'' versions.

%The main type of application we would like to explore is to models of subversion forcing principles in which $\CH$ fails in some controlled way. The point is this: by Jensen's work on $\SCFA$, we know that many consequences of $\MM$ are not consequences of $\SCFA$ because the latter is consistent with $\diamondsuit$. However, this answer is uninformative with respect to what one can expect from the forcing axiom and its class. What we show here, in a variety of ways, is that $\SCFA$ is consistent with many implications of $\diamondsuit$ and $\CH$, even if the continuum is $\aleph_2$. The result is that $\SCFA$ essentially determines nothing at the level of $\omega_1$ and the reals.

The main technique used here is as follows. %the notion of a {\em weaving construction}.
We start in a model with a supercompact cardinal.
Suppose we have a subproper forcing $\mathbb P$ which preserves some property that is also preserved by all subcomplete forcing (for instance not killing a fixed Souslin tree $S$). Then we can add $\mathbb P$ into the standard Baumgartner type iteration to produce a model of $\SCFA$ while preserving that property (so $S$ remains Souslin). Moreover, $\SCFA$ will be forced in the final model, but $\mathbb P$ makes some contribution as well.  %We begin by looking at this idea with respect to preserving Souslin trees.

%\begin{thm}
%Over a model with a reflecting cardinal, we can produce a forcing extension in which \BSCFA holds, together with:
%\begin{enumerate}
  %\item $\mathfrak{b}=\aleph_1<2^{\aleph_0}$ (by iterating subcomplete forcing, interspersed with adding random reals. This is an iteration of subproper $\oo$-bounding forcing),
 % \item $\neg\CH+\neg\SH$ (by iterating subcomplete forcing and adding Cohen reals. This is an iteration of subproper, Souslin-tree preserving forcing - note that Cohen forcing preserves Souslin trees because if $G$ is Cohen-generic and $T$ is a Souslin tree in the ground model, if $\dot{A}$ is a name for an unbounded antichain in $T$ in $\V[G]$, then in $\V[G]$, let $\dot{f}^G:\check{\omega}_1\To\dot{A}^G$ enumerate $\dot{A}^G$, and let, for $\alpha<\omega_1$, $p_\alpha\in G$ force that $\dot{f}(\check{\alpha})=\dot{f}^G(\check{\alpha})$. Then uncountably often, $p_\alpha=p$, for some $p\in G$. But then already $p$ decides an uncountable part of $\dot{A}$, which already exists in $\V$.),
 % \item $\CH+\neg\diamondsuit$ (by iterating dee-complete $\omega_1$-subproper forcing)
%\end{enumerate}
%\end{thm}
%%%%%%%%%Why use BSCFA and not SCFA?

Recall that a forcing notion $\mathbb P$ is $\sigma$-linked if it can be written as the countable union, $\mathbb P = \bigcup_{n < \omega} \mathbb P_n$ where for each $n < \omega$ $\mathbb P_n$ consists of pairwise compatible elements. Note that $\sigma$-linked forcing notions are ccc. The following is well known.

\begin{proposition}
If $S$ is a Souslin tree and $\mathbb P$ is $\sigma$-linked then forcing with $\mathbb P$ does not kill $S$.
\end{proposition}

\begin{proof}
Let $\mathbb P$ and $S$ be as in the statement and since $\mathbb P$ is $\sigma$-linked it can be written as $\bigcup_{n < \omega} \mathbb P_n$. Now suppose $\dot{A}$ names a maximal antichain in $S$ and suppose $p \in \mathbb P$ forces that $\dot{A}$ is uncountable. For each $n < \omega$ let $A_n = \{s \in S \; | \; \exists q \leq p \, q \in \mathbb P_n \, q\forces \check{s} \in \dot{A}\}$. Since each $\mathbb P_n$ consists of pairwise compatible elements, it follows that each $A_n$ is an antichain (in $\V$). Therefore it's countable. But that means that $\bigcup_{n < \omega} A_n$ is countable i.e. the set of all $s\in S$ so that there is some condition stronger than $p$ forcing $s$ to be in $\dot{A}$ is countable, which contradicts the fact that $p$ forced $\dot{A}$ to be uncountable.
\end{proof}

In the following, we will treat both \SCFA, the subcomplete forcing axiom, and its bounded version $\BSCFA$. \SCFA states that if $\P$ is a subcomplete forcing and $\seq{D_i}{i<\omega_1}$ is a sequence of dense subsets of $\P$, then there is a filter $F\sub\P$ such that for every $i<\omega_1$, $F\cap D_i\neq\emptyset$. The bounded version of the axiom, denoted \BSCFA, is the weaker form of the axiom stating that whenever $\P$ is a subcomplete forcing and $\seq{A_i}{i<\omega_1}$ is a sequence of maximal antichains in $\P$, each of which has size at most $\aleph_1$, then there is a filter $F\sub\P$ such that for every $i<\omega_1$, $F\cap A_i\neq\emptyset$. This axiom was originally introduced for proper forcing by Goldstern and Shelah \cite{GoldsternShelah:BPFA}, where the consistency strength of the bounded proper forcing axiom was shown to be exactly a reflecting cardinal. In Fuchs \cite{Fuchs:HierarchiesOfForcingAxioms}, the version for subcomplete forcing was analyzed and shown to have the same consistency strength.
In all the applications below $\SCFA$ and $\BSCFA$ could be replaced by their ``$\infty$'' versions however, since we do not know if these statements are equivalent or not we leave this out.

\begin{thm}
Assume that $\kappa$ is supercompact. Then there is a $\kappa$-length iteration $\mathbb P_\kappa$ of subproper forcing notions so that if $G\subseteq \mathbb P_\kappa$ is generic over $\V$ then in $\V[G]$ there are Souslin trees, $\mathfrak{c} = \aleph_2$ and $\SCFA$ holds. If $\kappa$ is only a reflecting cardinal then the same conclusion holds true with $\SCFA$ replaced by $\BSCFA$, the bounded subcomplete forcing axiom.
\end{thm}

As mentioned in the beginning of this section, the word ``iteration'' in the theorem can be interpreted either way discussed in this paper.

\begin{proof}
By forcing if necessary, assume first that in $\V$ that there is a Souslin tree $S$. Let $\kappa$ be supercompact. We will define a $\kappa$-length iteration, $\mathbb P_\kappa$ as follows: let $f: \kappa \to V_\kappa$ be a Laver function. At stage $\alpha$ if $f(\alpha) = (\dot{\mathbb P}, \mathcal D)$ is a pair of $\mathbb P_\alpha$ names such that $\dot{\mathbb P}$ is a subcomplete forcing notion and $\mathcal D$ is a $\gamma$-sequence of dense subsets of $\dot{\mathbb P}$ for some $\gamma < \kappa$ then let $\dot{\mathbb Q}_\alpha = \dot{\mathbb P}$ otherwise add a Cohen real. By the iteration theorems proved in the earlier sections, at limit stages either we can decide to take RCS limits, in which case we need to collapse each iterand to $\aleph_1$ as well, or else nice limits in the sense of the previous section. Either way, since every iterand is either subcomplete or proper, the entire iteration is subproper. Moreover, since subcomplete forcing doesn't kill Souslin trees, and neither does Cohen forcing, since it's $\sigma$-linked, the entire iteration doesn't kill $S$.

A standard $\Delta$-system argument shows that $\mathbb P_\kappa$ has the $\kappa$-c.c. but since $\mathbb P_\kappa$ collapses everything inbetween $\omega_1$ and $\kappa$, in the extension $\kappa = \omega_2$. Also Cohen reals are added unboundedly often there are $\kappa$ many new reals in the extension so $\kappa = 2^{\aleph_0}$. Finally the usual Baumgartner argument shows that $\SCFA$ must be forced as well. For a detailed proof of this in the subcomplete context see Jensen \cite[pp.65-66, Proof of Theorem 5]{Jensen2014:SubcompleteAndLForcingSingapore}. There Jensen checks the Baumgartner proof when no reals are added but it's easily seen to go through in this case as well.

For the case of $\BSCFA$ the proof is nearly identical, replacing the argument for $\SCFA$ from the Laver diamond by the one for $\BSCFA$ with a reflecting cardinal, see Fuchs \cite[Lemma 3.5]{Fuchs:HierarchiesOfForcingAxioms}.
\end{proof}

Following language used by Jensen in \cite{Jensen:FAandCH}, let us refer to any model obtained by performing the Baumgartner style iteration below a supercompact cardinal, by iterating forcings belonging to some forcing class $\Gamma$ as the \emph{natural model for the forcing axiom for $\Gamma$}. The previous theorem then shows that the natural model for the forcing axiom for the class of all forcings which are either subcomplete or Cohen forcing, satisfies $\SCFA + \mathfrak{c}=\omega_2 + \text{``there is a Souslin tree''}$, because Cohen forcing adds a Souslin tree, and that Souslin tree survives.

Observe that all that was used about Cohen forcing in the proof above is that it is $\sigma$-linked and adds a real. It follows that we could have ensured that we force with every $\sigma$-linked forcing the Laver diamond guessed as well. As a result essentially the same proof gives the following.

\begin{thm}
Assume that $\kappa$ is supercompact. Then there is a $\kappa$-length iteration $\mathbb P_\kappa$ of subproper forcing notions so that if $G\subseteq \mathbb P_\kappa$ is generic over $\V$ then in $\V[G]$ there are Souslin trees, $\mathfrak{c} = \aleph_2$, and both $\SCFA$ and $\MA_{\aleph_1} (\sigma {\rm -linked})$ hold. If $\kappa$ is only a reflecting cardinal then the same conclusion holds true with $\SCFA$ replaced by $\BSCFA$. As a result, $\SCFA + \MA_{\aleph_1} (\sigma{\rm - linked})$ does not imply $\MA$ (as $\MA$ implies Souslin's Hypothesis).
\end{thm}

Again, the model of the previous theorem can be taken to be the natural model for the forcing axiom for the class of all forcing notions that are subcomplete or $\sigma$-linked.

The $\oo$-bounding preservation theorem gives us another result along these lines. Recall that $\mathfrak{d}$, the dominating number, is the smallest cardinal $\kappa$ such that there is a collection $D$ of reals such that every real is dominated by some real in $D$.

\begin{thm}
Assume that $\kappa$ is supercompact. Then there is a $\kappa$-length iteration $\mathbb P_\kappa$ of subproper forcing notions so that if $G\subseteq \mathbb P_\kappa$ is geneic over $\V$ then in $\V[G]$ we have that $\mathfrak{d} = \aleph_1 < \mathfrak{c} = \aleph_2$ and $\SCFA$ holds. Moreover, it can be arranged that there are either Souslin trees or not. If $\kappa$ is only a reflecting cardinal then the same conclusion holds true with $\SCFA$ replaced by $\BSCFA$.
\end{thm}

\begin{proof}
The idea is the same as in the previous theorem, replacing Cohen forcing with some fixed $\omega^\omega$-bounding forcing which adds a real, for example random forcing. By the preservation theorem, the entire iteration will be $\omega^\omega$-bounding and so $\mathfrak{d} = \aleph_1$ as witnessed by the collection of the ground model reals, which will be dominating and have size $\aleph_1$ in the final model. For the ``moreover'' part, note that since random forcing is $\sigma$-linked, the resulting iteration will preserve any given Souslin tree, thus giving the consistency of the above with a Souslin tree. However, one could also choose to force with every Souslin tree as they are guessed by the Laver diamond. Since forcing with a Souslin tree is ccc and does not add reals, it's in particular proper and $\omega^\omega$-bounding. In the latter case there will be no Souslin trees in the final model.
\end{proof}

Finally let us note, one more application in this spirit, this one due to Jensen \cite[\S4]{Jensen:FAandCH}.

\begin{thm}[Jensen]
Assume that $\kappa$ is supercompact. Then there is a $\kappa$-length iteration $\mathbb P_\kappa$ of subproper forcing notions so that if $G\subseteq \mathbb P_\kappa$ is generic over $\V$ then in $\V[G]$ $\SCFA$ holds, $\CH$ holds but all Aronszajn trees are special. In particular $\diamondsuit$ fails. If $\kappa$ is only a reflecting cardinal then the same conclusion holds true with $\SCFA$ replaced by $\BSCFA$.
\end{thm}

\begin{proof}
In \cite{Jensen:DSF} Jensen introduces the class of ``Dee-subcomplete and $<\omega_1$-subproper forcing notion'' and proves that the associated forcing axiom $\mathsf{DSCFA}$ is consistent relative to a supercompact. This class does not add reals hence, like with $\SCFA$, the natural model of this axiom satisfies $\CH$ as well. While we omit the definition of this class here, we note that Jensen shows that it contains all subcomplete forcing, plus a forcing notion for specializing Aronszajn trees. It follows that under $\mathsf{DSCFA}$ both $\SCFA$ and ``all Aronszajn trees are special" hold (so $\diamondsuit$ fails). Since this axiom is consistent with $\CH$, we're done.
\end{proof}

\section{Conclusion and Open Questions}
\label{sec:Questions}

The possible structure of the continuum under $\SCFA$ remains something of a mystery. In particular, the above methods show that when $\CH$ fails, \SCFA does not say much about the reals or combinatorics on $\omega_1$, and the \CH setting may in some sense be more attractive for \SCFA. One question we had was whether \SCFA is consistent with $\diamondsuit^+$ (and hence with the existence of a Kurepa tree). It is easy to see that in the natural model of \SCFA, there are no Kurepa trees. In fact, in that model $\SCFA^+$ holds, and already $\MA^+_{\omega_1}(\sigma{\rm -}closed)$ implies that there are no Kurepa trees. However, in recent work, Hirohsi Sakai obtained the very interesting result that \SCFA is consistent with $\diamondsuit^+$, thus answering this question.

We can ask similar questions about $\omega_2$, and at that level, much more remains open. For instance, we do not even know whether the continuum can be larger than $\aleph_2$ under $\SCFA$. \footnote{The proof of \cite[\S 3, Cor.~9.1]{Jensen:FAandCH}, which, if true, would imply that under \SCFA, $2^\omega\le\omega_2$, relies on \cite[Lemma 6.3]{Jensen2014:SubcompleteAndLForcingSingapore}. However, the latter lemma is missing the assumption that $\kappa>2^\omega$. The issue with its proof was discovered by Sean Cox, and Hiroshi Sakai observed that the additional assumption was necessary.}

\begin{question}
Is $2^{\aleph_0} > \aleph_2$ consistent with $\SCFA$?
\end{question}

Another line of questioning concerns the utility and uniqueness of $\infty$-subproper and $\infty$-subcomplete forcing notions.

\begin{question}
Is every $\infty$-subcomplete ($\infty$-subproper) forcing notion subcomplete (subproper)? What about just up to forcing equivalence? Are the forcing axioms for these classes equivalent?
\end{question}

%More questions along these lines are as follows.
%\begin{question}
%Is $\SCFA + \neg \CH$ consistent with $\square^*_{\omega_1}$? Is it consistent with the existence of an $\aleph_2$-Aronszajn tree? What if $\SCFA$ is replaced by $\mathsf{DSCFA}$?
%\end{question}
%Note that both $\square^*_{\omega_1}$ and the existence of an $\aleph_2$-Aronszajn tree are implied by $\CH$ so they hold in the natural model of $\SCFA$. %%%not needed because Hiroshi Sakai answered this question in the positive

On another note, we can ask about the relationship between RCS and nice iterations. Note that proper, semiproper, subproper and subcomplete forcing notions are all iterable by both types of limits. It's worth asking if this is simply a $\ZFC$-phenomenon.
\begin{question}
Suppose $\Gamma$ is a definable class of forcing notions. If $\Gamma$ is iterable with RCS, is it iterable by nice iterations? What about the converse?
\end{question}

%\bibliography{literatur}
%\bibliographystyle{plain}

\end{document}